\documentclass[11pt,reqno]{amsart}
\usepackage{xifthen,setspace,bm}

\usepackage[utf8]{inputenc}
\usepackage[T1]{fontenc} 
\usepackage{bm}
\usepackage{amsmath,amsthm,amsfonts,amssymb,mathrsfs,stmaryrd,bigints,graphicx}
\usepackage{dsfont}

\usepackage{mathtools}
\usepackage[usenames]{color}
\usepackage{thmtools, thm-restate}

\usepackage[colorlinks=true,linkcolor=blue]{hyperref}
\DeclareFontFamily{U}{matha}{\hyphenchar\font45}
\DeclareFontShape{U}{matha}{m}{n}{
      <5> <6> <7> <8> <9> <10> gen * matha
      <10.95> matha10 <12> <14.4> <17.28> <20.74> <24.88> matha12
      }{}
\DeclareSymbolFont{matha}{U}{matha}{m}{n}
\DeclareMathSymbol{\varleftrightarrow}{3}{matha}{"D8}
\DeclareMathSymbol{\nvarleftrightarrow}{3}{matha}{"DC}

\usepackage{epstopdf} 
\usepackage{amssymb}
\usepackage{setspace}
\usepackage{enumerate}
\usepackage{bigstrut}
\usepackage{esvect}
\usepackage{multirow}
\usepackage{mathtools,xparse}
\usepackage{mathrsfs}
\usepackage{hyperref}
\usepackage{xcolor}
\newtheorem*{lemma*}{Lemma}
\newtheorem{theorem}{Theorem}[section]

\newtheorem{lemma}[theorem]{Lemma}
\newtheorem{proposition}[theorem]{Proposition}

\newtheorem{remark}[theorem]{Remark}
\newtheorem{assumption}{Assumption}

\theoremstyle{definition}
\newtheorem{definition}[theorem]{Definition}

\allowdisplaybreaks
\usepackage{chngcntr}

\renewcommand{\P}{\mathbb{P}}

\newcommand{\E}{{\mathbb{E}}}

\newcommand{\Z}{\mathbb{Z}}

\newcommand{\e}{\varepsilon}

\newcommand{\extr}{\textup{ext}}

\newcommand{\cable}{\widetilde{\Z}^d}

\newcommand{\B}{\mathbb{B}}

\newcommand{\N}{\mathbb{N}}

	\renewcommand{\P}{\mathbb{P}}

\newcommand{\ble}{\textup{\textsf{BLE}}}

\newcommand{\cC}{\mathcal{C}}

\newcommand{\cL}{\mathcal{L}}

\newcommand{\cN}{\mathcal{N}}

\newcommand{\wt}{\widetilde}

\newcommand{\ar}{\leftrightarrow}

\newcommand{\lbij}{\longleftrightarrow}

\renewcommand{\emptyset}{\varnothing}

\renewcommand{\setminus}{\backslash}

\newcommand{\IIC}{\textup{IIC}}

\def\ba{\begin{align}}
\def\ea{\end{align}}
\def\bs{\begin{split}}
\def\es{\end{split}}

\begin{document}

\title[Critical level set percolation for the GFF in $d>6$]{Critical level set percolation for the GFF in $d>6$: comparison principles and some consequences}

\author{Shirshendu Ganguly and Kaihao Jing}

\begin{abstract} The intrinsic geometry of  
the critical percolation cluster induced by the level set of the metric Gaussian free field on $\Z^d$ has been the subject of much recent activity. In \cite{lupu2016loop} Lupu established that the critical percolation cluster has the same law as that in a Poisson loop soup where the loop intensity is dictated by the Green's function of the usual random walk. Throughout the mean field regime $d>6$, a sharp Euclidean one arm exponent was proven recently in \cite{cai2023one}. Subsequently in \cite{ganguly2024ant}, other results about the chemical one arm exponent, volume growth and the  Alexander-Orbach conjecture were established for all $d>20.$ 

In this article, we introduce new methods to obtain several sharp estimates about the intrinsic geometry which hold for all $d>6$. We develop two primary comparison methods. The first involves a comparison of the extrinsic (Euclidean) and intrinsic metrics allowing estimates about the former to be transferrable to the latter. The second compares intrinsic geodesics to a class of modified paths which are non-backtracking in a certain loop sense. The latter is amenable to analysis using more classical methods of bond percolation. We finally prove that such non-backtracking paths are not much longer than geodesics allowing us to establish comparable estimates for the geodesic. As applications of such methods, we establish, for all $d>6,$ the chemical one-arm exponent, as well as an averaged version of a conjecture of Werner \cite{werner2021clusters} which asserts that deletion of loops with diameter larger than $r^{6/d}$ does not affect the connection probabilities between points at Euclidean distance $r.$  An important ingredient of independent interest is a local connectivity estimate which asserts that the connection probability between two points at distance $r$ remains, up to constants, unchanged even if the connecting path is confined to a ball of diameter comparable to $r.$
\end{abstract}

\address{Department of Statistics, UC Berkeley, USA}
\email{sganguly@berkeley.edu}

\address{Department of Statistics, UC Berkeley, USA}
\email{khjing@berkeley.edu}

\maketitle


\tableofcontents



\section{Introduction}

Percolation was introduced by Broadbent and Hammersley in 1957 to study how the random properties of a medium influence the passage of a fluid through it. The simplest version is bond percolation on the Euclidean lattice $\Z^d$ where every edge appears independently with probability $p \in (0,1)$.

As the reader is perhaps already aware, bond percolation has been extensively studied. The first fundamental result that it exhibits a non-trivial phase transition as $p$ varies, was proved by \cite{broadbent1957percolation} Broadbent and Hammersley. More precisely, for any $d\geq 2$, there exists $p_{c}(d) \in (0,1)$ such that when $p > p_{c}$, the model percolates (i.e. contains an infinite cluster) almost surely and does not percolate when $p<p_{c}$. Subsequently, the critical setting (i.e. the case $p=p_{c}$) has been the subject of central interest of the bond percolation. Kesten \cite{kesten1980critical} proved that for $\mathbb{Z}^{2}$, $p_{c}=\frac{1}{2}$ and this combined with Harris' work \cite{harris1960lower} {implies that at criticality the system doesn't percolate almost surely}. This phenomenon has been termed as the continuity of phase transitions. The continuity was then proved for $d\geq 19$ and $d>6$ on a ``spread-out'' version of $\mathbb{Z}^{d}$ by Hara and Slade \cite{hara1990mean} using the powerful technique of lace-expansion. The bound on the dimension was subsequently lowered to  $d\geq 11$ by Fitzner and van der Hofstad \cite{fitzner2017mean}. We refer readers to \cite{duminil2018sixty} for a comprehensive introduction. 
While proving the absence of percolation at criticality remains a major open problem for the remaining dimensions we now move on to behavior of critical percolation cluster. It is believed that in high enough dimensions, percolation clusters at criticality should exhibit properties similar to that on a tree, a critical branching process. A particularly interesting manifestation of this is the Alexander-Orbach conjecture \cite{alexander1982density} which predicts that the random walk on such clusters (termed as the ant in the labyrinth) exhibits certain critical exponents. These exponents in the case of the critical tree were rigorously established in the seminal work of Kesten \cite{kesten1986subdiffusive}. Subsequently, there's been a series of impressive developments aimed at establishing the same for the lattice models. 
While the AO conjecture is known to be false in low dimensions \cite{chayes1987upper}, it was shown to hold in critical oriented percolation for dimensions greater than $19$ in \cite{barlow2008random}. Finally, in a breakthrough in \cite{kozma2009alexander}, the case of critical bond percolation was addressed. 

More recently, motivated by field theory and random geometry considerations, there has been a significant interest in analyzing percolation models  whose correlation functions decay as a power law.
Examples include random interlacements, 
loop soup percolation \cite{chang2016phase,le2013markovian,lupu2016loop}, and level-set percolation of random fields.
While random fields considered in the literature include  Gaussian ensembles such as  
randomized spherical harmonics (Laplace eigenfunctions) \cite{lap1,lap2,lap3} at high frequencies,
 the object of focus in this article will be the massless Gaussian free field; we point the interested reader to the very recent survey recording some of the noteworthy developments \cite{universality}.

The level set percolation of GFF, where one considers all points with GFF value bigger than a threshold $h$, will be the primary model of consideration in this article. This was originally investigated by Lebowitz and Saleur in \cite{ls}    
 as a percolation model with algebraic decay of correlations.  The model has witnessed some particularly intense activity recently. {An important progress was made by  Rodriguez-Sznitman in \cite{gff2}, where a non-trivial phase transition for the level set percolation of GFF was established in all dimensions $d\ge 3$.
  In a recent impressive development, Duminil-Copin, Goswami,
Rodriguez and Severo 
 \cite{gff5}  established the
  equality of parameters dictating the onset of various forms of sub-critical to super-critical transitions. 
  Subsequently, across the works of Drewitz,   
R\'ath and Sapozhnikov  \cite{drs},   
Popov and R\'ath  \cite{pr}, Popov and Teixeira  \cite{pt}, Goswami,
Rodriguez and Severo  \cite{grs}, a rather detailed picture of both the super-critical and sub-critical regimes have emerged.  
The critical regime also has witnessed a recent explosion of activity.

 Lupu \cite{lupu2016loop} proved that the critical value $h_{*}$ equals to 0. Moreover, it was also shown that at the critical value $h=0$, the level set does not percolate. 
The key observation in \cite{lupu2016loop} was that the signed clusters of the
  Gaussian free field on $\widetilde \Z^d$ (the metric graph of $\Z^d$ obtained 
  by including the edges as one dimensional segments, often termed as the {cable graph} in the literature)
is given by the clusters in a loop soup model (a Poisson process on the space of loops) introduced earlier in \cite{le1}. 
This connection allows one to simply consider the loop soup model which is what we will do for the majority of this article. 

Based on Lupu's coupling, in \cite{werner2021clusters}, Werner had initiated the study of high dimensional level set
 percolation for the GFF.  Drawing parallel 
 with the mean field bond percolation picture established in high enough dimensions, he put forth 
 a variety of conjectures as well as arguments. In particular,  Werner asserted that in high dimensions (i.e. $d > 6$), the GFF on
$\Z^d$ becomes 
asymptotically independent and thus shares similar behavior with
bond percolation.

Along these lines, there has been a flurry of activity understanding refined properties of the critical percolation cluster. We list below some notable developments which are of particular relevance to this paper.

An observable of particular interest is the arm probability 
\begin{equation}\label{onearm}\pi_1(n):= \P\big(0 \lbij \partial\B_n \big),
\end{equation}
i.e., the probability that the percolation cluster of the origin (henceforth denoted by $\widetilde\cC(0):=\{x\in \widetilde \Z^{d}: 0\overset{\tilde{E}^{\ge 0}}{\lbij}x  \}$) extends to the boundary of {$\B_n$, the Euclidean box $[-n,n]^d,$}  which was already shown in \cite{lupu2016loop} to converge to $0$ as $n\to \infty,$ i.e. there is no percolation at criticality. 
Ding and Wirth \cite{ding2020percolation} employed
a martingale argument and proved various polynomial bounds for the one-arm probability in various dimensions.  Subsequently, in  work by  Drewitz-Pr\'evost-Rodriguez \cite{dpr}, a detailed analysis of the near-critical behavior for the level set percolation of GFF on  general transient graphs was carried out.
The sharp one-arm probability was finally obtained for all $d>6$ in \cite{cai2023one},  building on the earlier work of  Kozma and Nachmias \cite{kozma2011arm} in the bond percolation case.
In particular it was shown that $\pi_1(n) \sim n^{-2}$ for dimensions $d>6$.
In \cite{drewitz2023arm}, the right leading order was established for a general class of graphs which include the Euclidean lattices of dimensions $3$ and $4$. More recently, in \cite{cai2024one}, the up-to-constant estimates for $\pi_{1}(n)$ were obtained in all dimensions except $d=6$, while in $d=6$, the correct leading order was determined.}  



Very recently in \cite{ganguly2024ant}, the first named author and Kyeognsik Nam established the mean field exponents for the intrinsic geometry of the percolation cluster for dimensions bigger than 20 and in particular verified the validity of the Alexander-Orbach conjecture about the spectral dimension of the random walk on the IIC (the incipient infinite cluster, i.e. the critical percolation cluster conditioned to survive). A key tool to tackle the long range nature of the model introduced in \cite{ganguly2024ant} was the concept of a big loop ensemble $\mathsf{BLE}$ which, informally, denotes an ensemble of three loops, with points on them whose euclidean distances are smaller than the size of the loops, making them unlikely. It turns out that long range connections in the loop percolation model occur via occurrences of $\ble$ which allows to bound the probabilities of various events. However, union bound estimates render $\ble$s not useful in $d\le 8$. This and other related constraints manifest in the $d>20$ assumption in the main result in \cite{ganguly2024ant}.  

Even more recently, in \cite{cai2024incipient,cai2024one,cai2024quasi,drewitz2023arm,drewitz2024cluster,drewitz2024critical} the intermediate dimensions were treated,  while in \cite{cai2024incipient} which has overlaps with the themes of this present paper, the existence of the IIC was established for all dimensions but six.\\

\noindent
\textbf{Our contributions.} In this article we develop new arguments and estimates  which hold all the way down to $d>6$.  We primarily develop two geometric techniques.\\

\noindent
(1)  The first is a  comparison of the chemical distance to Euclidean distance. This will allow us to use extrinsic estimates  to assert intrinsic conclusions. \\

\noindent
(2) The second perspective adopted in this paper which we believe could be a powerful approach to study such long range  models involves comparison of geodesics induced by the percolation geometry to certain `simplified' versions of such paths which essentially don't backtrack in a `loop' sense. While the term non-backtracking is often used to describe paths on a graph that do not ever traverse the same edge in succession in opposite orientations, here we use it to mean a path such that for any loop, once the path exits it, it never returns to the same later in its journey unlike geodesics which on account of length minimization may return to the same loop multiple times. We will not elaborate on this definition further here; see Definition \ref{nonbacktracking def} for the formal description. The advantage is that  arguments that apply in the bond case but fail in the loop model due to the presence of long loops can often be made to work for such modified non-backtracking paths. Finally, we establish comparison principles between the lengths of the geodesics and the simplified paths showing that the latter cannot be much longer, allowing us to transfer back the estimates to the original geodesic.
\\

We next offer a quick informal glimpse of the results in this paper.
\\

\noindent
$\bullet$ 
We start with the following local connectivity estimate which while being of independent interest will serve as a key input in several subsequent results in this article.
The statement asserts that local connectivity is comparable to global connectivity. More precisely, the two point function $\P(0\ar x)$ is comparable to the restricted two point function $\P(0\stackrel{\mathbb{B}_{\beta r}} \ar x)$, i.e., when the connecting path is not allowed to leave $\B_{\beta r}$, for a large but universal constant $\beta.$ While we expect that this is indeed true for all $x,$ we prove that this is true on average which will suffice for many of our applications.
$$
\sum_{y \in \mathbb{B}_{2 r}}\P(0\stackrel{\mathbb{B}_{\beta r}}{\longleftrightarrow} y) \geq c r^{2}.
$$
Surprisingly, we could not spot an estimate of this form in the literature even in the case of bond percolation. Shortly before this article was posted, a more sophisticated version of a similar statement appeared in \cite{cai2024quasi} which was then used in the construction of the IIC for all dimensions excluding the critical one $d=6$. However, the methods in this paper are quite different and rely less on the exact properties of the GFF.
\\

\noindent
$\bullet$ {Our next result is about a comparison of the chemical distance to Euclidean distance showing a quadratic relation. This is consistent with the fractal dimension of the geodesics in the $\IIC$ being $2.$} A version of this for the usual bond percolation in high dimensions was proven in \cite{chatterjee2023subcritical}.
\\

\noindent
$\bullet$ The theme of comparing the extrinsic and the intrinsic metrics as well as true and simplified geodesics will both feature crucially in the proof of the intrinsic/chemical one arm exponent matching the mean field prediction for all $d>6$. 
\\

\noindent
$\bullet$ As a final application of the local connectivity estimate, we rigorously establish the validity of a heuristic presented by Werner in \cite{werner2021clusters} asserting that in the loop percolation model, removal of loops with diameter bigger than $r^{6/d}$ does not affect the two point function between points at Euclidean distance of order $r$.  We prove that an averaged version of this statement holds even on the removal of all loops of diameter $r^{2/(d-4)}$. Note that $6/d >2/(d-4)$ since $d>6.$

\subsection{Models and main results.}

We now move on to a more formal description of the percolation model we study to set up the framework for the statements of our main results.

To define the GFF on $\cable$ (defined formally shortly), we first  introduce the Discrete Gaussian Free Field (DGFF) on $\Z^d$. For each point $x \in \mathbb{Z}^{d}$, we consider a continuous-time simple random walk on $\mathbb{Z}^{d}$ starting from $x$ with jump rate $\frac{1}{2d}$ in each direction. Let $\mathbb{P}_{x},$  $\mathbb{E}_{x}$ be the measure of this simple random walk, denoted by $\{S_{t}\}_{t \geq 0}$. Then the Green's function is defined as 
$$
G(x,y):= \mathbb{E}_{x}\left[\int_{0}^{\infty} \mathds{1}_{S_{t}=y} d t\right].
$$
The DGFF $\{\phi_{x}\}_{x\in \mathbb{Z}^{d}}$ is a mean-zero Gaussian field on $\mathbb{Z}^{d}$ with covariance structure 
$$
\text{Cov} (\phi_{x},\phi_{y}) = G(x,y) ,\; \forall x,y \in \mathbb{Z}^{d}.
$$
Next, we define the cable graph $\cable$. 
 For each edge $e \in E(\Z^d),$ say with endpoints $x$ and 
 $y$, let $I_e$ be the line segment joining $x$ and $y$, 
 scaled by $d,$ thus a compact interval of length $d$ which
  will often also be denoted by $[x,y]$ or interchangeably $[e]$. Then  $\cable:= \cup_{e\in E(\Z^d)}I_e.$

%
The GFF $\{\tilde{\phi}_{x}\}_{x\in \cable}$ on the metric graph $\cable$ is a generalization of the DGFF. Given a DGFF $\{\phi_{x}\}_{x\in \mathbb{Z}^{d}}$, set $\tilde{\phi}_{v}= \phi_{v}$ for $v \in \mathbb{Z}^{d}$, and for each interval $I_{\{x,y\}}$, $\{\tilde{\phi}_{x}\}_{x\in I_{\{x,y\}}}$ is an independent Brownian bridge of length $d$ with variance 2 at time 1, condition on $\tilde{\phi}_{x} = \phi_{x}$ and $\tilde{\phi}_{y} = \phi_{y}$; see \cite{lupu2016loop} for more details of the construction of GFF on the metric graphs. 
The GFF level-set, as a site percolation model, is defined as $E^{\geq h}= \{x \in \tilde{\mathbb{Z}^{d}}:\; \tilde{\phi}_{x} \geq h\}$ ($h\in \mathbb{R}$). We denote by {$\widetilde{\mathbb{P}}$, $\widetilde{\mathbb{E}}$}  the law of GFF on $\tilde{\mathbb{Z}}^{d}$ and the associated expectation respectively.

As alluded to already, in \cite{lupu2016loop} Lupu also established an isomorphism theorem for this model connecting it to a loop percolation model which we will define shortly.
  Relying on this, in \cite{lupu2016loop} he had further established the following all important two-point estimate: There exist $c,C>0$ such that for any $x,y\in \Z^d$,
\begin{align} \label{two point0}
   c |x-y|^{2-d} \le  \wt \P(x \ar y) \le C |x-y|^{2-d},
\end{align}
{where $x\overset{\tilde{E}^{\ge 0}}{\lbij} y$,  or simply $x \ar y$,
 will be the notation to denote that $x,y$ belong to the same connected component of {$\tilde{E}^{\ge 0}$}. }

The proof of the above was based on an exact description of the percolation cluster in terms of a Poisson loop soup. While defining it precisely will involve some preparation and hence is deferred to Section \ref{pre}, we introduce some basic objects that will allow us to state our results.  It can be often simpler to think of loops as rooted at various points. The intensity of a loop rooted at a given point  of length in $[\ell, 2\ell]$ is approximately $\frac{1}{\ell^{d/2}}$ corresponding to the Green's function of the usual random walk on $\Z^d.$ One should informally then think of the loop soup as a union of independent Poisson processes, one for each vertex with the process corresponding to a vertex supported on loops rooted at that vertex. We will use  $\wt \cL$ to denote the random collection of all loops in the union of the point processes. $\wt \cL_{\le \ell}$ will denote the set of all loops in $\wt\cL$ of length at most $\ell.$ For precise definitions of all these notions refer to Section \ref{pre}.

Finally, we introduce some notations to allow us to state our results formally.

\subsection{Notations}
For $A_1,A_2\subseteq \Z^d$, define $d^\extr(A_1,A_2):= \min \{|x_1-x_2| : x_1\in A_1,x_2\in A_2\},$ 
where $|\cdot|$ denotes the Euclidean $\ell^1$-norm.  For $r\in \N$, let $\B(x,r):=B^\extr(x,r) := \{y \in \Z^d:  |x-y| \le r\}$ be a box of center $x$ and radius $r$, 
with $\B_r$ often used as a short hand for $\B(0,r)$. 
For $A\subseteq \cable$, let $|A|$ be the number of lattice points in $A$.
Given a percolation cluster, a subset of $\Z^d$ or $\cable,$ and two points $x,y$ in this cluster, $d(x,y)$ will be the chemical length, i.e., the length of the shortest path connecting $x$ and $y$. We call $d\left(\cdot,\cdot\right)$ the chemical distance. We will denote by  $B(v,r)$  the metric ball of radius $r$ around $v\in \wt \Z^d.$

\subsection{Main results}

\newcommand{\wtp}{\wt \P}

Our first result is a comparison of the Euclidean and chemical distances proving a quadratic relation reflecting the fractal nature of the critical percolation cluster.  A related statement for the bond case appears in \cite{chatterjee2023subcritical}.

\begin{theorem}\label{thm 1.3}

For any $d>6$ the following holds.
\begin{enumerate}
\item \textbf{Point to point case.} There exists constants $c\left(d\right),C\left(d\right) >0$ such that for any $x \in \mathbb{Z}^{d}$, 
\begin{align}\label{1.4}
c|x| ^{2} \leq \widetilde{\mathbb{E}}[d \left(0,x\right) \mid 0 \leftrightarrow x ]\leq C |x| ^{2}.
\end{align}
\item \textbf{Point to boundary case upper bound.}\\
Tightness:  For any $\delta>0$, there exists $M(d,\delta)>0$ and $c=c(\gamma)$ such that for any $N\in \mathbb{N}$, 
\begin{align}\label{7}
{\frac{\# \{N \leq r\leq 2N: \widetilde{\mathbb{P}}(d(0,\partial \mathbb{B}_{r})\geq M r^{2}\mid 0\leftrightarrow \partial \mathbb{B}_{r}) \leq \delta\}}{N} \geq c.}
    \end{align}

\noindent
Expectation bound: For any $\kappa \in (0,1)$, there exists $C(d,\kappa)>0$ such that for any $r \in \mathbb{N}$, 
\begin{align}\label{6}
{\widetilde{\mathbb{E}}[d(0,\partial \mathbb{B}_{\kappa r})\mid 0\leftrightarrow \partial \mathbb{B}_{r}] \leq C r^{2}.}
\end{align}

\item \textbf{Point to boundary case lower bound.} There exists $c(d)>0$ such that for any $r \in \mathbb{N}$,
{
\begin{align} \label{1.5}
\widetilde{\mathbb{E}}[d\left(0,\partial \mathbb{B}_{r}\right) \mid 0 \leftrightarrow  \partial \mathbb{B}_{r}] \geq cr^{2}.
\end{align}
}
\end{enumerate}
\end{theorem}

In words, the first part of the above result states that conditioned on a point $|x|$ to be connected to the origin, the chemical distance $d(0,x)$ scales as $|x|^2$ while the second part states the same for the chemical distance to the boundary of the Euclidean box $\B_r.$

Our next result while being of independent interest will serve as an important input in the proof of the above as well as many of the subsequent arguments while being of independent interest. The statement asserts that local connectivity is comparable to global connectivity. More precisely, it claims that the two point function $\wtp(0\ar x)$ is comparable to the restricted two point function $\wtp(0\stackrel{\mathbb{B}_{\beta r}} \ar x)$ 
for a large but universal constant $\beta.$ While we expect that this is indeed true for all $x,$ the next result says that at least this is true on average.

\begin{theorem} \label{lem 4.4}For $d>6$, there exist constants $c\left(d\right)>0$ and $\beta(d)>1$  such that for any $r\geq 1$

$$
\sum_{y \in \mathbb{B}_{2 r}}\widetilde{\mathbb{P}}(0\stackrel{\mathbb{B}_{\beta r}}{\longleftrightarrow} y) \geq c r^{2}.
$$
\end{theorem}
As already alluded to earlier, we could not find a statement of a similar flavor even in the bond percolation literature. 

\begin{remark}\label{dingresult} Shortly before the posting of this paper, in \cite{cai2024quasi} the authors proved a strengthening of the above result drawing a parallel of connection probabilities to harmonic functions and invoking the method of harmonic averages. On the other hand, our approach is more geometric in nature only using the two point function as input. 
\end{remark}

{Our next result establishes the chemical one-arm exponent.}
\begin{theorem}\label{thm 1.4}
For $d>6$, there exists $c\left(d\right),C\left(d\right)>0$ such that for all $r>0$, 
\begin{align}\label{1.6}
\frac{c}{r} \leq \widetilde{\mathbb{P}}(\partial B(0,r) \neq \emptyset) \leq \frac{C}{r}.
\end{align}
\end{theorem}
For the bond case, matching upper and lower bounds were established in the seminal work \cite{kozma2009alexander} under \eqref{onearm}. The results were extended to the GFF level-set case in the recent work \cite{ganguly2024ant} for all $d>20$. Thus, the above result extends the latter throughout the entire high dimensional regime.

\newcommand{\BB}{\mathbb{B}}

Our final result addresses a question of Werner posed in \cite{werner2021clusters} wherein the following heuristic was presented: when $a \in (0,d)$, the $r^{a}$-th largest Brownian loop intersecting $\mathbb{B}_{r}$ will have diameter of the order of $r^{1-a/d +o(1)}$. A quick way to see this is by observing that by the Poisson loop representation, the number of loops of length in $[\ell, 2\ell]$ (essentially equivalently diameter $\sqrt \ell$ by diffusivity of random walk) intersecting $\BB_r$ is approximately $r^d \frac{1}{\ell^{d/2}}.$ Setting $\ell=r^{2\beta}$ and equating this to $r^{a}$ we get  $d -\beta d =a$, i.e., $\beta =1-a/d.$

Now, relying on an a priori bound of $r^{4+o(1)}$ (delivered using moment arguments and tree expansions) on the maximum volume of a connected component intersecting $\B_{r}$, Werner argued that the number of connected components of Euclidean diameter comparable to $r$ is at least $r^{d-6+o(1)}.$
This indicates that for $b >6/d$ we remove all loops intersecting $\mathbb{B}_{r}$ of diameter larger than $r^{b}$, essentially all the large clusters of Euclidean diameter comparable to $r$, i.e. $r^{d-6+o(1)}$ large clusters. This naturally leads to the conjecture that the removal of these big loops do not affect the two-point function. In the recent article \cite{cai2023one} the authors established a statement similar in spirit but for the one arm exponent.
In this article we prove the following local connectivity estimate confirming the conjecture in an averaged sense. 
Thus, Werner's assertion in terms of loop lengths can be stated as the removal of all loops of length greater than $r^{12/d+o(1)}$ having no effect on the two point function.

The following result states that even the removal of all loops of length at least $r^{4/(d-4)}$ will not affect the two point function in an averaged sense. 
Note in particular that $4/(d-4) < 12/d$ since $d>6$. Thus, up to the averaging, this is a stronger statement than what comes out of Werner's heuristic.

\begin{theorem} \label{werner12}
For $d>6$ and $b > \frac{4}{d-4}$ we have that for $\beta$ as in Theorem \ref{lem 4.4} there exists some {constant $C=C(d,\beta)>0$ such that for any $r \in \mathbb{N}$},

\begin{align} \label{large loop 3}
    0 \leq \sum_{x\in \mathbb{B}_{r}}\widetilde{\mathbb{P}}(0 \overset{\mathbb{B}_{\beta r}}{\longleftrightarrow}x) - \sum_{x\in \mathbb{B}_{r}}\widetilde{\mathbb{P}}(0 \overset{\mathbb{B}_{\beta r},\ \wt{\mathcal{L}}_{\leq r^{b}} }{\longleftrightarrow}x) \leq C r^{2-\eta(b)},
\end{align}
where $\eta(b) := 2-b(d/2-2)$. Note that $\eta(b)>0$ when $b >\frac{4}{d-4}$. Moreover, we have 
\begin{align} \label{large loop 4}
    \sum_{x\in \mathbb{B}_{r}}\widetilde{\mathbb{P}}(0 \overset{ \widetilde{\mathcal{L}}_{\leq r^{b}}}{\longleftrightarrow}x) \geq c r^{2},
\end{align}
where $c =c(d) >0$ is some constant.
\end{theorem}

To handle some of the key difficulties in analyzing this percolation model stemming from its long range nature the  approach we adopt involves an important idea of comparing geodesics (shortest paths in the percolation cluster) to their ``simple" counterparts (a notion to be defined precisely later) which are non-backtracking in terms of the sequence of loops they trace out.  See Figure \ref{chain figure}. \\

Before proceeding with the formal arguments, let us review the key ideas and ingredients in the proofs. 

\subsection{Key ideas and ingredients in the proofs.}\label{iop}
We start with the proof of the local connectivity statement Theorem \ref{lem 4.4}, i.e.,
there exists a constant $c(d)>0$ and $\beta>0$ such that for any $r\in\mathbb{N}$, 
\begin{align*}
    \sum_{y\in \mathbb{B}_{2 r}}\widetilde{\P}(0 \overset{\mathbb{B}_{\beta r}}{\longleftrightarrow} y) \geq c r^{2}.
\end{align*}
A key step is to obtain a regularity estimate of the one-arm probability $\pi_1(r)$ defined in \eqref{onearm}. In \cite{cai2023one} it was shown that there exists universal constants $c,C>0$ such that $$\frac{c}{r^2}\le \pi_1(r) \le \frac{C}{r^2}.$$ Thus, in principle $r^2\pi_{1}(r)$ can oscillate wildly in $[c,C]$ as $r$ varies. However, we employ an averaging argument to show that given any scale, one can find a radius $r_0$ of that scale, where the desirable regularity holds. Namely, for any small enough $\gamma,$ and $r$, there exists an $r_{0}$ between $r$ and $2r$ such that,  in a suitable sense,
\begin{align*}
\widetilde{\P}(0\leftrightarrow \partial\mathbb{B}_{r_{0}}) \approx \widetilde{\P}(0\leftrightarrow \partial\mathbb{B}_{(1+\gamma)r_{0}}),
\end{align*}  

With this regularity estimate as an input we can show that conditioning on $0\leftrightarrow \partial \mathbb{B}_{r_{0}}$, with high probability, $|\{x\in \partial \mathbb{B}_{r_{0}}: 0\overset{\mathbb{B}_{r_{0}}}{\longleftrightarrow} x\}| \gtrsim r^{2}$. Having such a lower bound on the surface volume, allows us to conclude akin to  \cite[Theorem 2]{kozma2011arm}, which shows that conditioning on $0\leftrightarrow \partial \mathbb{B}_{r_{0}}$, with at least constant probability, $|\{y\in\mathbb{B}_{2r}\setminus \mathbb{B}_{r}: 0\leftrightarrow y\}| \gtrsim r^{4}$.
Notice that the discussion here only works for $r_{0}$ satisfying the regularity estimate, which we don't know how to derive for any $r$ but nonetheless this suffices since $r< r_0 <2r.$ 
To render the  argument from \cite{kozma2011arm} which is in the setting of the usual bond percolation applicable, we rely on  \cite[Proposition 5.1]{cai2023one} which shows that in the setting of GFF level set percolation the one arm estimate continues to hold even if we only consider small loops deleting large loops. 

We now review some of the ideas appearing in the proof of Theorem \ref{thm 1.3} on the comparison between the extrinsic Euclidean metric and the intrinsic metric induced by the percolation cluster. 
The proof is inspired by \cite{kozma2009alexander}. To get an upper bound for the chemical distance in the point to point case, i.e. $\widetilde{\mathbb{E}}[d(0,x)\mathds{1}_{0\leftrightarrow x}]$, one may consider a self-avoiding path from 0 to $x$ with the shortest length among paths connecting 0 and $x$ (if there are more than one such paths, we choose any of them). Then at least in the bond case for any point $y$ on this path, the event $\{0\leftrightarrow y \circ y\leftrightarrow x\}$ occurs where $\circ$ denotes disjoint occurrence. Thus, 
\begin{align*}
    \widetilde{\mathbb{E}}[d(0,x)\mathds{1}_{0\leftrightarrow x}]\leq &\sum_{y\in\Z^{d}}\widetilde{\P}(0\leftrightarrow y \circ y\leftrightarrow x) \\ \overset{(\text{BK inequality})}{\leq}&\sum_{y\in\Z^{d}}\widetilde{\P}(0\leftrightarrow y)  \widetilde{\mathbb{P}}(y\leftrightarrow x),
\end{align*}
where BK denotes the van den Berg-Kesten (BK) inequality.
Such bounds are examples of what are known as tree expansions. 
This at least allows one to obtain upper bounds with more refined arguments needed for the lower bound accounting for overcounting in the above sum. However, this already runs into some issues in the loop soup setting which was remedied by Werner in \cite{werner2021clusters} by ``tree expanding around loops". Nonetheless, when dealing with constraints involving the intrinsic metric, serious complications arise. For instance, for an event of the type $\{0\overset{r}{\leftrightarrow}z , 0\leftrightarrow x\}$ one may  hope
to obtain analogous events of the form$\{0 \overset{r}{\leftrightarrow}w \circ  z \overset{r}{\leftrightarrow}w \circ x \leftrightarrow w\}$ leading to a tree expansion argument.
While indeed true for bond percolation, in the GFF case, the  tree expansion breaks since for any path in the union of loops and any point $w$ on the path, the same loop might appear multiple times along the path, both before and after $w$.

This issue was already encountered in \cite{ganguly2024ant} in the study of the geometry of the IIC where the authors proposed a new framework termed as the \textit{intrinsic tree expansion} (\textsf{ITE}) to address it. {Instead of expanding around a single loop as done in the vanilla version, \textsf{ITE} expands the connection events around triples of loops satisfying certain geometric constraints, which are named as \textit{big-loop ensembles} (\textsf{BLE}).} \textsf{ITE} turns out to be useful in high dimensions and was used in \cite{ganguly2024ant} to prove the Alexander-Orbach conjecture when $d>20$. However,  the framework involving \textsf{BLE}s faces a real obstruction to work for all $d>6$. Indeed, in the application of the \textsf{BLE}, one needs to consider the following type of events: for $u\in \Z^{d}$, $0\leftrightarrow u$ and there exists a loop $\Gamma \in \wt{\mathcal{L}}$ such that $u\in \Gamma$ and $|\Gamma| \geq |u|$. A crucial first moment computation involves an expression of the form 
\begin{align*}
    \sum_{u\in \Z^{d}}\widetilde{\mathbb{P}}(0\leftrightarrow u)\sum_{\Gamma \ni u: \; |\Gamma|\geq|u|}|\Gamma|\widetilde{\mathbb{P}}(\Gamma\in \wt {\mathcal{L}})\leq C \sum_{r\geq 1}r^{d-1}\cdot r^{2-d}\cdot r \cdot \frac{1}{r^{d/2-1}}.
\end{align*}
Note that the RHS is summable only when $d>8$ not $d>6$. The actual application of \textsf{BLE}s involves control on higher moments, pushing the threshold $d>8$ further up; see \cite{ganguly2024ant} for a detailed discussion.

One of the main contributions of this paper is to introduce a new approach allowing results to hold all way down to  $d>6$. The key idea is to study \textit{simple paths}, instead of the usual self-avoiding paths. Recall that given any self-avoiding path connecting two points and a chain of loops containing this path, the same loop might appear multiple times along the path causing difficulty with the tree expansion argument. {However, the path can be  ``simplified" to yield another path which only uses the loops in the chain but crucially with every loop only appears at most once; see Lemma \ref{simlpe chain lemma}.} More formally, we say a self-avoiding path is \textit{simple} if there exists a chain of loops containing this path and each loop only appears once. Thus, the statement above  says that every for every chain of loops we can always find a simple path on it.

This will enable us to carry out many of the bond percolation arguments for \textit{simple paths} when $d>6$ and obtain the sought after estimates modulo the potential issue that modifying a path to a simple path can significantly increase its length. Thus, a key part of our arguments are devoted to proving that with high probability the simplified path can not be much longer than the original path. We refrain from discussing further details.

This perspective of comparing the intrinsic and extrinsic metrics and then applying the known extrinsic estimates will also allow us to prove Theorem \ref{thm 1.4} about the intrinsic one-arm probability. The well known bond percolation argument from \cite{kozma2009alexander} establishing the chemical one arm exponent relies on volume estimates which we don't know at this point when $d>6$ (only for $d>20$ due to \cite{ganguly2024ant}). However, if we apply the extrinsic one-arm estimate $\pi_{1}(r)\lesssim r^{-2}$ and the quadratic relation between the extrinsic and intrinsic metrics from Theorem \ref{thm 1.3}, we get 
\begin{align*}
    \widetilde{\mathbb{P}}(\partial B(0,r) \neq \emptyset) \lesssim (\sqrt{r})^{-2} \lesssim r^{-1}.
\end{align*}
The proof of Theorem \ref{thm 1.4} involves making this idea rigorous. The lower bound involves showing that on the event $0\leftrightarrow \partial \B_{r},$ with constant probability $d(0,\partial \B_{r})\gtrsim r^2.$ This part of the argument relies on both the comparison statements crucially.

We end this section by mentioning that the local connectivity estimate from Theorem \ref{lem 4.4} will also be a key ingredient in the proof of Theorem \ref{werner12} on Werner's conjecture from \cite{werner2021clusters} which indicates that deleting large loops doesn't affect connections. 

\subsection{Organization of the article}
In Section \ref{pre} we introduce notations and record preliminary lemmas. In Section \ref{local} we prove Theorem \ref{lem 4.4}. In Section \ref{tail} we prove \eqref{7} in Theorem \ref{thm 1.3}, assuming the validity of. In Section \ref{compare} we establish \eqref{1.4}, \eqref{6} and \eqref{1.5}. In Section \ref{one-arm} we prove Theorem \ref{thm 1.4}. In Section \ref{largeloop} we prove Theorem \ref{werner12}. Finally, in Section \ref{appendix} we prove several technical statements whose proofs are omitted from the main body of the paper.

\subsection{Acknowledgements}
S.G. was partially supported by NSF grant DMS-1945172.

\section{Preliminaries}\label{pre}
In this section, we introduce some further notation as well as recall previously used ones, set up the definitions and record various useful results about the already alluded to loop percolation model, that will be frequently in play throughout the rest of the paper. 
\subsection{Notations}
{We denote by $\mathbb{Z}^{d}$ the $d$-dimensional square lattice. For $x,y \in \mathbb{Z}^{d}$, we add an edge ($\{x,y\}$) between $x$ and $y$ if and only if $|x-y|_{1} = 1$, where $|\cdot|_{1}$ denote the $\ell^{1}$-norm. Let $\mathbb{L}^{d}:= \{\{x,y\}: x,y\in \Z^{d},\ |x-y|_{1}=1\}$ represent the edge set of $\Z^{d}$. For $r\in\mathbb{N}$ and $x\in\Z^{d}$, let $\mathbb{B}(x,r)$ be the Euclidean box $\{y\in \Z^{d}: |x-y| \leq r\}$, where $|\cdot|= |\cdot|_{\infty}$ denotes the $\ell^{\infty}$-norm. When $x=0$, we abbreviate $\mathbb{B}(x,r) $ to $\mathbb{B}_{r}$. For $A\subseteq \Z^{d}$, define $\partial A:= \{x\in A: \exists \Z^{d} \setminus A \text{ such that }\{x,y\}\in \mathbb{L}^{d}\}$. For $A_{1},A_{2}\subseteq \Z^{d}$, define $d^{\text{ext}}(A_{1},A_{2}) : =\min\{|x_{1}-x_{2}|: x_{1}\in A_{1},x_{2}\in A_{2}\}$. For $A \subseteq \Z^{d}$ and $v \in \Z^{d}$, we say $v \sim A$ or $A \sim v$ if $d^{\text{ext}}(A,v) \leq 1$. We will use $\widetilde{\mathbb{Z}}^{d}= \cup_{\{x,y\} \in \mathbb{L}^{d}} I_{\{x,y\}}$, where each $I_{\{x,y\}}$ is an interval with length $d$ and endpoints $x,y$ to denote the metric graph. For $B \subseteq \widetilde{\Z}^{d}$, let $|B|$ be the number of lattice points contained in $B$.

Throughout the paper we denote by $c,C$ deterministic constants whose values may change from line to line. The notations $c_{1}, C_{1},c_{2}, C_{2},\cdots$ are global constants and fixed once they appear, and without additional specification, they depend only on $d$. When we choose constants $\delta,\gamma, K,\cdots$, we will clarify the way they depend on parameters and variables.}

We devote the next section to formally set up the loop percolation model, introduce definitions of loops and loop measures on the discrete lattice and the corresponding metric graph, and establish the connection to GFF level-set percolation. While comprehensive accounts maybe found in \cite[Section 2]{cai2023one} or \cite[Section 2]{ganguly2024ant}, to keep the article self-contained we will include the important definitions. 

\subsection{Loops and loop measures on $\Z^{d}$ and $\widetilde{\Z}^{d}$.}
A discrete-time path on $\Z^{d}$ is a function $\eta: \{0,1,\cdots,k\} \rightarrow \Z^{d}$ such that $\eta(i-1)$ and $\eta(i)$ are adjacent for any $i \in \{1,\cdots,k\}$ ($k$ is a non-negative integer). We say a path $\eta$ is a (discrete-time) loop, rooted at $\eta(0)$, if $\eta(0)=\eta(k)$. Two rooted discrete-time loops $\eta$ and $\eta^{\prime}$ are called equivalent if there exists $0\leq j\leq k$ such that $\eta(i) = \eta^{\prime}(i+j)$ for all $i$, where the time indices are considered as modulo $k$. We say $\eta \sim \eta^{\prime}$ if they are equivalent. In the rest of the paper, we write a discrete loop as $\Gamma = (x_{0},x_{1}, \cdots,x_{k})$ and define its length as $|\Gamma|=k$. 

A continuous-time path is a function $\eta: [0,T) \rightarrow \Z^{d}$ such that there exists a discrete-time path $\eta_{0}: \{0,1,\cdots,k\} \rightarrow \Z^{d}$ and $0=t_{0}<t_{1}<\cdots<t_{k} = T$ such that 
\begin{align*}
    \eta(t)= \eta_{0}(j), \qquad \forall t_{j}\leq t<t_{j+1}.
\end{align*}
For each $i =0,1,\cdots,k$, $\eta^{(i)}:= \eta^{\prime}(i)$ is called the $i$-th lattice point of $\eta$ and $t_{i+1}-t_{i}$ is the $i$-th holding time.
We say a continuous-time loop $\eta$ is a continuous loop, rooted at $\eta(0)$, if $\eta(0)= \eta(T)$.

To define the measure of loops, we consider $\{X_{t}^{x}\}_{t\geq 0}$ as a continuous-time simple random walk on $\Z^{d}$ starting from $x$ with holding time at every vertex being i.i.d. exponential random variables with rate 1 and define $p_{t}(x,y) := \mathbb{P}(X_{t}^{x}=y)$ as the transition kernel of this random walk. The bridge measure, $\P^{t}_{x,x}(\cdot)$ is the conditional distribution of $\{X_{s}^{x}\}_{0\leq s\leq t}$ given $X_{t} =x$. We define a loop measure $\mu$ on the set of rooted continuous-time loops on $\Z^{d}$ as the following: 
\begin{align*}
    \mu(\cdot) = \sum_{x\in \Z^{d}}\int_{0}^{\infty} \frac{1}{t} \P^{t}_{x,x}(\cdot) p_{t}(x,x) dt.
\end{align*}
Notice that since $\mu$ is invariant under time-shift, it also induces a measure on the space of continuous-time loops modulo the equivalent relation above. The equivalent classes are called continuous-time loops.

For a discrete loop $\Gamma=(x_{0},x_{1},\cdots,x_{k})$, we define its multiplicity $J(\Gamma)$ as the maximum integer such that the subsequences $(x_{(j-1)kJ^{-1}},x_{(j-1)kJ^{-1}+1},\cdots , x_{jkJ^{-1}})$ are identical for $j =1,2,,\cdots,J$. Then we have 
\begin{align*}
    \mu &( \{  \text{continuous-time loop $\gamma$ on $\Z^{d}$}: \exists  \text{ rooted continuous-time loop $\eta$ on $\Z^d$} \text{ such that } \nonumber \\
    &\eta \in \gamma \text{ and } \eta^{(i)} = x_i \text{ for $i=0,1,\cdots,k$}\}) = (2d)^{-k} / {J((x_0,\cdots,x_k))}.
\end{align*}

Now we define the continuous-time loops on $\widetilde{\Z}^{d}$ and the corresponding loop measure. A path on $\widetilde{\Z}^{d}$ is a continuous function $ \tilde{\rho}: [0,T] \rightarrow \widetilde{\Z}^{d}$ and a rooted continuous-time loop is a path $\tilde{\rho}$ such that $\tilde{\rho}(0)= \tilde{\rho}(T)$. Similarly, as the discrete case, a continuous-time loop on $\tilde{\Z}^{d}$ is an equivalent class of rooted continuous-time loops on $\tilde{\Z}^{d}$ such that one can be transformed into another by a time-shift. In this paper, we write a continuous-time loop on $\widetilde{\Z}^{d}$ as $\widetilde{\Gamma}$.

As said in \cite[Section 2.6.2]{cai2023one}, any continuous-time loop on $\widetilde{\Z}^{d}$ must be one of the following types: 
\begin{enumerate}
    \item fundamental loops: loops which visit at least two points on $\Z^{d}$;
    \item point loops: loops which visit exactly one point on $\Z^{d}$; 
    \item edge loops: loops which don't visit any points on $\Z^{d}$ but are only contained in an interval $I_{e}$ for some edge $e$ on $\Z^{d}$.
\end{enumerate}
For a loop $\widetilde{\Gamma}$ on $\widetilde{\Z}^{d}$ of either fundamental or point type, we define the corresponding discrete loop  as 
\begin{align*}
    \textsf{Trace}(\widetilde{\Gamma}):= (x_{0},x_{1},\cdots,x_{k}),
\end{align*}
where $x_{i}$s denote consecutive adjacent points on $\Z^{d}$ that $\widetilde{\Gamma}$ passes through. It follows from definition that the outputs of two equivalent rooted continuous-time loops on $\widetilde{\Z}^{d}$ belong to the same equivalent class of discrete loops on $\Z^{d}$.

To define the counterpart loop measure $\widetilde{\mu}$, we introduce Brownian motion on $\widetilde{\Z}^{d}$ denoted by  $\{\widetilde{X}_{t}^{x}\}_{t\geq 0}$. It starts from $x\in \widetilde{\Z}^{d}$ and moves as a standard Brownian motion in the interior of edges $I_{e}$ for some $e$ and when it hits a point $x \in \Z^{d}$, it uniformly chooses one of the edges incident on $x$ and then behaves as a Brownian excursion from $x$ in this edge. Once we have this canonical diffusion, by the methods in \cite{fitzsimmons2014markovian} we can construct an associated measure $\tilde{\mu}$ on the space of continuous-time loops on $\widetilde{\Z}^{d}$.
\subsection{Loop soup}\label{loopsoup1234}
For $\alpha>0$, we define the loop soup $\mathcal{L}_{\alpha}$ to be the Poisson point process in the space of discrete loops with intensity measure $\alpha \mu$. Similarly, we can define a Poisson point process $\widetilde{\mathcal{L}}$ on the space of continuous-time loops with intensity measure $\alpha \tilde{\mu}$. Since we will focus on the case $\alpha =\frac{1}{2}$ in the paper, we drop $\alpha$ from the subscript and just write $\mathcal{L}$ and $\widetilde{\mathcal{L}}$. Next, we explain how to obtain the loop soup $\widetilde{\mathcal{L}}$ from $\mathcal{L}$. Let $\widetilde{\mathcal{L}}^{\text{f}}$, $\widetilde{\mathcal{L}}^{\text{p}}$ and $\widetilde{\mathcal{L}}^{\text{e}}$ be the point processes consisting of the fundamental loops, point loops, and edge loops in $\widetilde{\mathcal{L}}$ respectively, and they are independent by the thinning property of Poisson process. We also define $\mathcal{L}^{\text{f}}$ and $\mathcal{L}^{\text{e}}$ as the point processes of fundamental loops and point loops in $\mathcal{L}$ analogously.

For $\gamma \in \mathcal{L}^{\text{f}}$, choose any $\eta$ in the equivalence class $\gamma$ and define the range $\textsf{Range}(\gamma)$ be the union of edges passed by $\eta$ and additional Brownian excursions at each $\eta^{(i)}$ conditioned on returning to $\eta^{(i)}$ before visiting its neighbors and the total local time at $\eta^{(i)}$ is the $i$-th holding time of $\eta$. Then $\textsf{Range}(\gamma)$ is the range of the corresponding loop of $\gamma$ in $\widetilde{\mathcal{L}}^{\text{f}}$. Similarly, one can construct point loops in $\widetilde{\mathcal{L}}^{\text{p}}$ using loops in $\mathcal{L}^{\text{p}}$. For any edge $e$, the union of the ranges of the loops in $\widetilde{\mathcal{L}}^{\text{e}}$ has the same low as the union of non-zero points of a standard Brownian bridge on $I_{e}$. From this construction, we have that projecting $\tilde{\mu}$ to the discrete underlying loop yields the measure $\mu$. We refer readers to \cite[Section 2]{lupu2016loop} and \cite[section 2.6]{cai2023one} for more details.\\

Though the main object in this paper is the Brownian loop soup on a metric graph, we find it more convenient to consider the projected discrete loop soup $\mathcal{L}$. In particular, the following inequality will be helpful.
For any given discrete loop $\Gamma$, 
\begin{align}\label{loop prob}
    \widetilde{\P}(\Gamma \in \mathcal{L}) &= \widetilde{\P}(\exists \widetilde{\Gamma}\in \widetilde{\mathcal{L}}:\; \mathsf{Trace}(\widetilde{\Gamma}) = \Gamma) \nonumber\\ 
    &\leq \frac{1}{2} \widetilde{\mu}(\{\widetilde{\Gamma}\text{ loop on }\widetilde{\Z}^{d}:\; \mathsf{Trace}(\widetilde{\Gamma}) = \Gamma\})\nonumber\\ 
    &=\frac{1}{2}J(\Gamma)^{-1}(2d)^{-|\Gamma|}\leq (2d)^{-|\Gamma|}.
\end{align}

\noindent
\textbf{Glued loops:} As we will shortly see, instead of loops, we will in fact consider another object called glued loops introduced in \cite[Section 3.3]{cai2023one}. For any connected set $A \subseteq \Z^{d}$ with $|A| \geq 2$, denote by $\overline{\gamma}_{A}$  the union of ranges of fundamental loops that visit and only visit points in $A$. For $x\in \Z^{d}$, let $\overline{\gamma}_{x}$ be the union of ranges of point loops that pass through $x$. Finally, for any edge $e$ in $\Z^{d}$, let $\overline{\gamma}_{e}$ be the union of ranges of edge loops whose ranges are subsets of $I_{e}$. We call $\overline{\gamma}_{A}$, $\overline{\gamma}_{x}$ and $\overline{\gamma}_{e}$ as {glued loops}. We denote by $\overline{\mathcal{L}}$ the point process on the space of glued loops induced by $\widetilde{\mathcal{L}}$ and use the notation $\overline{\Gamma}$ for glued loops in the paper. Also, note that the collections of glued loops, with different index sets, are independent. Now we introduce some notations that will be used throughout the paper. For a discrete loop $\Gamma$, we define its vertex projection onto $\mathbb{Z}^{d}$ as 
\begin{align*}
    \mathsf{VRange}(\Gamma) := \{v\in \Z^{d}: v \text{ is contained in }\Gamma\},
\end{align*}
Similarly, for a glued loop $\overline{\Gamma}$, we define 
\begin{align*}
    \mathsf{VRange}(\overline{\Gamma}) := \{v\in \Z^{d}: v \text{ is contained in }\overline{\Gamma}\}.
\end{align*}
For a discrete loop $\Gamma$ in $\mathcal{L}$, we define the corresponding glued loop 
\begin{align*}
    \overline{\Gamma}:= \overline{\gamma}_{\mathsf{VRange}(\Gamma)} \in \overline{\mathcal{L}}.
\end{align*}
Note that in the above case, the glued loop $\overline{\mathcal{L}}$ is of  fundamental or point type. Conversely, for a glued loop $\overline{\Gamma} \in \overline{\mathcal{L}}$ of fundamental or point type, we denote by $\mathsf{Dis}(\overline{\Gamma})$ the collection of discrete loops $\Gamma \in \mathcal{L}$ such that $\mathsf{VRange}(\Gamma) = \mathsf{VRange}(\overline{\Gamma})$. Let $|\overline{\Gamma}| := |\Gamma|$ where $\mathsf{VRange}(\Gamma) = \mathsf{VRange}(\overline{\Gamma})$.

Finally, we introduce some notations to denote various connection events some of which have already appeared earlier. Let $A,B \subseteq \Z^{d}$. For a collection of glued loops $\overline{\mathscr{S}}$ on $\widetilde{\Z}^{d}$, define $\{A \overset{\overline{\mathscr{S}}}{\longleftrightarrow} B\}$ as the event that there exist points $x\in A$ and $y\in B$ such that $x$ and $y$ are connected using only glued loops in $\overline{\mathscr{S}}$. For $r \in \mathbb{N}$, define $\{A \overset{\mathbb{B}_{r},\overline{\mathscr{S}}}{\longleftrightarrow} B\}$ as the event that exist points $x\in A$ and $y\in B$ such that $x$ and $y$ are connected using only glued loops in $\overline{\mathscr{S}}$ and edges in $\mathbb{B}_{r}$. For sets of a single point like $A = \{x\}$, we simply write $A =x$. We abbreviate 

\begin{align*}
    A \leftrightarrow B: = A \overset{\overline{\mathcal{L}}}{\longleftrightarrow} B. \qquad A \overset{\mathbb{B}_{r}}{\longleftrightarrow} B : = A \overset{\mathbb{B}_{r},\overline{\mathcal{L}}}{\longleftrightarrow} B.
\end{align*}

\subsection{Isomorphism theorem} The following is the key coupling between the GFF on $\widetilde{\Z}^{d}$ and the loop soup $\widetilde{\mathcal{L}}$ found by Lupu. 
\begin{lemma}[Proposition 2.1 in \cite{lupu2016loop}] \label{iso}
There is a coupling between the loop soup $\widetilde{\mathcal{L}}$ and the GFF $\{\tilde{\phi}_{x}\}_{x\in \widetilde{\Z}^{d}}$ such that the clusters composed of loops in $\widetilde{\mathcal{L}}$ are the same as the sign clusters of $\{\tilde{\phi}_{x}\}_{x\in \widetilde{\Z}^{d}}$. Here the sign clusters of GFF are the maximal connected components of $\{x\in \widetilde{\Z}^{d}: \tilde{\phi}_{x}^{2}\neq 0\}$.
\end{lemma}

\subsection{FKG inequality and BKR inequality} We record two classical correlation inequalities, well known for classical bond percolation, in the loop setting.
First, we state the FKG inequality for $\widetilde{\mathcal{L}}$. We say a function of the collection of loops $f$ is increasing if for any two collections of loops $\widetilde{\mathscr{S}}_{1}$ and $\widetilde{\mathscr{S}}_{2}$ with $\widetilde{\mathscr{S}}_{1} \subseteq \widetilde{\mathscr{S}}_{2}$ one has $f(\widetilde{\mathscr{S}}_{1})\leq f(\widetilde{\mathscr{S}}_{2})$. The following FKG inequality for the loop soup follows directly from the FKG inequality for a Poisson process \cite[Lemma 2.1]{janson1984bounds}.
\begin{lemma}
    $f$ and $g$ are two bounded increasing measurable functions of the collection of loops. Then 
    \begin{align*}
        \widetilde{\E}[f(\widetilde{\mathcal{L}})g(\widetilde{\mathcal{L}})] \geq \widetilde{\E}[f(\widetilde{\mathcal{L}})]\widetilde{\E}[g(\widetilde{\mathcal{L}})].
    \end{align*}
\end{lemma}

Next, we introduce the van den Berg-Kensten-Reimer (BKR) inequality. This inequality was first proposed by den Berg and Kesten in \cite{van1985inequalities} and was proved by van den Berg and Fiebig \cite{van1987combinatorial} and Reimer \cite{reimer2000proof}. We refer readers to a nice exposition by Borgs, Chayes and Randall \cite{borgs1999van}. 

We say a collection of glued loops certifies some event $A$ if on the realization of this collection of glued loops, $A$ occurs regardless of the realization of all other glued loops. For two events $A$ and $B$, define $A \circ B$ to be the event there are two disjoint collections of glued loops such that one collection certifies $A$ and the other certifies $B$.  Note that for two disjoint collections of glued loops, it is still possible that a glued loop in on recollection can intersect a glued loop in the other. 
\begin{lemma}[BKR inequality]
    For any two events $A$ and $B$ that depend on finitely many glued loops, we have 
    \begin{align*}
        \widetilde{\P}(A \circ B) \leq \widetilde{\P}(A) \widetilde{\P}(B).
    \end{align*}
\end{lemma}

A potential issue is that some events we study do not satisfy the finitary condition in BKR inequality, but we can still apply this inequality by taking limits. Recall our notations $A \leftrightarrow B$ and $A \overset{\mathbb{B}_{r}}{\longleftrightarrow} B$, for some subsets $A,B \subseteq \Z^{d}$. We call these two types of events ``connecting events". The next lemma states that the BKR inequality still holds for connecting events, even though they do not satisfy the finitary condition. The proof is the same as the proof of \cite[Corollary 3.4]{cai2023one}.
\begin{lemma}\label{bkr}
    For any sequence of connecting events $A_{1},A_{2},\cdots, A_{k}$, then 
    \begin{align*}
        \widetilde{\P}(A_{1}\circ A_{2}\circ \cdots \circ A_{k}) \leq \prod_{i=1}^{k} \widetilde{
            \P}(A_{i}).
    \end{align*}
\end{lemma}

\subsection{Tree expansion}
Now we present a key combinatorial tool we will employ to study loop soups, termed as  tree expansion. This was first introduced by Aizenman and Newman \cite{aizenman1984tree} in the context of bond percolation. Later this approach was adapted to the setting of loop soup by Werner in \cite{werner2021clusters}. This approach together with the BKR inequality will feature  in many of our arguments.  We now review some aspects of this referring the readers to \cite[Section 3.2]{ganguly2024ant} and \cite[Section 3.4]{cai2023one} for a more elaborate treatment. 

For a discrete loop $\Gamma$, $\overline{\Gamma}$ denotes the corresponding glued loop with the same $\mathsf{VRange}$. Note that different discrete loops can give rise to the same glued loop. Also, for a discrete loop $\Gamma =\{x\}$ of singleton type, we set $\overline{\Gamma}:= \emptyset$.

We record the all important tree expansion as the following lemma. 

\begin{lemma}[Tree expansion] \label{ete}
Let $A,B_{1},B_{2} \in \mathbb{Z}^{d}$. Then under the event $\{A \leftrightarrow B_{1}, A \leftrightarrow B_{2}\}$, there exists a discrete loop $\Gamma \in \mathcal{L}$ and $v_{1},v_{2},v_{3} \in \mathbb{Z}^{d}$ such that the following holds: 
\begin{enumerate}
    \item $d^{\text{ext}}(\Gamma, v_{i}) \leq 1$ for $i=1,2,3$;
    \item The following connections occur disjointly:
    \begin{align*}
        v_{1} \overset{\overline{\mathcal{L}}\setminus \{\overline{\Gamma}\}}{\longleftrightarrow} B_{1} \circ v_{2} \overset{\overline{\mathcal{L}}\setminus \{\overline{\Gamma}\}}{\longleftrightarrow} B_{2}  \circ v_{3} \overset{\overline{\mathcal{L}}\setminus \{\overline{\Gamma}\}}{\longleftrightarrow} A.
    \end{align*}
\end{enumerate}
\end{lemma}

Though our statement of tree expansion is slightly different from the statements in \cite{cai2023one} and \cite{ganguly2024ant}, the proof is essentially the same so we omit it and refer the readers to the proof of \cite[Lemma 3.8]{ganguly2024ant}.

\subsection{Preliminary estimates}
In this part, we collect some preliminary results that will be useful for our purpose. 

\subsubsection{Two-point function estimate}
As a consequence of Lemma \ref{iso} combining with the estimates of Green's function, Lupu \cite{lupu2016loop} proved the sharp bound for the two-point function of $\widetilde{\mathcal{L}}_{1/2}$. Recall that $x\leftrightarrow y$ is the event $x\overset{\widetilde{\mathcal{L}}}\longleftrightarrow y$.
\begin{lemma}[Proposition 5.2 in \cite{lupu2016loop}]\label{lem 2.5}
For $d>6$, there exists $c(d),C(d)>0$ such that for any $x,y \in\Z^{d}$,
\begin{align}\label{two point}
    c |x-y|^{2-d} \leq \widetilde{\mathbb{P}}(x\leftrightarrow y) \leq C|x-y|^{2-d}.
\end{align}
\end{lemma}

\subsubsection{One arm probability}
The one arm probability, $\mathbb{P}(0\leftrightarrow \partial \mathbb{B}_{r})$, is an observable of central interest in the study of critical behavior of percolation models. For critical bond percolation, it was conjectured that there exists a constant $\rho$ such that $\mathbb{P}(0\leftrightarrow \partial \mathbb{B}_{r}) = r^{-\frac{1}{\rho} +o(1)}$ and the constant $\rho$ is called the critical exponent. The groundbreaking work of Smirnov \cite{smirnov2001critical} and Lawler, Schramm, Werner \cite{lawler2002one} showed that for critical site percolation on two-dimensional triangular lattice $\rho = \frac{5}{48}$. On $\Z^{d}$ with $d>6$ it was conjectured that $\rho$ exists and is independent of the dimension $d$. Moreover, the conjectured value of $\rho$ is $\frac{1}{2}$. This conjecture was proven for large enough $d$ ($d \geq 11$)\cite{kozma2011arm, hara1990mean, fitzner2017mean, barsky1991percolation}. In fact Kozma and Nachmias \cite{kozma2011arm} proved that this is true for $d>6$ assuming the sharp estimate of the two-point function. For the level set of GFF, Werner \cite{werner2021clusters} asserted that when $d>6$, the GFF on $\widetilde{\Z}^{d}$ is asymptotically independent and thus should have the similar behaviors with bond percolation in high dimension, in particular the critical exponent for high dimensional GFF exists and equals $\frac{1}{2}$. Ding and Wirth \cite{ding2020percolation} employed a martingale argument and proved various polynomial bounds for one-arm probability in different dimensions. Recently, Cai and Ding \cite{cai2023one} established up-to-constant bounds for GFF when $d>6$ and obtained the critical exponent $\rho =\frac{1}{2}$. This confirmed Werner's heuristics from the perspective of the one-arm probability. We state this estimate as the following lemma. 
\begin{lemma}[\cite{cai2023one}] \label{lem 2.3}
For $d>6$, there exists $c_{1}\left(d\right),C_{1}\left(d\right)>0$ such that for any $r>0$,
\begin{align}\label{2.6}
    \frac{c_{1}}{r^2} \leq \widetilde{\mathbb{P}}\left( 0 \leftrightarrow  \partial \mathbb{B}_{r} \right) 
\leq \frac{C_{1}}{r^2}.
\end{align}
\end{lemma}

\subsubsection{Some technical bounds}\label{pre_bounds} Next we collect some straightforward bounds that will be repeatedly used in our arguments. The proofs are technical and already appear in \cite[Section 8]{ganguly2024ant} and \cite[Lemma 4.2]{cai2023one}, so we skip them.

\begin{lemma} \label{lem 2.1}
For $d >6$, there exists $C\left(d\right) > 0$ such that for all $x,y \in \mathbb{Z}^{d}$:
\begin{enumerate}
    \item If $\min\{\alpha_{1},\alpha_{2}\}>-d$, then
    \begin{align}\label{2.1}
    \sum_{z \in \mathbb{Z}^{d}}|z-x|^{\alpha^{1}}|z-y|^{\alpha_{2}} \leq 
    C|x-y|^{\alpha_{1}+\alpha_{2}+d},
    \end{align}

    \item 
    \begin{align}\label{2.2}
    \sum_{z_{1},z_{2} \in \mathbb{Z}^{d}}|x-z_{1}|^{2-d} |z_{1}-z_{2}|^{2-d}|z_{2}-y|^{2-d} \leq C|x-y|^{6-d},
    \end{align}
  
    \item 
    \begin{align}\label{5:53}
        \sum_{u,v\in \mathbb{Z}^{d}}|u-x|^{2-d}|v-y|^{2-d}|u|^{2-d}|v|^{2-d}|u-v|^{2-d} \leq C|x|^{2-d}|y|^{2-d}.
    \end{align}
    \item For $a \neq d-1$, there exists $C\left(d,a\right)>0$ such that for any $r \in \mathbb{N}$,
    \begin{align}\label{2.3}
    \mathop{\max}_{y \in \mathbb{Z}^{d}} \sum_{x \in \partial \mathbb{B}_{r}} |x-y|^{-a} \leq 
    Cr^{\left(d-1-a\right)  \vee 0}.
    \end{align}

    \item For $a \neq d$, there exists $C\left(d,a\right)>0$ such that for any $r \in \mathbb{N}$,
    \begin{align}\label{2.4}
    \mathop{\max}_{y \in \mathbb{Z}^{d}} \sum_{x \in  \mathbb{B}_{r}} |x-y|^{-a} \leq 
    Cr^{\left(d-a\right)  \vee 0}.
    \end{align}
\end{enumerate}
\end{lemma}
Note that (\ref{2.2}) combining with two-point function estimate implies the triangle condition, an indicator of mean-field behavior suggested in \cite{aizenman1984tree, barsky1991percolation}
\begin{align}\label{triangle}
    \sum_{x,y \in \Z^{d}}\widetilde{\mathbb{P}}(0\leftrightarrow x) \widetilde{\mathbb{P}}(x\leftrightarrow y)\widetilde{\mathbb{P}}(y\leftrightarrow 0) <\infty.
\end{align}

Alone with these estimates with a direct computation, we get the following lemma:

\begin{lemma}
    For $d>6$ and any $\varepsilon>0$, there exists $K(\varepsilon)>0$ such that
    \begin{enumerate}
        \item For any $x\in \Z^{d}$,
        \begin{align}\label{bond computation 1}
        \sum_{\substack{
        u:|u| \geq K, \\
        v :|v| \geq K
        }} |u|^{2-d}|v|^{2-d}|u-v|^{2-d}|x-v|^{4-d} \leq \varepsilon |x|^{4-d}.
        \end{align}
        \item For any $v\in \Z^{d}$ with $|v| \geq K$ and $x\in \Z^{d}$,
        \begin{align}\label{bond computation 1.5}
            \sum_{y,u,w \in \mathbb{Z}^{d}} |u|^{2-d} |w-u|^{2-d} |y-w|^{2-d}|y+v-u|^{2-d} |u-x|^{2-d} \leq \varepsilon |x|^{4-d}.
        \end{align}
        \item There exists $C(d)>0$ such that for any $\alpha \in (0,1/3)$ and $v\in \Z^{d}$ with $|v|\geq K$,
        \begin{align}\label{bond computation 2}
            \sum_{x,y \in \B_{\alpha r}}\sum_{z\in \B_{\alpha^{2}r}} |x|^{2-d}|y-x|^{2-d}|x-z|^{2-d}|z+v-y|^{2-d}d^{\text{ext}}(y, \partial \B_{r})^{-2} \leq C\alpha^{2}\e.
        \end{align}
        \item There exists $C(d)>0$ such that for any $\alpha \in (0,1/3)$ and $r\in \mathbb{N}$:
        \begin{align}\label{bond computation 2.5}
            \sum_{v\in \mathbb{B}_{r}}\sum_{z\in \mathbb{B}_{\alpha^{2} r}}\sum_{\substack{u\in\mathbb{B}_{\alpha r}\\|u-z|\geq K}}|u|^{2-d}|u-z|^{2-d}|z-v|^{2-d}|v-u|^{2-d}d^{\text{ext}}(v,\partial \mathbb{B}_{r})^{-2} \leq C \alpha^{2}\varepsilon + C \alpha^{6}.
        \end{align}
    \end{enumerate}
\end{lemma}

\subsubsection{Discrete loop estimates} To finish this section we present various estimates for discrete loops, which will be frequently used in the  paper. Most of the estimates are cited from \cite[Section 2.3]{ganguly2024ant} and we omit the proofs. 

We denote by $\mathscr{S}_{0}$ the collection of all discrete loops on $\Z^{d}$. Throughout the paper, a summation over discrete loops satisfying certain conditions $A$ will be shortened as 
\begin{align*}
    \sum_{\Gamma \text{ satisfying } A}: = \sum_{\substack{\Gamma \in \mathscr{S}_{0} \\ \Gamma \text{ satisfying }A}}.
\end{align*}

The first estimate is the bound for discrete loops passing through a given point. 

\begin{lemma}[Lemma 2.9 in \cite{ganguly2024ant}]\label{loop one point}
    For $d>6$, there exists $C(d)>0$ such that for any $L\geq 1$ and $i \geq 0$ with $i+1<d/2$, 
    \begin{align*}
        \sum_{\substack{\Gamma \ni 0\\ |\Gamma| \geq L}}|\Gamma|^{i}\widetilde{\mathbb{P}}(\Gamma \in \mathcal{L}) \leq C L^{i+1-d/2}.
    \end{align*}
\end{lemma}

Next, we move to a two-point estimate. 

\begin{lemma}[Lemma 2.11 in \cite{ganguly2024ant}]\label{loop two points}
For $d>6$, there exists $C(d)>0$ such that for any $u,v\in\Z^{d}$ and $L \geq 1$,

\begin{align*}
    \sum_{\substack{\Gamma \in u,v \\ |\Gamma|\geq L}} \widetilde{\P}(\Gamma \in \mathcal{L}) \leq C L^{1-d/2}|u-v|^{2-d}.
\end{align*}
\end{lemma}  

The following probability bounds for the existence of loops passing through three points is another key input in our proof.

\begin{lemma}[Lemma 2.10 in \cite{ganguly2024ant}]\label{loop three points}
For $d>6$, there exists $C(d)>0$ such that for any $u,v,w\in\Z^{d}$, 
\begin{align*}
    \sum_{\Gamma \ni u,v,w} \widetilde{\mathbb{P}}(\Gamma \in \mathcal{L})\leq C |u-v|^{2-d}|v-w|^{2-d}|w-u|^{2-d}.
\end{align*}
\end{lemma}
In our applications, the condition may not be exactly $u,v,w \in\Gamma$ but $u,v,w \sim \Gamma$. However, by a union bound and translation invariance, the estimates above continue to hold in the latter case as well.

Applying the three-point estimate, one obtains a bound on the following quantity which appears frequently in tree expansions. 

\begin{lemma}[Lemma 2.14 in \cite{ganguly2024ant}]\label{loop tree estimate}
    For $d>6$, there exists $C(d)>0$ such that for any $x,y \in \Z^{d}$, 
    \begin{align*}
        \sum_{\Gamma}\sum_{\substack{z_{1}\in\Z^{d} \\ z_{1} \sim \Gamma}}\sum_{\substack{z_{2}\in\Z^{d} \\ z_{2} \sim \Gamma}} |\Gamma| \widetilde{\mathbb{P}}(x \leftrightarrow z_{1})\widetilde{\mathbb{P}}(z_{2} \leftrightarrow y) \widetilde{\mathbb{P}}(\Gamma \in \mathcal{L}) \leq C|x-y|^{4-d}.
    \end{align*}
\end{lemma}

The next lemma is stating that the diameter of a loop of length $\ell$ is around $\sqrt{\ell}$.

\begin{lemma}[Lemma 2.13 in \cite{ganguly2024ant}]\label{local clt}
    For any $\kappa>0$, there exists $C(d)>0$ such that for any $u,v\in\mathbb{Z}^{d}$, 
    \begin{align*}
        \sum_{\substack{\Gamma \ni u,v\\ |\Gamma|\leq |u-v|^{2(1-\kappa)}}}\widetilde{\P}(\Gamma\in \mathcal{L}) \leq C e^{-|u-v|^{2\kappa}/4}.
    \end{align*}
\end{lemma}

The following lemmas will be useful in Section \ref{compare} 
\begin{lemma}\label{lem 5.5}
    For $d>6$ and any $x, y \in \mathbb{Z}^{d}$, there exists some constant $C(d)>0$ such that 
    \begin{align*}
        \sum_{\substack{\Gamma\ni 0}}\widetilde{\mathbb{P}}(\Gamma \in \mathcal{L}) \widetilde{\mathbb{P}}(\overline{\Gamma} \overset{\overline{\mathcal{L}}\setminus \{\overline{\Gamma}\}}{\longleftrightarrow} x\circ \overline{\Gamma} \overset{\overline{\mathcal{L}}\setminus \{\overline{\Gamma}\}}{\longleftrightarrow} y) \leq C |x|^{2-d}|y|^{2-d}.
    \end{align*}
    Recall that $\overline{\Gamma}$ is the glued loop such that $\mathsf{VRange}(\overline{\Gamma}) = \mathsf{VRange}(\Gamma)$.
\end{lemma}

\begin{proof}
    If $\{\overline{\Gamma} \overset{\overline{\mathcal{L}}\setminus \{\overline{\Gamma}\}}{\longleftrightarrow} x\circ \overline{\Gamma} \overset{\overline{\mathcal{L}}\setminus \{\overline{\Gamma}\}}{\longleftrightarrow} y\}$ happens, then we can find $u,v \sim \Gamma$ such that $u \overset{\overline{\mathcal{L}}\setminus \{\overline{\Gamma}\}}{\longleftrightarrow} x \circ v \overset{\overline{\mathcal{L}}\setminus \{\overline{\Gamma}\}}{\longleftrightarrow} y$. Then by a union bound and BKR inequality we have 
    \begin{align}
         &\sum_{\substack{\Gamma\ni 0}}\widetilde{\P}(\Gamma\in \mathcal{L})\widetilde{\mathbb{P}}(\overline{\Gamma} \overset{\overline{\mathcal{L}}\setminus \{\overline{\Gamma}\}}{\longleftrightarrow} x\circ \overline{\Gamma} \overset{\overline{\mathcal{L}}\setminus \{\overline{\Gamma}\}}{\longleftrightarrow} y) \leq \sum_{\Gamma\ni 0}\sum_{u,v \sim \Gamma}\widetilde{\mathbb{P}}(\Gamma \in \mathcal{L}) \widetilde{\mathbb{P}}(u \overset{\overline{\mathcal{L}}\setminus \{\overline{\Gamma}\}}{\longleftrightarrow} x \circ v \overset{\overline{\mathcal{L}}\setminus \{\overline{\Gamma}\}}{\longleftrightarrow} y) \nonumber \\ 
         &\leq \sum_{u,v \in \mathbb{Z}^{d}} \mathbb{P}(u \leftrightarrow x)\widetilde{\mathbb{P}}(v \leftrightarrow y)\sum_{\substack{\Gamma \in 0\\ \Gamma \sim u,v}}\widetilde{\mathbb{P}}(\Gamma \in \mathcal{L}) \nonumber \\ &\overset{(\text{Lemma \ref{loop three points}})}{\leq}C \sum_{u,v \in \mathbb{Z}^{d}}|u-x|^{2-d}|v-y|^{2-d}|u|^{2-d}|v|^{2-d}|u-v|^{2-d}\nonumber\\
         &\overset{(\ref{5:53})}{\leq} C|x|^{2-d}|y|^{2-d}.\label{103}
    \end{align}
\end{proof}

The last lemma is an estimate involving one-arm probability. The proof is rather technical and deferred to the appendix.

\begin{lemma}\label{lem 5.8}
    There exists $C(d)>0$ such that for any $r\in \mathbb{N}$, $y\in \mathbb{B}_{r/3}$ and $w \in \mathbb{B}_{r}$,
    \begin{align*}
        \sum_{\Gamma \sim w} \widetilde{\mathbb{P}}(\Gamma \in \mathcal{L})\widetilde{\mathbb{P}}(\overline{\Gamma} \overset{\overline{\mathcal{L}}\setminus \{\overline{\Gamma}\},\mathbb{B}_{r}}{\longleftrightarrow} y\circ \overline{\Gamma} \overset{\overline{\mathcal{L}}\setminus \{\overline{\Gamma}\}}{\longleftrightarrow} \partial\mathbb{B}_{r}) \leq C |y-w|^{2-d}d^{\text{ext}}(w,\partial \mathbb{B}_{r})^{-2} .
    \end{align*}
\end{lemma}
The $1/3$ in the condition $y\in \B_{r/3}$  above is arbitrary, it can be replaced by any $ c \in (0,1)$ with a different constant $C$. 

With all the above preparation, we are now in a position to dive into the proofs.


\section{Local connectivity estimates}\label{local}

In this section we prove Theorem \ref{lem 4.4}.
The first order of business is to introduce all the different definitions of the various kinds of loops we will encounter. Again, we will be brief, reviewing only the bare minimum and the interested reader is encouraged to refer to  \cite[Section 6]{cai2023one} for a more comprehensive treatment.

Some of these definitions are technical and will not appear explicitly in our proof but implicitly through results we will quote from \cite{cai2023one}.
For $n \in \mathbb{N}$ and $x\in \Z^{d}$, let $\widehat{\mathbb{B}}_{n}(x): = \{y\in\Z^{d}: |y-x| \leq n, |y-x|_{1} < nd\}$ be the box obtained by removing all the conner points of $\mathbb{B}_{n}(x)$, recall that $|\cdot|_{1}$ is the $\ell^{1}$ norm on $\mathbb{Z}^{d}$. When x is the origin, we may write $\widehat{\mathbb{B}}_{n} := \widehat{\mathbb{B}}_{n}(0)$. For any $x\in \partial \widehat{\mathbb{B}}_{n}$, we denote $x^{\text{in}}$ the unique point in $\partial\mathbb{B}_{n-1}$ such that $x^{\text{in}}\sim x$. For $r \in \mathbb{N}$ and $M \in [1,\infty]$, let 
$\widetilde{\mathcal{L}}( M)$ be the set of continuous loops $\widetilde{\Gamma} \in \widetilde{\mathcal{L}}$ with $\text{diam}(\widetilde{\Gamma}) \leq M$  and  $\Psi_{r,M}^{1}$ be the cluster containing 0 formed by the following loops in $\widetilde{\mathcal{L}}( M)$: 
\begin{enumerate}
    \item fundamental loops intersecting $\widetilde{\mathbb{B}}_{r}$; 
    \item point loops intersecting $\mathbb{B}_{r-1}$;
    \item edge loops contained in $\widetilde{\mathbb{B}}_{r}$.
\end{enumerate}
These three types of loops are called `involved loops'. Let $\Psi_{r,M}^{2}$ be defined as 
\begin{align*}
    \Psi_{r,M}^{2}:= \{x \in \partial \widehat{\mathbb{B}}_{n}: x \notin \Psi_{r,M}^{1},\ I_{\{x,x^{\text{in}}\}} \subseteq \gamma_{x}^{\text{p}} \cup \Psi_{r,M}^{1}\}
\end{align*}
and $\Psi_{r,M} = (\Psi_{r,M}^{1},\Psi_{r,M}^{2})$. The object $\Psi_{r,M}$ is a modified version of the `collection of loops intersecting $\mathbb{B}_{r}$ with diameter at most $M$'. Such a modification was indeed introduced in \cite{cai2023one} and we will shortly quote a result from the latter (see Lemma \ref{lem 4.2} below). 

We say $\mathbf{A} = (\mathbf{A}^{1},\mathbf{A}^{2})$ is \textit{admissible} if $\mathbf{A}$ is a possible configuration of $\Psi_{r,M}$. For any $\mathbf{A}$, let $\widetilde{\mathcal{L}}_{U,M,\mathbf{A}}$ be the set of the following types of loops in $\widetilde{\mathcal{L}}(M)$ (the subscript $U$ stands for \textit{unused} following \cite{cai2023one}): 
\begin{enumerate}
    \item involved loops $\widetilde{\Gamma}$ with $\widetilde{\Gamma} \cap \mathbf{A}^{1} = \emptyset$; 
    \item loops $\widetilde{\Gamma}$ with $\widetilde{\Gamma} \cap \widetilde{\mathbb{B}}_{r}=\emptyset$;
    \item point loops intersecting $\mathbf{A}^{1} \cap \partial \widehat{\mathbb{B}}_{r}$; 
    \item point loops intersecting $\partial \widehat{\mathbb{B}}_{r} \setminus (\mathbf{A}^{1} \cup \mathbf{A}^{2})$ and not intersecting $I_{\{x,x^{\text{in}}\}}\cap \mathbf{A}^{1}$,
\end{enumerate}
and let $\overline{\mathcal{L}}_{U,M,\mathbf{A}}$ be the  collection of glued loops corresponding to $\widetilde{\mathcal{L}}_{U,M,\mathbf{A}}$. Then we define  $\widetilde{\mathcal{L}}_{U,M}$ (or $\overline{\mathcal{L}}_{U,M}$) as $\widetilde{\mathcal{L}}_{U,M,\mathbf{A}}$ (or $\overline{\mathcal{L}}_{U,M,\mathbf{A}}$) on the event $\{\Psi_{r,M} = \mathbf{A}\}$. 

\begin{remark}\label{Markov}
    A key property of the definitions is that for any admissible tuple $\mathbf{A}$, $\{\Psi_{r, M}= \mathbf{A}\}$ is measurable with respect to $\widetilde{\mathcal{L}}(M) \setminus \widetilde{\mathcal{L}}_{U,M}$. Moreover, given $\{\Psi_{r, M}= \mathbf{A}\}$, by the thinning property of Poisson processes, the conditional distribution of $\widetilde{\mathcal{L}}_{U,M}$ is the same as $\widetilde{\mathcal{L}}_{U,M, \mathbf{A}}$ without conditioning.
\end{remark}

Let $\overline{\Psi}_{r,M} : = (\Psi_{r,M}^{1} \cap \Z^{d} \setminus \mathbb{B}_{r-1}) \cup \Psi_{r,M}^{2}$, $\psi_{r,M} := |\overline{\Psi}_{r,M}|$ and in addition
\begin{align*}
    \widehat{\Psi}_{r,M} := \Psi_{r,M}^{1} \cup \bigcup_{x \in \Psi_{r,M}^{2}}\gamma_{x}^{\text{p}}.
\end{align*}
When $M=\infty$, we may omit the subscript and denote $\overline{\Psi}_{r}:= \overline{\Psi}_{r,\infty}$ etc..

The following lemma is the key observation which appeared in \cite[Remark 6.5]{cai2023one} as the analogy to an important property in bond percolation: if $\{0\leftrightarrow (1+\gamma)r\}$, there exists $x \in \mathcal{C}_{\mathbb{B}_{r}}(0)\cap \partial \mathbb{B}_{r}$ such that $x \leftrightarrow \partial \mathbb{B}_{r}$ off $\mathcal{C}_{\mathbb{B}_{r}}(0)\cap \partial \mathbb{B}_{r}$.

\begin{lemma}[Remark 6.5 in \cite{cai2023one}]\label{recursive}
    If $0 \overset{\overline{\mathcal{L}}(M)}{\longleftrightarrow} \partial \mathbb{B}_{m}$ for some $m >r+M$, then there exists $v \in \widehat{\Psi}_{r,M}$ such that $v \overset{\overline{\mathcal{L}}_{U,M}}{\longleftrightarrow} \partial \mathbb{B}_{m}$. Moreover, we have
    \begin{align*}
        \{0 \overset{\overline{\mathcal{L}}(M)}{\longleftrightarrow} \partial \mathbb{B}_{m}\} \subseteq \bigcup_{z_{1}\in \overline{\Psi}_{r,M}} \bigcup_{z_{2}\in\Z^{d}:|z_{1}-z_{2}|_{2}\leq 1} \{z_{2} \overset{\overline{\mathcal{L}}_{U,M}}{\longleftrightarrow} \partial \mathbb{B}_{m}\}.
    \end{align*}
\end{lemma}

We fix constants $b \in (\frac{6}{d},1)$ and $\lambda \in (0,1]$. For a fixed large integer $r$, let $\overline{\Psi}_{n}^{*} = \overline{\Psi}_{n}\cap \mathbb{B}_{n+[(1+\lambda)r]^{b}}$ and $\psi_{n}^{*} = |\overline{\Psi}_{n}^{*}|$. We also pick a constant $\e$ (chosen later) and $L= \e^{\frac{3}{10}} r$. The next lemma which appeared as \cite[Theorem 6.7]{cai2023one} is another key in put to our proof. This is a loop counterpart of one of the main geometric results from \cite{kozma2011arm}. For any $n \in \mathbb{N}$, define 
\begin{align*}
    \chi_{n} = |\{x \in \mathbb{B}_{n+L} \setminus \mathbb{B}_{n}:  0 \leftrightarrow x\}|.
\end{align*}

\begin{lemma}[Theorem 6.7 in \cite{cai2023one}]\label{lem 4.2}
For $d>6$, there exists $C_{2}\left(d\right)>0$, $c_{2}\left(d\right) \in \left(0,1\right)$ such that for each fixed $\lambda \in \left(0,1\right)$ and sufficiently small fixed $\e>0$, the following holds for any large enough $r \in\mathbb{N}$ and any $n\in \left((1+\frac{\lambda}{4})r,(1+\frac{\lambda}{3})r\right)$:
\begin{align}\label{31}
\widetilde{\mathbb{P}}\left(\psi^{*}_{n} \geq L^{2}, \chi_{n} \leq C_{2} L^{4} \right) \leq \left(1-c_{2}\right) \pi_{1}\left(r\right).
\end{align}
Recall that $\pi_{1}(r) = \widetilde{\mathbb{P}}(0 \leftrightarrow \partial\mathbb{B}_{r})$.
\end{lemma}

The next ingredient we need allows us to ignore the effect of large loops. This is the already alluded to one arm variant of our result Theorem \ref{werner12} and was proven in \cite{cai2023one}. 

\begin{lemma}[Proposition 5.1 in \cite{cai2023one}] \label{lem 4.1}
For $d>6$ and any $b \in \left(\frac{6}{d},1\right)$, there exists $C\left(d\right),c\left(d,b\right)>0$ such that for all $r\geq 1$,
\begin{align}\label{4.1}
0 \leq \widetilde{\mathbb{P}}\left( 0 \leftrightarrow  \partial \mathbb{B}_{r} \right) - \widetilde{\mathbb{P}}( 0 \stackrel{\overline{\mathcal{L}}(r^{b})}{\longleftrightarrow} \partial \mathbb{B}_{r} ) \leq \frac{C_
{2}}{r^{2+c_{2}}}.
\end{align}
\end{lemma}

Recalling the strategy we presented in Section \ref{iop}, with the above preparation, we are finally ready to prove Theorem \ref{lem 4.4}.

\begin{proof}[Proof of Theorem \ref{lem 4.4}]
First we recall some constants that will be used in the proof: $c_{1},C_{1}$  are the constants in Lemma \ref{lem 4.2}, $c_{2},C_{2}$  are the constants in Lemma \ref{lem 4.1} and $c_{3},C_{3}$ are the constants in the lower and upper bounds in Lemma \ref{lem 2.5}. 

We choose $\delta$ such that $1-2\delta-\frac{1-c_{1}}{1-\delta} \geq \frac{c_{1}}{2}$, $\gamma>0$ such that $ \log \frac{c_{3}}{4C_{3}} \geq \frac{\log 2}{\log 1+\gamma} \log (1-\delta) $ and $\lambda = \frac{\gamma}{2}$. We say $r$ is \textit{regular} if $\frac{\pi_{1}((1+\gamma)r)}{\pi_{1}\left(r\right)} \geq 1-\delta$. From now on, we assume $r$ is regular until the very end where we show this suffices. Note that $\frac{\pi_{1}\left(\left(1+\gamma\right)r\right)}{\pi_{1}((1+\frac{\gamma}{2})r)} \geq \frac{\pi_{1}\left(\left(1+\gamma\right)r\right)}{\pi_{1}\left(r\right)} \geq 1-\delta$. Letting $b$ be as above, let us define the events
\begin{align*}
\widetilde{\mathcal{A}}:= \{\chi_{(1+\frac{\gamma}{2})r} \leq C_{1} L^{4}\} , \; 
\widetilde{\mathcal{B}} : = \{0 < \psi^{*}_{(1+\frac{\gamma}{2})r}\leq L^{2}\}, \; 
\widetilde{\mathcal{C}} := \{ 0\leftrightarrow \partial \mathbb{B}_{\left(1+\gamma\right)r}\} , \;
\widetilde{\mathcal{D}} := \{0 \leftrightarrow \partial \mathbb{B}_{\left(1+\frac{\gamma}{2}\right)r}\}.
\end{align*}
Since $r$ is regular, we have 
\begin{align}\label{4.6}
\widetilde{\mathbb{P}}(\widetilde{\mathcal{C}} \mid  \widetilde{\mathcal{D}}) = \frac{\pi_{1}((1+\gamma)r)}{\pi_{1}((1+\frac{\gamma}{2})r)}\geq 1-\delta.
\end{align}
Define 
\begin{align*}
\widetilde{\mathcal{E}}:= \{0 < \psi_{(1+\frac{\gamma}{2})r, [(1+\frac{\gamma}{2})r]^{b}}\leq L^{2}\},\;\widetilde{\mathcal{F}}:= \{0\stackrel{\overline{\mathcal{L}}([(1+\gamma/2)r]^{b})}{\longleftrightarrow} \partial \mathbb{B}_{\left(1+\gamma\right)r}\}.
\end{align*}
Eventually we need to bound $\widetilde{\mathbb{P}}(\widetilde{\mathcal{B}}^{c}\mid \widetilde{\mathcal{D}})$ from below. First, we seek to apply Lemma \ref{recursive} after removing loops with large diameters and hence consider the events $\widetilde{\mathcal{E}}$ and $\widetilde{\mathcal{F}}$. 
We claim that for large enough $r$,
\begin{align} \label{43}
    \widetilde{\mathbb{P}}(\widetilde{\mathcal{E}}^{c}\mid \widetilde{\mathcal{D}}) \geq 1-2\delta.
\end{align}
If $\widetilde{\mathbb{P}}(\widetilde{\mathcal{E}}) \leq \frac{C_{2}}{r^{2+c_{2}/2}}$, by the one arm estimate (\ref{lem 2.3}) we know that 
\begin{align*}
\widetilde{\mathbb{P}}(\widetilde{\mathcal{E}} \mid \widetilde{\mathcal{D}}) \leq \frac{\widetilde{P}(\widetilde{\mathcal{E}})}{\widetilde{P}(\widetilde{\mathcal{D}})}\leq
\frac{C_{2}}{r^{c_{2}/2}} \leq 2\delta
\end{align*}
for large enough $r$. Thus, we only need to care about the case that $\widetilde{\mathbb{P}}(\widetilde{\mathcal{E}}) \geq \frac{C_{2}}{r^{2+c_{2}/2}}$.
When $\widetilde{\mathcal{F}} \cap \widetilde{\mathcal{E}}$ occurs, $0<\psi_{(1+\frac{\gamma}{2})r,[(1+\frac{\gamma}{2})r]^{b}} \leq L^{2}$. By Lemma \ref{recursive} there exists $z_{1} \in \overline{\Psi}_{(1+\frac{\gamma}{2})r,[(1+\frac{\gamma}{2})r]^{b}}$ and $z_{2} \sim z_{1}$ such that $z_{2} \stackrel{\overline{\mathcal{L}}_{U, [(1+\frac{\gamma}{2})r]^{b}}}{\longleftrightarrow} \partial \mathbb{B}_{\left(1+\gamma\right)r}$. 
Thus, by the union bound 
\begin{align}\label{F cap E}
    \widetilde{\mathbb{P}}(\widetilde{\mathcal{F}} \cap \widetilde{\mathcal{E}}) \leq \widetilde{\mathbb{E}}\left[\sum_{z_{1}\in \overline{\Psi}_{(1+\frac{\gamma}{2})r,[(1+\frac{\gamma}{2})r]^{b}}}\sum_{z_{2}:|z_{2}-z_{1}|\leq 1}\widetilde{\mathbb{P}}(z_{2} \stackrel{\overline{\mathcal{L}}_{U, [(1+\frac{\gamma}{2})r]^{b}}}{\longleftrightarrow} \partial \mathbb{B}_{\left(1+\eta\right)r}\mid \Psi_{(1+\frac{\gamma}{2})r,[(1+\frac{\gamma}{2})r]^{b}})\cdot\mathds{1}_{\widetilde{\mathcal{E}}}\right].
\end{align}
Here we use the fact that $\widetilde{\mathcal{E}}$ is measurable with respect to $\Psi_{(1+\frac{\gamma}{2})r,[(1+\frac{\gamma}{2})r]^{b}}$.
 
As discussed in Remark \ref{Markov}, conditioning on $\Psi_{(1+\frac{\gamma}{2})r,[(1+\frac{\gamma}{2})r]^{b}}= \mathbf{A}$ for any admissible $\mathbf{A}$, the distribution of $\widetilde{\mathcal{L}}_{(1+\frac{\gamma}{2})r,[(1+\frac{\gamma}{2})r]^{b}}$ is the same as $\widetilde{\mathcal{L}}_{(1+\frac{\gamma}{2})r,[(1+\frac{\gamma}{2})r]^{b},\mathbf{A}}$ without conditioning. Also, for any $z_{2}$ in the sum above, 
\begin{center}
$d^{\text{ext}}(z_{2},\partial \mathbb{B}_{(1+\gamma)r}) \geq \frac{\gamma r}{8}$ since $|z_{1}| \leq (1+\frac{\gamma}{2})r + [(1+\frac{\gamma}{2})r]^{b}$. 
\end{center}
Thus, for any admissible $\mathbf{A}$, on $\{\Psi_{(1+\frac{\gamma}{2})r,[(1+\frac{\gamma}{2})r]^{b}} = \mathbf{A}\}$,
\begin{align}\label{exploration}
    \widetilde{\mathbb{P}}(z_{2} \stackrel{\overline{\mathcal{L}}_{U, [(1+\frac{\gamma}{2})r]^{b}}}{\longleftrightarrow} \partial \mathbb{B}_{\left(1+\eta\right)r}\mid \Psi_{(1+\frac{\gamma}{2})r,[(1+\frac{\gamma}{2})r]^{b}}= \mathbf{A}) &= \widetilde{\mathbb{P}}( z_{2} \overset{\overline{\mathcal{L}}_{(1+\frac{\gamma}{2})r,[(1+\frac{\gamma}{2})r]^{b},\mathbf{A}}}{\longleftrightarrow } \partial \mathbb{B}_{(1+\gamma)r}) \nonumber\\ &\leq \widetilde{\mathbb{P}}(z_{2}\leftrightarrow \partial \mathbb{B}_{(1+\gamma)r}) \leq \pi_{1}(\gamma r/8).
\end{align}
Plugging (\ref{exploration}) into (\ref{F cap E}) we  have
\begin{align}
&\widetilde{\mathbb{P}}(\widetilde{\mathcal{F}} \cap\widetilde{\mathcal{E}}) \leq \widetilde{\mathbb{E}}\left[\sum_{z_{1}\in \overline{\Psi}_{(1+\frac{\gamma}{2})r,[(1+\frac{\gamma}{2})r]^{b}}}\sum_{z_{2}:|z_{2}-z_{1}|\leq 1} \pi_{1}(\gamma r/8) \mathds{1}_{\widetilde{\mathcal{E}}} \right] \nonumber\\ &\leq \widetilde{\mathbb{E}}\left[(2d+1) \psi_{(1+\frac{\gamma}{2})r,[(1+\frac{\gamma}{2})r]^{b}}\pi_{1}(\gamma r/8) \mathds{1}_{\widetilde{\mathcal{E}}}\right] \leq C\e^{\frac{3}{5}} r^{2} \frac{C}{\gamma^{2}r^{2}}\widetilde{\mathbb{P}}(\widetilde{\mathcal{E}}) \leq C\e^{\frac{3}{5}}\widetilde{\mathbb{P}}(\widetilde{\mathcal{E}}),
\end{align}
where we use the fact that under $\widetilde{\mathcal{E}}$, $\psi_{(1+\frac{\gamma}{2})r,[(1+\frac{\gamma}{2})r]^{b}} \leq L^{2} =\e^{\frac{3}{5}}r^{2}$.
This shows that 
\begin{align}\label{F mid E}
\widetilde{\mathbb{P}}(\widetilde{\mathcal{F}} \mid \widetilde{\mathcal{E}}) \leq C \e^{\frac{3}{5}}.
\end{align}
Notice that $\widetilde{\mathcal{F}}\subseteq \widetilde{\mathcal{C}}$. By Lemma \ref{lem 4.1} 
\begin{align}\label{C-F}
&\widetilde{\mathbb{P}}(\widetilde{\mathcal{C}}\cap \widetilde{\mathcal{E}}) - \widetilde{\mathbb{P}}(\widetilde{\mathcal{F}}\cap \widetilde{\mathcal{E}}) = 
\widetilde{\mathbb{P}}((\widetilde{\mathcal{C}} \setminus \widetilde{\mathcal{F}})\cap \widetilde{\mathcal{E}}) \leq \widetilde{\mathbb{P}}(\widetilde{\mathcal{C}} \setminus \widetilde{\mathcal{F}}) \nonumber\\ &= \widetilde{\mathbb{P}}(0\leftrightarrow \partial \mathbb{B}_{(1+\gamma)r}) - \widetilde{\mathbb{P}}(0\stackrel{\overline{\mathcal{L}}([(1+\gamma/2)r]^{b})}{\longleftrightarrow} \partial \mathbb{B}_{\left(1+\gamma\right)r})
\leq \frac{C_{2}}{r^{2+c_{2}}},
\end{align}
Since we assume $\widetilde{\mathbb{P}}(\widetilde{\mathcal{E}}) \geq \frac{C_{2}}{r^{2+c_{2}/2}}$, for large enough $r$ and small enough $\e$, using (\ref{F mid E}) and (\ref{C-F}),
\begin{align}\label{4.7}
&\widetilde{\mathbb{P}}(\widetilde{\mathcal{C}} \mid \widetilde{\mathcal{E}}) =  \widetilde{\mathbb{P}}(\widetilde{\mathcal{F}}\mid \widetilde{\mathcal{E}}) + \widetilde{\mathbb{P}}(\widetilde{\mathcal{C}} \mid \widetilde{\mathcal{E}}) - \widetilde{\mathbb{P}}(\widetilde{\mathcal{F}}\mid \widetilde{\mathcal{E}})\nonumber\\ 
&= \widetilde{\mathbb{P}}(\widetilde{\mathcal{F}}\mid\widetilde{\mathcal{E}})+ \frac{\widetilde{\mathbb{P}}(\widetilde{\mathcal{C}}\cap \widetilde{\mathcal{E}}) - \widetilde{\mathbb{P}}(\widetilde{\mathcal{F}}\cap \widetilde{\mathcal{E}})}{\widetilde{\mathbb{P}}({\widetilde{\mathcal{E}}})} \leq C\e^{\frac{3}{5}} + \frac{C}{r^{c_{2}/2}} \leq \delta,
\end{align}
Using (\ref{4.6}) and (\ref{4.7}) we have 
\begin{align*}
\widetilde{\mathbb{P}}(\widetilde{\mathcal{E}}^{c} \mid \widetilde{\mathcal{D}}) \geq \widetilde{\mathbb{P}}(\widetilde{\mathcal{C}}\mid \widetilde{\mathcal{D}}) - \widetilde{\mathbb{P}}(\widetilde{\mathcal{E}}\cap \widetilde{\mathcal{C}} \mid \widetilde{\mathcal{D}}) &=
\widetilde{\mathbb{P}}(\widetilde{\mathcal{C}}\mid \widetilde{\mathcal{D}}) - \widetilde{\mathbb{P}}(\widetilde{\mathcal{C}} \mid \widetilde{\mathcal{E}})\widetilde{\mathbb{P}}(\widetilde{\mathcal{E}}\mid \widetilde{\mathcal{D}}) \\ &\geq 
1-\delta -\delta \cdot 1 = 1-2\delta,
\end{align*}
where for the second line we use the fact that $\widetilde{\mathcal{E}} \subseteq \widetilde{\mathcal{D}}$. Therefore, the claim (\ref{43}) is true.

Overall, since $\psi^{*}_{(1+\gamma)r} \geq \psi_{(1+\gamma)r,[(1+\gamma)r]^{b}}$, we have $\widetilde{\mathcal{E}}^{c} \cap \widetilde{\mathcal{D}}\subseteq \widetilde{\mathcal{B}}^{c} \cap \widetilde{\mathcal{D}}$. Therefore, by (\ref{43})
\begin{align}\label{40}
\widetilde{\mathbb{P}}(\widetilde{\mathcal{B}}^{c}\mid \widetilde{\mathcal{D}}) \geq \widetilde{\mathbb{P}}(\widetilde{\mathcal{E}}^{c}\mid \widetilde{\mathcal{D}})\geq 1-2\delta.
\end{align}
By Lemma \ref{lem 4.2} and the assumption that $r$ is regular, we know 
\begin{align*}
\widetilde{\mathbb{P}}(\widetilde{\mathcal{B}}^{c} \cap \widetilde{\mathcal{A}}\mid \widetilde{\mathcal{D}}) = \widetilde{\mathbb{P}}(\widetilde{\mathcal{B}}^{c} \cap \widetilde{\mathcal{A}}\mid 0\leftrightarrow \partial \mathbb{B}_{r})\cdot \frac{\pi_{1}(r)}{\pi_{1}((1+\gamma)r)}\leq \frac{1-c_{2}}{1-\delta},
\end{align*}
and then 
\begin{align*}
\widetilde{\mathbb{P}}(\widetilde{\mathcal{A}}^{c}\mid \widetilde{\mathcal{D}}) =
1 - \widetilde{\mathbb{P}}(\widetilde{\mathcal{A}}\mid \widetilde{\mathcal{D}}) &\geq 
1- \widetilde{\mathbb{P}}(\widetilde{\mathcal{B}}^{c} \cap \widetilde{\mathcal{A}} \mid \widetilde{\mathcal{D}}) - \widetilde{\mathbb{P}}(\widetilde{\mathcal{B}} \mid \widetilde{\mathcal{D}}) \\ &\geq 
1- \frac{1-c_{2}}{1-\delta} -2 \delta  \geq \frac{c_{2}}{2}.
\end{align*}
If we define $\widetilde{\mathcal{G}} := \{0 \leftrightarrow \partial \mathbb{B}_{\theta r}\}$ for some $\theta >1$, then 
\begin{align*}
\widetilde{\mathbb{P}}(\widetilde{\mathcal{G}}\mid \widetilde{\mathcal{D}}) = \frac{\pi_{1}\left(\theta r\right)}{\pi_{1}\left(r\right)} \leq \frac{C_{3}}{c_{3}\theta^{2}}.
\end{align*}
Finally, since on $\widetilde{\mathcal{G}}^{c}$, $|\widetilde{\mathcal{C}}_{\mathbb{B}_{\theta r}}(0) \cap \mathbb{B}_{2r}| = |\widetilde{\mathcal{C}}(0) \cap \mathbb{B}_{2r}|$ and 
$\chi_{\left(1+\frac{\gamma}{2}\right)r} \leq |\widetilde{\mathcal{C}}(0) \cap \mathbb{B}_{2r}|$, we have 
\begin{align}\label{4.9}
\widetilde{\mathbb{P}}(|\widetilde{\mathcal{C}}_{\mathbb{B}_{\theta r}}(0) \cap \mathbb{B}_{2r}| \geq C_{1}\e^{\frac{6}{5}} r^{4} \mid \widetilde{\mathcal{D}}) &\geq 
\widetilde{\mathbb{P}}(\{|\widetilde{\mathcal{C}}(0) \cap \mathbb{B}_{2r}| \geq C_{1}\e^{\frac{6}{5}} r^{4}\} \cap \widetilde{\mathcal{G}}^{c}\mid \widetilde{\mathcal{D}})  \nonumber\\ &\geq \widetilde{\mathbb{P}}(\widetilde{\mathcal{A}}^{c} \cap \widetilde{\mathcal{G}}^{c}\mid \widetilde{\mathcal{D}}) \nonumber\\ &\geq 
\widetilde{\mathbb{P}}(\widetilde{\mathcal{A}}^{c}\mid \widetilde{\mathcal{D}}) - \widetilde{\mathbb{P}}(\mathcal{G}\mid \widetilde{\mathcal{D}}) \geq 
\frac{c_{2}}{2} - \frac{C_{3}}{c_{3}\theta^{2}}.
\end{align}  
Choosing $\theta$ large enough such that $\frac{C_{3}}{c_{3}\theta^{2}} \leq \frac{c_{2}}{4}$. By (\ref{4.9}) we have 

\begin{align*}
&\sum_{y\in B_{2r}} \widetilde{\mathbb{P}}(0 \stackrel{\mathbb{B}_{\theta r}}{\longleftrightarrow} y) = \widetilde{\mathbb{E}}\left[|\widetilde{\mathcal{C}}_{\mathbb{B}_{\theta r}}(0) \cap \mathbb{B}_{2r}|\right]\geq  c_{2} \e^{\frac{6}{5}} r^{4} \widetilde{\mathbb{P}}(|\widetilde{\mathcal{C}}_{\mathbb{B}_{\theta r}}(0) \cap \mathbb{B}_{2r}| \geq C_{1}\e^{\frac{6}{5}} r^{4} \mid \widetilde{\mathcal{D}}) \widetilde{\mathbb{P}}(\widetilde{\mathcal{D}})  \\ 
& \geq c_{2} \e^{\frac{6}{5}} r^{4} \widetilde{\mathbb{P}}(|\widetilde{\mathcal{C}}_{\mathbb{B}_{\theta r}}(0) \cap \mathbb{B}_{2r}| \geq C_{1}\e^{\frac{6}{5}} r^{4} \mid \widetilde{\mathcal{D}}) \widetilde{\mathbb{P}}(0\leftrightarrow \partial \mathbb{B}_{2r}) \\ 
&\geq 
\frac{4c_{3}}{r^{2}}c_{2}\e^{\frac{6}{5}}r^{4}(\frac{c_{2}}{2} - \frac{C_{3}}{c_{3}\theta^{2}}) \geq 
c_{3}c_{2}^{2}\e^{\frac{6}{5}}r^{2}.
\end{align*}
Thus choosing $\beta = \theta$ proves the statement in Theorem \ref{lem 4.4} when $r$ is regular.

Next we show why this indeed suffices. We will show this by establishing the existence of a regular $r_{0}$ in $[r,2r]$ for any $r>0$. This finishes the proof since 
\begin{align*}
    \sum_{y\in \mathbb{B}_{2r}} \mathbb{P}(0 \stackrel{\mathbb{B}_{\beta r}}{\longleftrightarrow} y) \geq 
    \sum_{y\in \mathbb{B}_{2r_{0}}} \mathbb{P}(0 \stackrel{\mathbb{B}_{\beta r_{0}}}{\longleftrightarrow} y) \geq cr_{0}^{2} \geq \frac{c}{4}r^{2}.
    \end{align*}
To find the regular $r_{0}$ as promised, define  $\sigma\left(j\right) := \log \frac{\pi\left(\left(1+\gamma\right)^{j}r/2\right)}{\pi\left(\left(1+\gamma\right)^{j-1}r/2\right)}$ for $1\leq j \leq \lfloor \frac{\log 2}{\log 1+\gamma}\rfloor$, then 
\begin{align*}
\sum_{j =1}^{\lfloor \frac{\log 2}{\log 1+\gamma}\rfloor} \sigma_{j} = \log \prod_{j =1}^{\lfloor \frac{\log 2}{\log 1+\gamma}\rfloor}\frac{\pi(\left(1+\gamma\right)^{j}r/2)}{\pi(\left(1+\gamma\right)^{j-1}r/2)} \geq \log \frac{\pi\left(r\right)}{\pi\left(r/2\right)} \geq \log \frac{c_{1}}{4C_{1}}.
\end{align*}
Since $\sigma(j) <0$ for any $j>0$, by pigeonhole principle, there exists $r_{0} \in \left(r/2,r\right)$ such that $\log \frac{\pi\left(\left(1+\gamma\right)r_{0}\right)}{\pi\left(r_{0}\right)} \geq \log \frac{c_{1}}{4C_{1}} / \frac{\log 2}{\log 1+\gamma} \geq \log \left(1-\delta\right)$. This verifies that $r_{0}$ is regular. 
\end{proof}

\section{Tail bounds for chemical distance}\label{tail}

In this section we prove (\ref{7}), i.e. the upper tail bound for the point to boundary case, assuming \eqref{6}. The idea was already introduced in the proof of Theorem \ref{lem 4.4}. The main challenge is the absence of upper bound for $\widetilde{\E}[d(0,\partial \B_{r})\mid 0\leftrightarrow \partial \B_{r}]$. To address this issue, similar to the approach in the proof of Theorem \ref{lem 4.4}, first we establish the tail bound under an assumption of the regularity for the one-arm probability. Then we use an averaging argument to show that the one-arm probability satisfies the required regularity assumption for a constant fraction of $r \in [N,2N]$.

\begin{proof}[Proof of \eqref{7} assuming \eqref{6}]
    For any $\delta>0$, choose a constant $\gamma(\delta)>0$ such that 
    \begin{align}\label{continuity}
        \log\frac{c_{1}}{4C_{1}} \geq (\frac{\log 2}{\log 1+\gamma}-1)\log(1-\delta/2),
    \end{align}
    where $c_{1}$ and $C_{1}$ are constants such that 
    \begin{align*}
        c_{1}r^{-2} \leq \pi(r)\leq C_{1}r^{-2}.
    \end{align*}
    We first show that the tail bound in \eqref{7} is true when $r$ satisfies 
    \begin{align}\label{regular}
    \pi((1+\gamma)r)/\pi(r) \geq 1-\delta/2.
    \end{align}
    Indeed, for any $M>0$, by (6) and Markov inequality, we know that there exists constant $C^{\prime}(d,\gamma)$ such that 
    \begin{align*}
        \widetilde{\P}(d(0,\partial \mathbb{B}_{r})\geq Mr^{2}\mid 0\leftrightarrow \partial \mathbb{B}_{(1+\gamma)r}) \leq \frac{\mathbb{E}[d(0,\partial\mathbb{B}_{r})\mid 0\leftrightarrow\partial \mathbb{B}_{(1+\gamma)r}]}{Mr^{2}} \leq \frac{C^{\prime}}{M}.
    \end{align*}
    Thus, we can find $M=M(d, \gamma)$ such that 
    \begin{align*}
        \widetilde{\P}(d(0,\partial \mathbb{B}_{r})\geq Mr^{2}\mid 0\leftrightarrow \partial \mathbb{B}_{(1+\gamma)r}) \leq \frac{C^{\prime}}{M} \leq \delta/2.
    \end{align*}
    By (\ref{regular}) we have that 
    \begin{align*}
        &\widetilde{\P}(d(0,\partial \mathbb{B}_{r})\geq Mr^{2}\mid 0\leftrightarrow \partial \mathbb{B}_{r}) \\ &\leq \widetilde{\P}(d(0,\partial \mathbb{B}_{r})\geq Mr^{2}\mid 0\leftrightarrow \partial \mathbb{B}_{(1+\gamma)r})\widetilde{\P}(0\leftrightarrow \partial\mathbb{B}_{(1+\gamma)r}\mid 0\leftrightarrow \partial \mathbb{B}_{r}) + \widetilde{\P}(0\nleftrightarrow \partial \mathbb{B}_{(1+\gamma)r}\mid 0\leftrightarrow\partial \mathbb{B}_{r}) \\ &\leq \delta/2 \cdot 1 + 1- \frac{\pi_{1}((1+\gamma)r)}{\pi_{1}(r)}\overset{(\ref{regular})}{\leq} \delta/2+\delta/2 \leq \delta.
        \end{align*}
        Since $\gamma$ only depends on $\delta$, $M=M(d,\delta)$ only depends on $d$ and $\delta$. Thus, we get the upper tail bound for $r$ satisfying (\ref{regular}). 
        
        Next, we show that there are at least a constant fraction of $r$s in $[N,2N]$ satisfying (\ref{regular}). For any $N \in \mathbb{N}$ and $1\leq i\leq \lfloor \frac{\log 2}{\log (1+\gamma)}\rfloor$, define 
        \begin{align*}
            \Xi_{i}:= \log \frac{\pi((1+\gamma)^{i-1}N)}{\pi((1+\gamma)^{i}N)}.
        \end{align*}
        Then we know that 
        \begin{align*}
            \sum_{i=1}^{\lfloor \frac{\log 2}{\log (1+\gamma)}\rfloor} \Xi_{i} = \sum_{i=1}^{\lfloor \frac{\log 2}{\log (1+\gamma)}\rfloor}\log \frac{\pi((1+\gamma)^{i-1}N)}{\pi((1+\gamma)^{i}N)}  \leq \log \frac{\pi(N)}{\pi(2N)} \leq \log \frac{4C_{1}}{c_{1}}.
        \end{align*}
        Then by pigeonhole principle, there must exist $1\leq i_{0}\leq  \lfloor \frac{\log 2}{\log (1+\gamma)}\rfloor$ such that 
        \begin{align*}
        \Xi_{i_{0}} \leq \frac{\log 1+\gamma}{\log 2} \cdot \log\frac{4C_{1}}{c_{1}}\overset{(\ref{continuity})}{\leq} \log \frac{1}{1-\delta/2}.
        \end{align*}
        In particular, $r_{0}= (1+\gamma)^{i_{0}-1} N$ satisfies (\ref{regular}).
        Furthermore, for any $1\leq j \leq \gamma r$, if we define
        \begin{align*}
            \Xi^{j}_{i} = \log\frac{\pi((1+\gamma)^{i-1}(N+j))}{\pi((1+\gamma)^{i}(N+j))}, \; 1\leq i \leq \lfloor \frac{\log 2}{\log (1+\gamma)}\rfloor-1
        \end{align*}
        then by the exact same argument and (\ref{continuity}), we can find that there exists $1\leq i_{j} \leq \lfloor \frac{\log 2}{\log (1+\gamma)}\rfloor-1$ such that 
        \begin{align*}
            \pi((1+\gamma)r_{j})/\pi((1+\gamma)r_{j}) \geq 1-\delta/2,\; r_{j}:= (1+\gamma)^{i_{j}-1} (N+j).
        \end{align*}
        This implies that 
        \begin{align*}
            \# \{ N\leq r\leq 2N: \widetilde{\P}(d(0,\partial \mathbb{B}_{r})\geq Mr^{2}\mid 0\leftrightarrow \partial \mathbb{B}_{r})  \leq \delta\} \geq \gamma N,
        \end{align*}
        which finishes the proof.
\end{proof}

\section{Comparison of intrinsic and extrinsic metrics}\label{compare}
In this section we finish the proof of Theorem \ref{thm 1.3}.
Since in Section \ref{tail}, we have already proven \eqref{7} assuming that \eqref{6} holds, this section will focus on the proofs of \eqref{1.4}, \eqref{6} and \eqref{1.5}, which complete the proof of Theorem \ref{thm 1.3}.
As discussed in Section \ref{iop}, the key idea driving the results in this section involves comparing a geodesic which could have a rather complicated journey through loops, to a simplified path which passes through a sequence of distinct loops in order. However, we begin by introducing some language some of which will be imported from \cite[Section 3]{ganguly2024ant}. Recall the definition of glued loops from Section \ref{loopsoup1234}.

\subsection{Glued loop sequences and simple paths}\label{loopchain} Recalling the notion of glued loops from Section \ref{loopsoup1234}, we start with a fundamental definition.

\begin{definition}
Let $\ell: [0,T]\rightarrow \widetilde{\Z}^{d}$ be a path on $\widetilde{\Z}^{d}$. A sequence of glued loops $\overline{L}=(\overline{\Gamma}_{1},\overline{\Gamma}_{2},\cdots, \overline{\Gamma}_{k})$ is called a \textit{glued loop sequence of $\ell$} if there exist $0=t_{0}<t_{1}<\cdots<t_{k}=T$ such that for any $i=1,2,\cdots,k$,
\begin{align*}
    \ell([t_{i-1},t_{i}]) \text{ is a subset of $\overline{\Gamma}_{i}$.}
\end{align*}
In addition, for $A,B\subseteq \Z^{d}$, a sequence of glued loops $\overline{L}$ is called a \textit{glued loop sequence from $A$ to $B$} if there exists a path $\ell$ from $A$ to $B$ such that $\overline{L}$ is a glued loop sequence for $\ell$. We denote $\textsf{Set}(\overline{L})$ by the collection of glued loops in the sequence $\overline{L}$.
\end{definition}
To prevent confusion, we note that all paths dealt with in the paper will be lattice paths, i.e. sequence of adjacent lattice edges.
The next thing is to introduce the notion of a ``simple chain" and a ``simple path". 

\begin{definition}\label{nonbacktracking def}
    We say that a sequence of glued loops $\overline{L}= (\overline{\Gamma}_{1},\cdots, \overline{\Gamma}_{k})$ is a \textit{simple chain} or simple glued loop sequence if all the glued loops $\overline{\Gamma}_{i}$s are distinct. Moreover, a self-avoiding path $\ell$ is called a \textit{$\overline{\mathscr{S}}$-simple path} for a collection of glued loops $\overline{\mathscr{S}}\subseteq \overline{\mathcal{L}}$ if there exists a simple chain $\overline{L}$ for $\ell$, which only uses glued loops in $\overline{\mathscr{S}}$. For brevity, we will refer to a $\overline{\mathcal{L}}$-simple path as a simple path.
\end{definition}

Next, we introduce the definition of a geodesic and analogously a simple geodesic. 

\begin{definition}[Geodesic, simple geodesic]
    Let $A,B\subseteq \Z^{d}$ and $\ell$ be a path on $\widetilde{\Z}^{d}$. Let $\overline{\mathscr{S}}$ be a collection of glued loops in $\overline{\mathcal{L}}$. 
    \begin{enumerate}
        \item $\ell$ is called a \textit{$\overline{\mathscr{S}}$-path from $A$ to $B$} if it there exists a glued loop sequence of $\ell$ from $A$ to $B$ only using glued loops in $\overline{\mathscr{S}}$.
        \item An $\overline{\mathscr{S}}$-path $\ell$ is called an \textit{$\overline{\mathscr{S}}$-geodesic from $A$ to $B$} if it is the shortest path among all $\overline{\mathscr{S}}$-paths from $A$ to $B$.
        \item An $\overline{\mathscr{S}}$-simple path $\ell$ is called a \textit{$\overline{\mathscr{S}}$-simple geodesic from $A$ to $B$} if it is the shortest path among all $\overline{\mathscr{S}}$-simple paths from $A$ to $B$.
    \end{enumerate}
In the special case $\overline{\mathscr{S}}=\overline{\mathcal{L}},$ for brevity, 
    we will refer to $\overline{\mathcal{L}}$-geodesics and $\overline{\mathcal{L}}$-simple geodesics as geodesics and simple geodesics respectively.
\end{definition}
We next record the following lemma from \cite{ganguly2024ant}, which allows us to modify a path to a simple path. 
\begin{lemma}[(49) in \cite{ganguly2024ant}]\label{simlpe chain lemma}
    For $A,B\subseteq \Z^{d}$ and any glued loop sequence $\overline{L}$ from $A$ to $B$, 
    \begin{align}\label{simple chain}
        \text{$\exists$ a simple chain of glued loops, only using the glued loops in $\textsf{Set}(\overline{L})$, from $A$ to $B$.}
    \end{align}
\end{lemma}
This implies that for every path from $A$ to $B$, associated with a glued loop sequence $\overline{L}$, there exists a simple path from $A$ to $B$ that corresponds to a simple chaining using only the glued loops  in $\mathsf{Set}(\overline{L})$; see Figure \ref{chain figure}.

\begin{figure}[h]
    \centering
    \includegraphics[width=0.9\textwidth]{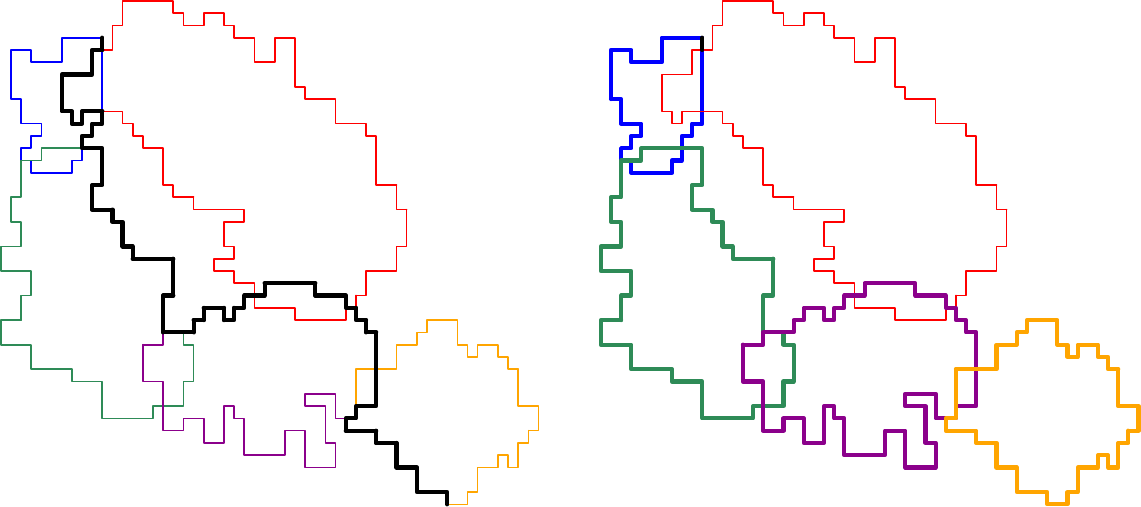}
    \caption{Illustration of the above lemma.  On the left is a path running through various loops, on the right, a subsequence of these loops is highlighted in bold, forming a simple chain.}
    \label{chain figure}
    \end{figure}

\subsection{A geometric lemma}
The next step is to introduce a key geometric lemma which allows us to compare paths and simple paths.

We say a glued loop sequence $\overline{L}$ is \textit{minimal} for a path $\ell$ if for any $\overline{\Gamma}\in \mathsf{Set}(\overline{L})$,
\begin{align}\label{5:48}
    \ell \text{ is not a $\textsf{Set}(\overline{L})\setminus \{\overline{\Gamma}\}$-path}.
\end{align}
The next lemma extracts a minimal loop sequence for any given simple self avoiding walk and bound its size.
\begin{lemma}\label{geometric lemma}
    For $A,B\subseteq \Z^{d}$ and any (self-avoiding) path $\ell$ from $A$ to $B$, there exists a minimal glued loop sequence $\mathsf{Set}(\overline{L})$ for $\ell$.
    Moreover,  \begin{align}\label{5:49}
            3|\ell| \geq |\textsf{Set}(\overline L)|.
        \end{align}
\end{lemma}

\begin{proof}
For any self-avoiding path $\ell=(x_{1},x_{2},\cdots,x_{t})$ from $A$ to $B$, there exists a glued loop sequence $\overline{L}= (\overline{\Gamma}_{1},\cdots,\overline{\Gamma}_{k})$ associated to $\ell$. We claim that there exists another glued sequence $\overline{L}^{\prime}$ of $\ell$ such that $\textsf{Set}(\overline{L}^{\prime}) \subseteq \textsf{Set}(\overline{L})$ and any $\overline{\Gamma}^{\prime} \in \textsf{Set}(\overline{L}^{\prime})$ satisfies (\ref{5:48}).
Indeed, we set $\overline{\mathscr{S}}^{0}:=\textsf{Set}(\overline{L})$. For $1\leq i\leq k$, define $\overline{\mathscr{S}}^{i}:= \overline{\mathscr{S}}^{i-1}\setminus \{\overline{\Gamma}_{i}\}$ 
if $\ell$ is a $\overline{\mathscr{S}}^{i-1}\setminus \{\overline{\Gamma}_{i}\}$-path; otherwise, define $\overline{\mathscr{S}}^{i}:=\overline{\mathscr{S}}^{i-1}$. Then we can get $\overline{\mathscr{S}}^{k} \subseteq \textsf{Set}(\overline{L})$ and there exists a glued loop sequence $\overline{L}^{\prime}$ of $\ell$ only using glued loops in $\overline{\mathscr{S}}^{k}$ satisfying (\ref{5:48}) which for brevity we will continue to refer to as $\overline L.$

To bound the size of $\overline L$ we first notice that the number of edge loops in $\textsf{Set}(\overline L)$ can at most be $|\ell|.$ Thus, we turn our attention to loops $\overline \Gamma \in \overline L $ only of fundamental or point type. 

For any $1\leq i\leq t$, define the following the collection of glued loops:
    \begin{align*}
        &\mathsf{Set}(\overline{L})^{i}:=\{ \overline{\Gamma}\in \textsf{Set}(\overline{L}): x_{i} \in \overline{\Gamma}\}.
    \end{align*}
    Note that for any $i$, the glued loops in $\mathsf{Set}(\overline{L})^{i}$ are of fundamental or point type because they pass through lattice points. Moreover, we claim that for any $\overline{\Gamma} \in \mathsf{Set}(\overline{L})$ of fundamental type or point type, there exists $1\leq i\leq t$ such that at least one of the following two conditions happen:
    \begin{enumerate}
        \item $\overline{\Gamma}$ is the unique glued loop in $\mathsf{Set}(\overline{L})^{i}$ such that 
        \begin{align}\label{optimizer 1}
            \overline{\Gamma}^{\prime}\cap [x_{i},x_{i+1}] \subseteq \overline{\Gamma}\cap [x_{i},x_{i+1}],\; \text{for any $\overline{\Gamma}^{\prime}\in \textsf{Set}(\overline{L})^{i}$.}
        \end{align}

        \item $\overline{\Gamma}$ is the unique glued loop in $\mathsf{Set}(\overline{L})^{i+1}$ such that 
        \begin{align}\label{optimizer 2}
            \overline{\Gamma}^{\prime}\cap [x_{i},x_{i+1}] \subseteq \overline{\Gamma}\cap [x_{i},x_{i+1}],\; \text{for any $\overline{\Gamma}^{\prime}\in \textsf{Set}(\overline{L})^{i+1}$.}
        \end{align}
    \end{enumerate}
    Otherwise, if $\overline{\Gamma}$ is neither the unique glued loop satisfying \eqref{optimizer 1} nor \eqref{optimizer 2}for any $1\leq i\leq t$, then $\ell$ remains a $\mathsf{Set}(\overline{\Gamma}) \setminus \{\overline{\Gamma}\}$-path after removing $\overline{\Gamma}$ from $\mathsf{Set}(\overline{L})$. This is because in this case, for any edge $[x_{i},x_{i+1}] \in \ell$, there exists another glued loop $\overline{\Gamma}^{\prime} \in \mathsf{Set}(\overline{\Gamma}) \setminus \{\overline{\Gamma}\}$ such that 
    \begin{align*}
        \overline{\Gamma}\cap [x_{i},x_{i+1}] \subseteq \overline{\Gamma}^{\prime}\cap [x_{i},x_{i+1}].
    \end{align*}
    Therefore, for every glued loop of fundamental type or point type in $\overline{\Gamma} \in \mathsf{Set}(\overline{L})$, we can assign it to the edge $[x_{i},x_{i+1}]$ for which $\overline{\Gamma}$ is either the unique glued loop satisfying \eqref{optimizer 1} or the unique glued loop satisfying \eqref{optimizer 2}. If there are multiple edges that meet this requirement, we choose the first one appearing in $\ell$. Then, by the uniqueness constraint, for every edge $e \in \ell$, there are at most two glued loops assigned to $e$. Hence, we have 
    \begin{align*}
        |\{\overline{\Gamma}\in \mathsf{Set}(\overline{L}): \text{ $\overline{\Gamma}$ is of fundamental type or point type}\}| \leq 2|\ell|.
    \end{align*} 
\end{proof}

We are now prepared to dive into the proof of Theorem \ref{thm 1.3}. We start with the straightforward proof of the upper bound in (\ref{1.4}).

\subsection{Proof of the upper bound in (\ref{1.4})} 
If $0\leftrightarrow x$, then there exists a glued loop sequence $\overline{L} = (\overline{\Gamma}_{1},\cdots,\overline{\Gamma}_{k})$ from 0 to $x$, then by Lemma \ref{simlpe chain lemma}, we know that there exists a simple path $\ell$ from 0 to $x$, associated with a simple chain. Also, if $y\in \Z^{d}$ is on the simple path $\ell$, then the following event occurs. 
\begin{center}
    There exists a glued loop $\overline{\Gamma}\in \overline{\mathcal{L}}$ such that $y \in \overline{\Gamma}$ and $0 \stackrel{ \overline{\mathcal{L}} \setminus \{\overline{\Gamma}\}}{\longleftrightarrow} \overline{\Gamma} \circ \overline{\Gamma} \overset{\overline{\mathcal{L}} \setminus \{\overline{\Gamma}\}}{\longleftrightarrow }x.$
\end{center}
Note that this glued loop can not be of edge type since $y$ is a lattice point. We call this event $\mathcal{T}_{x}(y)$. Thus, we have 
\begin{align*}
\widetilde{\mathbb{E}}[d\left(0,x\right) \mathds{1}_{0 \leftrightarrow  x}] \leq 
\sum_{y \in \mathbb{Z}^{d}}\widetilde{\mathbb{P}}(\mathcal{T}_{x}(y)).
\end{align*} Moreover, when $\mathcal{T}_{x}(y)$ occurs, there exists $z_{1},z_{2} \in \mathbb{Z}^{d}$ and a discrete loop $\Gamma$ with $y\in \Gamma$ ($\Gamma$ can be taken to be $\mathsf{Dis}(\overline{\Gamma})$ for the glued loop $\overline{\Gamma}$ in the definition of $\mathcal{T}_{x}(y)$) and $z_{1},z_{2} \sim \Gamma$  such that $\Gamma \in \mathcal{L}$ and $0\overset{\overline{\mathcal{L}} \setminus \{\overline{\Gamma}\}}{\longleftrightarrow} z_{1} \circ z_{2} \overset{\overline{\mathcal{L}} \setminus \{\overline{\Gamma}\}}{\longleftrightarrow} x$. Thus, by a union bound we have, 
\begin{align*}
\sum_{y \in \mathbb{Z}^{d}}\widetilde{\mathbb{P}}(\mathcal{T}_{x}(y)) &\leq 
\sum_{y\in \mathbb{Z}^{d}}\sum_{\Gamma \ni y}\sum_{\substack{z_{1} \in \mathbb{Z}^{d} \\ z_{1} \sim \Gamma}} \sum_{\substack{z_{2} \in \mathbb{Z}^{d} \\ z_{2} \sim \Gamma}}\widetilde{\mathbb{P}}(\Gamma \in \mathcal{L},\ 0\overset{\overline{\mathcal{L}} \setminus \{\overline{\Gamma}\}}{\longleftrightarrow} z_{1} \circ z_{2} \overset{\overline{\mathcal{L}} \setminus \{\overline{\Gamma}\}}{\longleftrightarrow} x)  \\ &\leq
\sum_{z_{1} \in \mathbb{Z}^{d} } \sum_{z_{2} \in \mathbb{Z}^{d}}\sum_{\Gamma : z_{1},z_{2} \sim \Gamma}\sum_{y\in \Gamma}\widetilde{\mathbb{P}}(\Gamma \in \mathcal{L},\ 0\overset{\overline{\mathcal{L}} \setminus \{\overline{\Gamma}\}}{\longleftrightarrow} z_{1} \circ z_{2} \overset{\overline{\mathcal{L}} \setminus \{\overline{\Gamma}\}}{\longleftrightarrow} x) \\ &\overset{\text{(BKR)}}{\leq} 
\sum_{z_{1} \in \mathbb{Z}^{d} } \sum_{z_{2} \in \mathbb{Z}^{d}}\sum_{\Gamma : z_{1},z_{2} \sim \Gamma}\sum_{y\in \Gamma}\widetilde{\mathbb{P}}(\Gamma \in \mathcal{L})\widetilde{\mathbb{P}}( 0\overset{\overline{\mathcal{L}} \setminus \{\overline{\Gamma}\}}{\longleftrightarrow} z_{1})\widetilde{\mathbb{P}}(z_{2} \overset{\overline{\mathcal{L}} \setminus \{\overline{\Gamma}\}}{\longleftrightarrow} x) \\ &\leq \sum_{z_{1} \in \mathbb{Z}^{d} } \sum_{z_{2} \in \mathbb{Z}^{d}}\sum_{\Gamma : z_{1},z_{2} \sim \Gamma}\sum_{y\in \Gamma}\widetilde{\mathbb{P}}(\Gamma \in \mathcal{L})\widetilde{\mathbb{P}}( 0 \leftrightarrow z_{1})\widetilde{\mathbb{P}}(z_{2} \leftrightarrow x)
\\ & \leq \sum_{z_{1} \in \mathbb{Z}^{d} } \sum_{z_{2} \in \mathbb{Z}^{d}}\sum_{\Gamma : z_{1},z_{2} \sim \Gamma}|\Gamma|\widetilde{\mathbb{P}}(\Gamma \in \mathcal{L})\widetilde{\mathbb{P}}( 0 \leftrightarrow z_{1})\widetilde{\mathbb{P}}(z_{2} \leftrightarrow x).
\end{align*}
By Lemma \ref{loop tree estimate} we know for any $d>6$
\begin{align*}
    \sum_{z_{1} \in \mathbb{Z}^{d} } \sum_{z_{2} \in \mathbb{Z}^{d}}\sum_{\Gamma : z_{1},z_{2} \sim \Gamma}|\Gamma|\widetilde{\mathbb{P}}(\Gamma \in \mathcal{L})\widetilde{\mathbb{P}}( 0 \leftrightarrow z_{1})\widetilde{\mathbb{P}}(z_{2} \leftrightarrow x) \leq C |x|^{4-d}.
\end{align*}
 Thus, we have 
 \begin{align*}
    \widetilde{\mathbb{E}}[d(0,x)\mathds{1}_{0\leftrightarrow x}] \leq \sum_{y\in \mathbb{Z}^{d}}\widetilde{\mathbb{P}}(\mathcal{T}_{x}(y))\leq C|x|^{4-d}.
 \end{align*}
This along with the lower bound of the two point function \eqref{two point} finishes the proof.
\qed

\subsection{Proof of the lower bound in (\ref{1.4})}
The proof involves a few parts. First, we study the geometry of simple geodesics. 
Once we obtain the estimate for a simple geodesic, we will apply Lemma \ref{simlpe chain lemma} and Lemma \ref{geometric lemma} to transfer it to the actual geodesic. The key idea is that for a simple path, the time spent on large loops is negligible while that spent on small loops is comparable with the length of the geodesic.

We begin with a bound for the simple geodesic.
\begin{proposition}\label{prop 5.2}
    For $d>6$, there exists $c(d)>0$ such that 
    \begin{align}\label{48}
        \widetilde{\mathbb{E}}[d^{\text{simple}}(0,x)\mathds{1}_{0\leftrightarrow x}] \geq c|x|^{4-d},
    \end{align}
    {where $d^{\text{simple}}(0,x)$ is the length of the simple geodesic}. 
\end{proposition}

We next record the estimate which says that the time spent on large loops is negligible. 
{Let
\begin{align*}
    \mathcal{N}_{x}(K) = \{y\in \mathbb{Z}^{d}: \text{there exists } \Gamma \in \mathcal{L}\text{ such that } y\in \Gamma,\; |\Gamma|\geq K \text{ and } 0\overset{\overline{\mathcal{L}}\setminus \{\overline{\Gamma}\}}{\leftrightarrow} \overline{\Gamma} \circ \overline{\Gamma}\overset{\overline{\mathcal{L}}\setminus \{\overline{\Gamma}\}}{\leftrightarrow} x \}.
\end{align*}}

\begin{lemma}\label{lem 5.3}
    For $d>6$, any constant $\e>0$ and any $x \in \mathbb{Z}^{d}$, there exists  $K(\varepsilon)>0$ such that 
    $$\E(\cN_x(K)) \leq \e |x|^{4-d}.$$
\end{lemma}
Recalling the event $\mathcal{T}_{x}(y)$, the above says that the points that lie on some large loop do not contribute significantly to the sum $\sum_{y\in \mathbb{Z}^{d}} \widetilde{\mathbb{P}}(\mathcal{T}_{x}(y)).$

Postponing the proofs of the above results we first finish the proof of the lower bound in (\ref{1.4}) assuming those.

\begin{proof}[Proof of the lower bound in \eqref{1.4}]

We now make the geometric claim that on the event $\{0\leftrightarrow x\}$,
\begin{align*}
    d^{\text{simple}}(0,x) \leq |\mathcal{N}_{x}(K)| + 2K d(0,x).
\end{align*}
Assuming this claim, the lower bound follows  from Proposition \ref{prop 5.2} and Lemma \ref{lem 5.3} by taking expectations on both sides.

Now we prove the claim. First, for a geodesic $\gamma_{0\rightarrow x}$, i.e. the shortest self-avoiding path from 0 to $x$, by Lemma \ref{geometric lemma}, there exists a minimal glued loop subsequence $\overline{L}$ for $\gamma_{0\rightarrow x}$ and
\begin{align*}
    |\textsf{Set}(\overline{L})| \leq 3|\gamma_{0\rightarrow x}|.
 \end{align*}
Moreover, by Lemma \ref{simlpe chain lemma} there exists a simple chain $\overline{L}_{1}$ only using glued loops in $\textsf{Set}(\overline{L})$ from 0 to $x$ and a $\textsf{Set}(\overline{L})$-simple path $\ell$ from 0 to $x$.

For any point $y\in \ell$, if $y \notin \mathcal{N}_{x}(K)$, then there exists a glued loop $\overline{\Gamma}\in \textsf{Set}(\overline{L})$ such that $y \in \overline{\Gamma}$ and $|\overline{\Gamma}| \leq K$, (if there are more than one $\overline{\Gamma}$ satisfying this requirement, we choose the first one). We denote this loop by $\overline{\Gamma}(y)$. Then we have that  
{\begin{align*}
    3|\gamma_{0\rightarrow x}| &\geq \sum_{\overline{\Gamma} \in \textsf{Set}(\overline{L})}\mathds{1}_{|\overline{\Gamma}|\leq K} \geq \sum_{\substack{\overline{\Gamma} \in \textsf{Set}(\overline{L})\\ |\overline{\Gamma}| \leq K}}\frac{1}{K}|\overline{\Gamma}| \\&\geq \frac{1}{K}\sum_{\substack{\overline{\Gamma} \in \textsf{Set}(\overline{L})\\ |\overline{\Gamma}| \leq K}} \sum_{y\in \ell \cap \overline{\Gamma}} \mathds{1}_{y\in \overline{\Gamma}} = \frac{1}{K}\sum_{y\in \ell}\sum_{\substack{\overline{\Gamma} \in \textsf{Set}(\overline{L})\\ y\in \overline{\Gamma} \\ |\Gamma| \leq K}}\mathds{1}_{y\in \overline{\Gamma}} \\ &\geq \frac{1}{K}\sum_{y\in \ell} \mathds{1}_{y\notin \mathcal{N}(K)}.
\end{align*}} 
Also, we know 
\begin{align*}
    d^{\text{simple}}(0,x) \leq |\ell| \leq \sum_{y\in \ell} \mathds{1}_{y\notin \mathcal{N}_{x}(K)} + \sum_{y\in \mathbb{Z}^{d}} \mathds{1}_{y\in \mathcal{N}_{x}(K)}.
\end{align*}
Overall we have that 
{\begin{align*}
    d^{\text{simple}}(0,x) \leq \sum_{y\in \ell} \mathds{1}_{y\notin \mathcal{N}_{X}(K)} + |\mathcal{N}_{x}(K)| \leq 3K |\gamma_{0\rightarrow x}| + |\mathcal{N}_{x}(K)|,
\end{align*}}
which proves the claim since $|\gamma_{0\rightarrow x}|= d(0,x)$.
\end{proof}

In the rest of this subsection, we prove  Proposition \ref{prop 5.2} and Lemma \ref{lem 5.3}. We start with the latter.

\begin{proof}[Proof of Lemma \ref{lem 5.3}]
    First, if $0\overset{\overline{\mathcal{L}}\setminus \{\overline{\Gamma}\}}{\longleftrightarrow} \overline{\Gamma}$, there exists $z\sim \Gamma$ such that $0\overset{\overline{\mathcal{L}}\setminus \{\overline{\Gamma}\}}{\longleftrightarrow} z$, then by BKR inequality and a union bound we have 
    \begin{align*}
        &\sum_{y\in \mathbb{Z}^{d}}\sum_{\substack{\Gamma \ni y \\ |\Gamma|\geq K}} \widetilde{\mathbb{P}}(0\overset{\overline{\mathcal{L}}\setminus \{\overline{\Gamma}\}}{\longleftrightarrow} \overline{\Gamma} \circ \overline{\Gamma}\overset{\overline{\mathcal{L}}\setminus \{\overline{\Gamma}\}}{\longleftrightarrow} x \circ \Gamma \in \mathcal{L}) \leq 
        \sum_{y\in \mathbb{Z}^{d}}\sum_{\substack{\Gamma \ni y \\ |\Gamma|\geq K}} \widetilde{\mathbb{P}}(\Gamma \in \mathcal{L})\widetilde{\mathbb{P}}(0\overset{\overline{\mathcal{L}}\setminus \{\overline{\Gamma}\}}{\longleftrightarrow} \overline{\Gamma}) \widetilde{\mathbb{P}}( \overline{\Gamma}\overset{\overline{\mathcal{L}}\setminus \{\overline{\Gamma}\}}{\longleftrightarrow} x) \\ 
        &\leq \sum_{y\in \mathbb{Z}^{d}}\sum_{\substack{\Gamma \ni y \\ |\Gamma|\geq K}} \sum_{z\sim \Gamma}\widetilde{\mathbb{P}}(\Gamma \in \mathcal{L})\widetilde{\mathbb{P}}(0\leftrightarrow z) \widetilde{\mathbb{P}}( \overline{\Gamma}\overset{\overline{\mathcal{L}}\setminus \{\overline{\Gamma}\}}{\longleftrightarrow} x) 
        \leq\sum_{y\in \mathbb{Z}^{d}}\sum_{\substack{\Gamma \ni y \\ |\Gamma|\geq K}} \sum_{z,w\sim \Gamma}\widetilde{\mathbb{P}}(\Gamma \in \mathcal{L})\widetilde{\mathbb{P}}(0\leftrightarrow z) \widetilde{\mathbb{P}}( w\leftrightarrow x) \\
        &\leq \sum_{z\in \mathbb{Z}^{d}}\widetilde{\mathbb{P}}(0\leftrightarrow z)\sum_{ \substack{\Gamma \sim z\\ |\Gamma|\geq K}}\sum_{w\sim \Gamma}\sum_{y\sim \Gamma}\widetilde{\mathbb{P}}(\Gamma \in \mathcal{L})\widetilde{\mathbb{P}}(w\leftrightarrow x)=\sum_{z\in \mathbb{Z}^{d}}\widetilde{\mathbb{P}}(0\leftrightarrow z)\sum_{ \substack{\Gamma \sim z\\ |\Gamma|\geq K}}\sum_{w\sim \Gamma}|\Gamma| \widetilde{\mathbb{P}}(\Gamma \in \mathcal{L})\widetilde{\mathbb{P}}(w\leftrightarrow x).
    \end{align*}
The proof now is a consequence of the following Lemma \ref{lem 5.4}. Taking $x$ in the lemma to be $z-x$ we know that there exists $K(\varepsilon)>0$ such that,
    \begin{align*}
        \sum_{z\in \mathbb{Z}^{d}}\widetilde{\mathbb{P}}(0\leftrightarrow z)\sum_{ \substack{\Gamma \sim z\\ |\Gamma|\geq K}}\sum_{w\sim \Gamma}|\Gamma| \widetilde{\mathbb{P}}(\Gamma \in \mathcal{L})\widetilde{\mathbb{P}}(w\leftrightarrow x) \leq \e\sum_{z\in \mathbb{Z}^{d}} |z|^{2-d}|z-x|^{2-d} \leq \e |x|^{4-d}.
    \end{align*}
    \end{proof}
    
        \begin{lemma}\label{lem 5.4}
        For any $\varepsilon$, there exists $K(\varepsilon)>0$ such that for any $x \in \mathbb{Z}^{d}$,
        \begin{align*}
            \sum_{ \substack{\Gamma \sim 0\\ |\Gamma|\geq K}}\sum_{w\sim \Gamma}|\Gamma| \widetilde{\mathbb{P}}(\Gamma \in \mathcal{L})\widetilde{\mathbb{P}}(w\leftrightarrow x) \leq \e |x|^{2-d}. 
        \end{align*}
    \end{lemma}
  The proof involves some algebra and is deferred to the appendix, Section \ref{appendix.1}.\\

We now prepare for the proof of Proposition \ref{prop 5.2}. The first step is the following lower bound. This is reminiscent of the argument from  \cite{kozma2009alexander}. Since the event $\mathcal{T}_{x}(y)$ does not involve intrinsic constraints, essentially a bond-percolation like argument can be carried out. The modifications needed to deal with the loop soup model were more or less developed in \cite[Section 4.1]{ganguly2024ant}. While we include some of the tedious details for completeness, we defer them to the appendix, Section \ref{appendix.3}. 
    
    \begin{lemma}\label{lem 5.1}
        There exists $c(d) >0$ such that for any $x\in \mathbb{Z}^{d}$
        \begin{align*}
            \sum_{y\in \mathbb{Z}^{d}} \widetilde{\mathbb{P}}(\mathcal{T}_{x}(y)) \geq c |x|^{4-d}.
        \end{align*}
    \end{lemma}

To go from the above estimate to a conditional expected lower bound for the length of a simple geodesic between $0$ and $x$, one needs to estimate the number of $y$ which may not lie on the simple geodesic but still contributes to the sum in Lemma \ref{lem 5.1}, i.e., estimate the over-counting. The basic approach would be to show that points which are at a large enough Euclidean distance from the simple geodesic do not contribute much. 

We will consider two kinds of over-counting which we define next. For any $y\in \Z^{d}$ and constant $K,K^{\prime}$, we say $y$ is \textbf{$(K,K^{\prime})$ over-counted} if there exists $v_{1},w_{1}\in \Z^{d}$ with $|y-v_{1}|\geq K^{\prime}$ and $|y-w_{1}|\geq K^{\prime}$ such that there exists discrete loops $\Gamma , \Gamma_{1}, \Gamma_{2}\in \mathcal{L}$ with $y\in \Gamma$, $v_{1}\sim \Gamma_{1}$, $w_{1}\sim \Gamma_{2}$ such that 
        \begin{align*}
            0\overset{\overline{\mathcal{L}}\setminus \{\overline{\Gamma},\overline{\Gamma}_{1},\overline{\Gamma}_{2}\}}{\longleftrightarrow} \overline{\Gamma}_{1} \circ \overline{\Gamma}_{1}\overset{\overline{\mathcal{L}}\setminus \{\overline{\Gamma},\overline{\Gamma}_{1},\overline{\Gamma}_{2}\}}{\longleftrightarrow}\overline{\Gamma} \circ \overline{\Gamma} \overset{\overline{\mathcal{L}}\setminus \{\overline{\Gamma},\overline{\Gamma}_{1},\overline{\Gamma}_{2}\}}{\longleftrightarrow} \overline{\Gamma}_{2} \circ \overline{\Gamma}_{2} \overset{\overline{\mathcal{L}}\setminus \{\overline{\Gamma},\overline{\Gamma}_{1},\overline{\Gamma}_{2}\}}{\longleftrightarrow}x \circ v_{1}\overset{\overline{\mathcal{L}}\setminus \{\overline{\Gamma},\overline{\Gamma}_{1},\overline{\Gamma}_{2}\}}{\longleftrightarrow}w_{1}.
        \end{align*}
        We say $y$ is \textbf{$K$ over-counted} if there exists $\Gamma \in \mathcal{L}$ with $y\in \Gamma$ and $|\Gamma|\geq K$ such that
        \begin{align*}
        0\overset{\overline{\mathcal{L}}\setminus \{\overline{\Gamma}\}}{\longleftrightarrow}\overline{\Gamma} \circ \overline{\Gamma} \overset{\overline{\mathcal{L}}\setminus \{\overline{\Gamma}\}}{\longleftrightarrow}x.
        \end{align*}
Before recording results which help us bound such overcounting let us emphasize that the fact that we are doing away with intrinsic constraints by virtue of working with a simple geodesic allows us to obtain such estimates. 
        
        The number of $K$ over-counted points can be controlled by Lemma \ref{lem 5.4}. The next lemma  controls $(K,K^{\prime})$ over-counted points.
        \begin{lemma}
            For any constant $\e>0$ and  any $K>0$, there exists large enough constant $K^{\prime}>0$ such that
        \begin{align*}
            \sum_{y\in \Z^{d}} \widetilde{\mathbb{P}}(y\text{ is $(K,K^{\prime})$ over-counted}) \leq \e |x|^{4-d}.
        \end{align*}
        \end{lemma}  
    While the proof bears strong resemblance to the arguments in \cite[Lemma 3.1]{kozma2009alexander} and often we will omit details in the paper in such situations, in this case we include the details to help the reader get a flavor.   
        \begin{proof}
        According to the definition of $(K,K^{\prime})$ over-counted, we apply a union bound and BKR inequality to get
        \begin{align*}
           &\sum_{y\in \Z^{d}} \widetilde{\mathbb{P}}(y \text{ is $(K,K^{\prime})$ over-counted}) \nonumber\\ 
           &\leq 
           \sum_{y\in \Z^{d}}\sum_{u,u^{\prime}\Z^{d}}\sum_{\substack{v_{1},v_{2},v_{3}\in\Z^{d}\\ |v_{1}-y|\geq K^{\prime}}}\sum_{\substack{w_{1},w_{2},w_{3}\in\Z^{d}\\ |w_{1}-y|\geq K^{\prime}}}\sum_{\substack{\Gamma \ni y\\ u,u^{\prime} \sim \Gamma\\ |\Gamma|\leq K}} \sum_{\Gamma_{1}\sim v_{1},v_{2},v_{3}}\sum_{\Gamma_{2}\sim w_{1},w_{2},w_{3}} \nonumber\\ 
           &\widetilde{\mathbb{P}}(\Gamma \in \mathcal{L})\widetilde{\mathbb{P}}(\Gamma_{1} \in \mathcal{L})\widetilde{\mathbb{P}}(\Gamma_{2} \in \mathcal{L})\widetilde{\mathbb{P}}(0\leftrightarrow v_{3})\widetilde{\mathbb{P}}(v_{2}\leftrightarrow u)\widetilde{\mathbb{P}}(u^{\prime}\leftrightarrow w_{2})\widetilde{\mathbb{P}}(v_{1}\leftrightarrow w_{1})\widetilde{\mathbb{P}}(w_{3}\leftrightarrow x).
        \end{align*}
        Reorganizing the terms to separate the ones involving $v_2$ and $w_2$, we rewrite this bound as 
        \begin{align}\label{long sum}
            &\sum_{y\in \Z^{d}}\sum_{\substack{v_{1}\in\Z^d\\|v_1-y|\geq K^{\prime}}}\sum_{\substack{w_{1}\in\Z^d\\|w_1-y|\geq K^{\prime}}}\sum_{\Gamma_{1}\sim v_{1}}\sum_{v_2,v_3\sim \Gamma_{1}}\sum_{\Gamma_2\sim w_1}\sum_{w_2,w_3\sim \Gamma_2}\nonumber\\ 
           &\cdot \left(\widetilde{\mathbb{P}}(\Gamma_{1}\in \mathcal{L})\widetilde{\mathbb{P}}(\Gamma_2\in \mathcal{L}) \widetilde{\mathbb{P}}(0\leftrightarrow v_3)\widetilde{\mathbb{P}}(v_1\leftrightarrow w_1)\widetilde{\mathbb{P}}(w_3\leftrightarrow x)\right)\cdot 
           \left(\sum_{\substack{\Gamma\ni y\\|\Gamma|\leq K}}\sum_{u,u^{\prime}\sim\Gamma}\widetilde{\mathbb{P}}(\Gamma\in\mathcal{L})\widetilde{\mathbb{P}}(v_2\leftrightarrow u)\widetilde{\mathbb{P}}(w_2\leftrightarrow u^{\prime})\right).
        \end{align}
        By Lemma \ref{loop three points} and (\ref{5:53}) recorded in Section \ref{pre_bounds}, 
        \begin{align*}
            &\sum_{\substack{\Gamma\ni y\\|\Gamma|\leq K}}\sum_{u,u^{\prime}\sim\Gamma}\widetilde{\mathbb{P}}(\Gamma\in\mathcal{L})\widetilde{\mathbb{P}}(v_2\leftrightarrow u)\widetilde{\mathbb{P}}(w_2\leftrightarrow u^{\prime}) \\ 
            \overset{(\text{Lemma \ref{loop three points}})}{\leq}& C\sum_{u,u^{\prime}\in \Z^{d}}|y-u|^{2-d}|u-u^{\prime}|^{2-d}|u^{\prime}-y|^{2-d}|v_{2}-u|^{2-d}|w_2-u^{\prime}|^{2-d}\overset{(\ref{5:53})}{\leq} C|y-v_2|^{2-d}|y-w_2|^{2-d}.
        \end{align*}
        Substituting this estimate into 
        \eqref{long sum} gives 
        \begin{align}\label{long sum 2}
            &\sum_{y\in \Z^{d}} \widetilde{\mathbb{P}}(y \text{ is $(K,K^{\prime})$ over-counted}) \leq C\sum_{y\in \Z^{d}}\sum_{\substack{v_{1}\in\Z^d\\|v_1-y|\geq K^{\prime}}}\sum_{\substack{w_{1}\in\Z^d\\|w_1-y|\geq K^{\prime}}}\sum_{\Gamma_{1}\sim v_{1}}\sum_{v_2,v_3\sim \Gamma_{1}}\sum_{\Gamma_2\sim w_1}\sum_{w_2,w_3\sim \Gamma_2}\nonumber\\ 
            &\cdot \widetilde{\mathbb{P}}(\Gamma_{1}\in \mathcal{L})\widetilde{\mathbb{P}}(\Gamma_2\in \mathcal{L}) \widetilde{\mathbb{P}}(0\leftrightarrow v_3)\widetilde{\mathbb{P}}(v_1\leftrightarrow w_1)\widetilde{\mathbb{P}}(w_3\leftrightarrow x)|y-v_2|^{2-d}|y-w_2|^{2-d} \nonumber\\ 
            &\leq C \sum_{y\in \Z^{d}}\sum_{\substack{v_{1}\in\Z^d\\|v_1-y|\geq K^{\prime}}}\sum_{\substack{w_{1}\in\Z^d\\|w_1-y|\geq K^{\prime}}}|v_1-w_1|^{2-d}\cdot \left(\sum_{\Gamma_{1}\sim v_{1}}\sum_{v_2,v_3\sim \Gamma_{1}}\widetilde{\mathbb{P}}(\Gamma_{1}\in \mathcal{L})|v_3|^{2-d}|y-v_3|^{2-d}\right) \nonumber\\ 
            &\cdot \left(\sum_{\Gamma_2\sim w_1}\sum_{w_2,w_3\sim \Gamma_2}\widetilde{\mathbb{P}}(\Gamma_2\in \mathcal{L})|w_3-x|^{2-d}|y-w_2|^{2-d}\right).
        \end{align}
        Next, applying Lemma \ref{loop three points} and (\ref{5:53}), we obtain the following bounds:
        \begin{align*}
            &\sum_{\Gamma_{1}\sim v_{1}}\sum_{v_2,v_3\sim \Gamma_{1}}\widetilde{\mathbb{P}}(\Gamma_{1}\in \mathcal{L})|v_3|^{2-d}|y-v_3|^{2-d} \leq C|v_1|^{2-d}|y-v_1|^{2-d};\\ 
            &\sum_{\Gamma_2\sim w_1}\sum_{w_2,w_3\sim \Gamma_2}\widetilde{\mathbb{P}}(\Gamma_2\in \mathcal{L})|w_3-x|^{2-d}|y-w_2|^{2-d} \leq C|w_1-y|^{2-d}|w_1-x|^{2-d}.
        \end{align*}
        Plugging these bounds back into \eqref{long sum 2}, we get
        \begin{align*}
            &\sum_{y\in \Z^{d}} \widetilde{\mathbb{P}}(y \text{ is $(K,K^{\prime})$ over-counted}) \\ &\leq C\sum_{y\in \Z^{d}}\sum_{\substack{v_{1}\in\Z^d\\|v_1-y|\geq K^{\prime}}}\sum_{\substack{w_{1}\in\Z^d\\|w_1-y|\geq K^{\prime}}}|v_1|^{2-d}|v_1-y|^{2-d}|w_1-y|^{2-d}|v_1-w_1|^{2-d}|w_1-x|^{2-d}.
        \end{align*}
         Finally, applying (\ref{bond computation 1}) from Section \ref{pre_bounds}, we get that for large enough constant $K^{\prime}$,
        \begin{align*}
            C\sum_{y\in \Z^{d}}\sum_{\substack{v_{1}\in\Z^d\\|v_1-y|\geq K^{\prime}}}\sum_{\substack{w_{1}\in\Z^d\\|w_1-y|\geq K^{\prime}}}|v_1|^{2-d}|v_1-y|^{2-d}|w_1-y|^{2-d}|v_1-w_1|^{2-d}|w_1-x|^{2-d} \leq \e |x|^{4-d},
        \end{align*}
        which completes the proof.
        \end{proof}
    
    Given the above we now finish the proof of Proposition \ref{prop 5.2}.
        \begin{proof}[Proof of Proposition \ref{prop 5.2}]
    We define 
    \begin{align*}
        N_{1}^{K,K^{\prime}}:=|\{y\in \Z^{d}: \mathcal{T}_{x}(y) \text{ happens and $y$ is neither $K$ over-counted nor $(K,K^{\prime})$ over-counted}\}|.
    \end{align*}
    We claim that for $K,K^{\prime}$ with $K<3K^{\prime}$, on the event $\{0\leftrightarrow x\}$ we have $N^{K,K^{\prime}} \leq C (K^{\prime})^{d} d^{\text{simple}}(0,x)$. 
    
    Indeed, we fix a simple geodesic $\ell^{(0,x)}=(x_{0}=0,x_{1},\cdots, x_{m}=x)$ from 0 to $x$ (there could be several simple geodesics, and we arbitrarily choose one) and we choose a simple chain $\overline{L}= (\overline{\Gamma}_{1},\cdots , \overline{\Gamma}_{k})$ for $\ell^{(0,x)}$.
    
{    For $y\in \Z^d$ with $d^{\text{ext}}(y,\ell^{(0,x)}) \geq K^{\prime}$, if $\mathcal{T}_{x}(y)$ occurs and $y$ is not $K$ over-counted, then for any $\overline{\Gamma} \in \mathcal{L}$ such that $y\in \Gamma$, we claim that $\overline{\Gamma} \notin \overline{L}$. Otherwise, there exists $z\in \ell^{(0,x)}$ such that $z\sim \Gamma$. Then since $d^{\text{ext}}(y,z) \geq K^{\prime}$ we know that $|\Gamma| \geq K^{\prime} \geq 3K$. Now since  $L= (\overline \Gamma_1, \ldots, \overline \Gamma_k)$ is simple and minimal, if $\overline \Gamma=\overline \Gamma_i$ for some $i,$ then 
$$0\overset{\overline{\mathcal{L}}\setminus \{\overline{\Gamma}\}}{\longleftrightarrow}\overline{\Gamma} \circ \overline{\Gamma} \overset{\overline{\mathcal{L}}\setminus \{\overline{\Gamma}\}}{\longleftrightarrow}x.$$
implying $y$ is $K$ over-counted, yielding a contradiction.}
    
{
 For a glued loop $\overline{\Gamma}$ such that $y\in \overline{\Gamma}$ and $0\overset{\overline{\mathcal{L}}\setminus \{\overline{\Gamma}\}}{\longleftrightarrow}\overline{\Gamma} \circ \overline{\Gamma} \overset{\overline{\mathcal{L}}\setminus \{\overline{\Gamma}\}}{\longleftrightarrow}x$, there exists two disjoint sequences of glued loops $\overline{L}_{1}$, $\overline{L}_{2}$, which do not utilize $\overline{\Gamma}$, such that 
    \begin{center}
        $\overline{L}_{1}$ is a simple chain from 0 to $\overline{\Gamma}$ and $\overline{L}_{2}$ is a simple chain from $\overline{\Gamma}$ to $x$.
    \end{center}
    Choose the indices 
    \begin{align*}
        a_{1}:= \max\{1\leq t\leq k: \overline{\Gamma}_{t}\in \overline{L}_{1}\};\; a_{2}:= \min\{a_{1} <t\leq k: \overline{\Gamma}_{t}\in \overline{L}_{2}\}.
    \end{align*} 
 If $\mathsf{Set}(\overline{L})\cap \mathsf{Set}(\overline{L}_{1}) = \emptyset$, we choose $a_{1}=1$. If $\overline{\Gamma}_{a_{1}}$ is of edge type we can replace $a_{1}$ by $a_{1}-1$. In this case, $a_{1}\neq 1$, since the first glued loop in $\overline{L}$ can not be of edge type due to the definition of the loop sequence. Notice that in this case $\overline{\Gamma}_{a_{1}-1}$ must be point of fundamental types since there cannot be consecutive glued loops of edge types in $\overline{L}$. Similarly, we set $a_{2}=k$ if $\mathsf{Set}(\overline{L})\cap \mathsf{Set}(\overline{L}_{2}) = \emptyset$, and replace $a_{2}$ by $a_{2}+1$ if $\overline{\Gamma}_{a_{2}}$ is of edge type. Thus, we can assume neither $\overline{\Gamma}_{a_{1}}$ nor $\overline{\Gamma}_{a_{2}}$ is of edge type.
    Then choose $v_{1}$ to be the exit point of $\ell^{(0,x)}$ from $\overline{\Gamma}_{a_{1}}$ and $w_{1}$ to be the entry point of $\ell^{(0,x)}$ to $\overline{\Gamma}_{a_{1}}$. 
    Hence, the following event happens
    \begin{align*}
        0\overset{\overline{\mathcal{L}}\setminus \{\overline{\Gamma},\overline{\Gamma}_{a_1},\overline{\Gamma}_{a_2}\}}{\longleftrightarrow} \overline{\Gamma}_{1} \circ \overline{\Gamma}_{a_1}\overset{\overline{\mathcal{L}}\setminus \{\overline{\Gamma},\overline{\Gamma}_{a_1},\overline{\Gamma}_{a_2}\}}{\longleftrightarrow}\overline{\Gamma} \circ \overline{\Gamma} \overset{\overline{\mathcal{L}}\setminus \{\overline{\Gamma},\overline{\Gamma}_{a_1},\overline{\Gamma}_{2}\}}{\longleftrightarrow} \overline{\Gamma}_{a_2} \circ \overline{\Gamma}_{a_2} \overset{\overline{\mathcal{L}}\setminus \{\overline{\Gamma},\overline{\Gamma}_{a_1},\overline{\Gamma}_{a_2}\}}{\longleftrightarrow}x \circ v_{1}\overset{\overline{\mathcal{L}}\setminus \{\overline{\Gamma},\overline{\Gamma}_{a_1},\overline{\Gamma}_{a_2}\}}{\longleftrightarrow}w_{1},
    \end{align*}
    Notice that $v_{1}\leftrightarrow w_{1}$ happens disjointly with other connection events since $\overline{L}$ is a simple chain. 
    Also, by the restriction $d^{\text{ext}}(y,\ell^{(0,x)}) \geq K^{\prime}$, we know that $d^{\text{ext}}(y,v_1) \geq K^{\prime}$ and $d^{\text{ext}}(y,w_1) \geq K^{\prime}$.

    Thus, we know that $y$ is $(K,K^{\prime})$ over-counted. 
    This shows that 
    \begin{align*}
        N_{1}^{K,K^{\prime}} \leq |\{y\in \Z^{d}: d^{\text{ext}}(y,\ell^{(0,x)})\leq K^{\prime}\} \leq C(K^{\prime})^{d} d^{\text{simple}}(0,x),
    \end{align*}
    which is the claim.
    
    Thus for any $\e$ and large enough constants $K,K^{\prime}$ we get
    \begin{align*}
        \widetilde{\mathbb{E}}[d^{\text{simple}}(0,x)\mathds{1}_{0\leftrightarrow x}] &\geq \sum_{y\in \Z^{d}}\widetilde{\mathbb{P}}(\mathcal{T}_{x}(y)) - \sum_{y\in \Z^{d}}\widetilde{\mathbb{P}}(y\text{ is $K$ over-counted}) - \sum_{y\in \Z^{d}}\widetilde{\mathbb{P}}(y\text{ is $(K,K^{\prime})$ over-counted}) \\ &\geq 
        c|x|^{4-d} - \e |x|^{4-d} - \e|x|^{4-d}.
    \end{align*} 
    Choose sufficiently small $\e$ we know that 
     \begin{align*}
        \widetilde{\mathbb{E}}[d^{\text{simple}}(0,x)\mathds{1}_{0\leftrightarrow x}] \geq \frac{c}{2}|x|^{4-d}.
    \end{align*}}
    \end{proof}

    In the rest of the section, we prove (\ref{6}) and (\ref{1.5}) for the point to boundary case.

    \subsection{Proof of (\ref{6})} 
    Since for any $\kappa \in(0,1)$, the proof is the same, we just choose $\kappa = 1/2$ for convenience. For any $r \in \mathbb{N}$ and $z \in \mathbb{B}_{r}$, define 
    \begin{align}\label{Tr}
        \mathcal{T}_{r}(z):= \{\text{there exists $\overline{\Gamma}\in \overline{\mathcal{L}}$ with $z\in \overline{\Gamma}$ such that } 0\overset{\overline{\mathcal{L}}\setminus \{\overline{\Gamma}\}, \mathbb{B}_{r}}{\longleftrightarrow} \overline{\Gamma} \circ \overline{\Gamma}\overset{\overline{\mathcal{L}}\setminus \{\overline{\Gamma}\}}{\longleftrightarrow} \partial \mathbb{B}_{r}\}.
    \end{align} 
     As the discussion in the proof of the lower bound in (\ref{1.4}), by Lemma \ref{simlpe chain lemma} and Lemma \ref{geometric lemma}, every geodesic path can be modified to some simple path only increasing its length as discussed in the proof of upper bound in the point to point case, we get
    \begin{align*}
        \widetilde{\mathbb{E}}[d(0,\partial \mathbb{B}_{1/2 r})\mathds{1}_{0\leftrightarrow \partial \mathbb{B}_{r}}] \leq \sum_{z\in \mathbb{B}_{1/2 r}}\widetilde{\mathbb{P}}( \mathcal{T}_{r}(z)).
    \end{align*}

    For any $z\in \mathbb{B}_{1/2 r}$, by Lemma \ref{lem 5.8}, 
    \begin{align*}
        \widetilde{\mathbb{P}}(\mathcal{T}_{r}(z)) &\leq \sum_{\Gamma \sim z} \widetilde{\mathbb{P}}(\Gamma \in \mathcal{L})\widetilde{\mathbb{P}}( 0\overset{\overline{\mathcal{L}}\setminus \{\overline{\Gamma}\}, \mathbb{B}_{r}}{\longleftrightarrow} \overline{\Gamma} \circ \overline{\Gamma}\overset{\overline{\mathcal{L}}\setminus \{\overline{\Gamma}\}}{\longleftrightarrow} \partial \mathbb{B}_{r})  \\ &\leq C |z|^{2-d}d^{\text{ext}}(z,\partial \mathbb{B}_{r})^{-2} \leq Cr^{-2}|z|^{2-d},
    \end{align*}
    where the last inequality is because $d^{\text{ext}}(z,\partial \mathbb{B}_{r}) \geq 1/2 r$ for any $z \in \mathbb{B}_{1/2 r}$. Thus, we get the upper bound 
    \begin{align*}
        \widetilde{\mathbb{E}}[d(0,\partial \mathbb{B}_{1/2 r})\mathds{1}_{0\leftrightarrow \partial \mathbb{B}_{r}}]\leq \sum_{z \in \mathbb{B}_{1/2 r}} \widetilde{\mathbb{P}}(\mathcal{T}_{r}(z)) \leq C r^{-2}\sum_{z\in \mathbb{B}_{1/2 r}} |z|^{2-d} \overset{(\ref{2.4})}{\leq} C r^{-2} \cdot r^{2} \leq C.
    \end{align*} 
    Then by the lower bound of one arm probability, we get 
    \begin{align*}
        \widetilde{\mathbb{E}}[d(0,\partial \mathbb{B}_{1/2 r})\mid 0\leftrightarrow \partial \mathbb{B}_{r}] \leq C r^{2},
    \end{align*}
    which finishes the proof.
\qed

    \subsection{Proof of (\ref{1.5})} The idea is the same as that  for the point to point case. Namely,  we get a lower bound for a simple path from 0 to $\partial \mathbb{B}_{r}$ and as before the large loops contribute negligibly which allows us to assert a similar lower bound for the true geodesic. 
 However, there is a crucial difference, namely we only consider the amount of time the path spends in  $\mathbb{B}_{\alpha r}$ for some $\alpha<1$ since we are unable to control its size close to the boundary.

Towards this we need to use a new event which is a modified version of $\mathcal{T}_{r}(z)$ defined in \eqref{Tr}. For any $r\in \mathbb{N}$, fixed $\alpha\in(0,1/3)$ (which will be chosen later) and $z\in \mathbb{B}_{\alpha^{2}r}$, we define 
    \begin{align*}
        \mathcal{T}_{r}^{\alpha}(z) := \{\text{there exists $\overline{\Gamma}\in \overline{\mathcal{L}}$ with $z\in \overline{\Gamma}$ such that } 0\overset{\overline{\mathcal{L}}\setminus \{\overline{\Gamma}\}, \mathbb{B}_{\alpha r}}{\longleftrightarrow} \overline{\Gamma} \circ \overline{\Gamma}\overset{\overline{\mathcal{L}}\setminus \{\overline{\Gamma}\}}{\longleftrightarrow} \partial \mathbb{B}_{r}\}.
    \end{align*} 
    Notice that in $\mathcal{T}_{r}^{\alpha}(z)$, the connection $0\leftrightarrow \overline{\Gamma}$ happens inside $\mathbb{B}_{\alpha r}$, while it happens in $\mathbb{B}_{r}$ on the event $\mathcal{T}_{r}(z)$.

    \begin{lemma}\label{lem 5.7}
        For $d >6$, there exists $c(d)>0$ and $\alpha_{0} \in (0,1/3)$ such that for any $r\in \mathbb{N}$ and $\alpha \in (0,\alpha_{0})$, 
        \begin{align*}
            \sum_{z\in \mathbb{B}_{\alpha^{2} r}}\widetilde{\mathbb{P}}(\mathcal{T}_{r}^{\alpha}(z)) \geq c\alpha^{4}.
        \end{align*}
    \end{lemma}
    
    The proof of this is similar to Lemma \ref{lem 5.1} with the two-point estimate \eqref{two point} replaced by its local counterpart Theorem \ref{lem 4.4}, which is deferred to Appendix, Section \ref{appendix.4}.

The next lemma bounds the contribution of large loops analogous to Lemma \ref{lem 5.4} (the proof can be found in the Appendix, Section \ref{appendix.2}).

\begin{lemma}\label{large loops on geo}
    For $d>6$, for any constant $\e>0$, there exists $K(\varepsilon)>0$ such that for any $r\in \mathbb{N}$,
        \begin{align}\label{5-59}
           \sum_{z\in \mathbb{B}_{1/3 r}} \sum_{\substack{\Gamma \ni z\\|\Gamma| \geq K}}\widetilde{\mathbb{P}}(\Gamma\in \mathcal{L})\widetilde{\mathbb{P}}(0\overset{\overline{\mathcal{L}}\setminus \{\overline{\Gamma}\},\mathbb{B}_{r}}{\longleftrightarrow} \overline{\Gamma} \circ \overline{\Gamma} \overset{\overline{\mathcal{L}}\setminus \{\overline{\Gamma}\}}{\longleftrightarrow} \partial \mathbb{B}_{r}) \leq \e.
        \end{align}
\end{lemma}
Note that the sum above is over $z \in \B_{1/3 r}$, which will survive in our purpose.

We need one more estimate for which we introduce a definition. \\

For a self-avoiding path $\ell=(0,x_{1},\cdots,x_{m})$ from 0 to $\partial \mathbb{B}_{r}$, let $\ell(\alpha)$ be the subpath of $\ell$ 
\begin{center}
    $\ell(\alpha):=(0,x_{1},\cdots,x_{t})$, where $t:= \min\{1\leq k\leq m: x_{k}\in \partial \mathbb{B}_{\alpha r}\}$.
\end{center}
    In fact, we will prove that for any $r\in \mathbb{N}$ and any sufficiently small constant $\alpha \in (0,1/3)$, 
    \begin{align}\label{initial part}
        \mathbb{E}[|\hat\ell(\alpha)\cap \mathbb{B}_{\alpha^{2}r}| \mid 0\leftrightarrow \partial \mathbb{B}_{r}] \geq c\alpha^{4} r^{2},
    \end{align}
    where $\hat\ell(\alpha)$ is the simple path minimizing $|\ell(\alpha) \cap \mathbb{B}_{\alpha^{2} r}|$ over all simple paths $\ell$ from 0 to $\partial \mathbb{B}_{r}$. Notice that even though there may be multiple minimizers, $|\hat\ell(\alpha) \cap \mathbb{B}_{\alpha^{2}r}|$ is always well-defined. If this is true, then by definition, 
    \begin{align*}
        \mathbb{E}[d^{\text{simple}}(0,\partial \mathbb{B}_{r})\mid 0\leftrightarrow \partial \mathbb{B}_{r}] \geq \mathbb{E}[|\hat \ell(\alpha) \cap \mathbb{B}_{\alpha^{2}r}| \mid 0\leftrightarrow \partial \mathbb{B}_{r}] \geq c\alpha^{4} r^{2}.
    \end{align*}

\begin{proposition}\label{prop 5.11}
    For $d>6$, there exists $c(d)>0$ such that for any $r$ and any sufficiently small $\alpha \in (0,1/3)$, 
    \begin{align}\label{5-66}
        \mathbb{E}[|\hat \ell(\alpha) \cap \mathbb{B}_{\alpha^{2}r}| \mathds{1}_{0\leftrightarrow \mathbb{B}_{r}}] \geq c \alpha^4.
    \end{align} 
\end{proposition}

The proof of this is similar to Proposition \ref{prop 5.2}, but since it involves a new case (see Figure \ref{overcount figure}) we provide all the details at the risk of being a bit repetitive.

\begin{figure}[h]
    \centering
    \includegraphics[scale=.24]{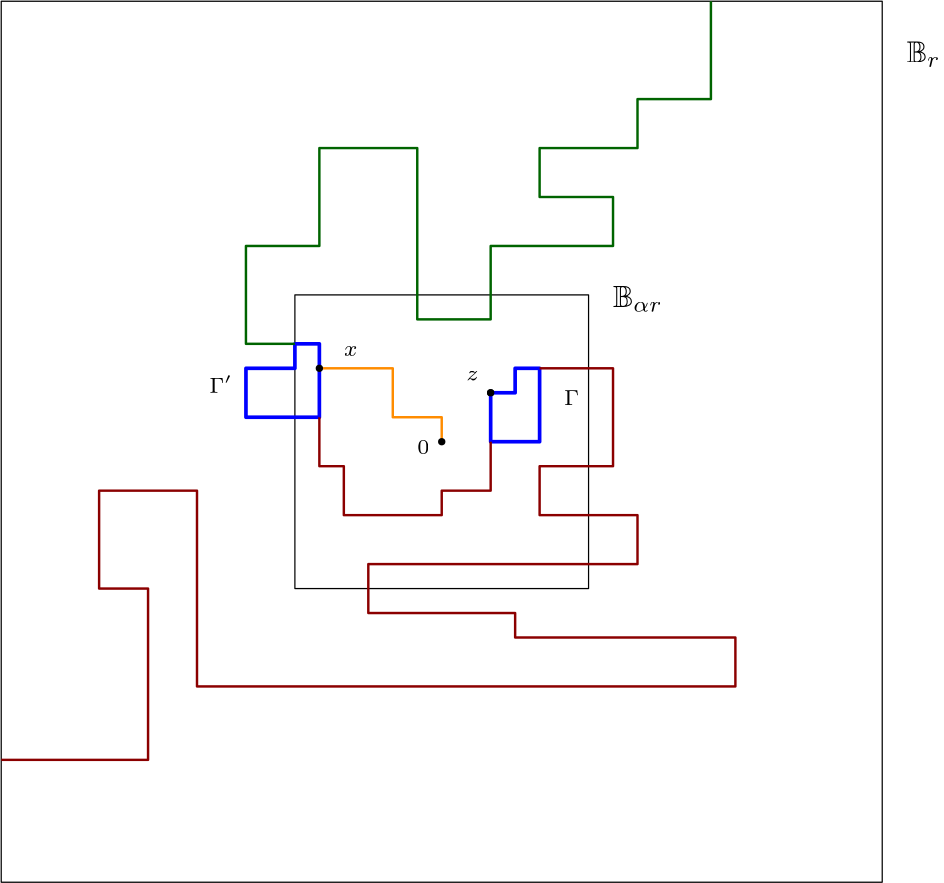}
    \caption{Illustration of an $x$-over-counted point $z$. $x$ and $z$ lie on two distinct simple paths connecting $0$ to $\partial \B_{r}$ with different end points on $\partial \B_{r}$. These two paths share a common segment (highlighted in orange) at the beginning until this orange segment enters $\Gamma_{2}$ from $x$. Beyond this, the path splits into two separate segments: a green one and a red one. The red path passes through $\Gamma$, where $z$ is located.
    }
    \label{overcount figure}
    \end{figure}

\begin{proof}
    For constants $K,K^{\prime}>0$ and $z\in \mathbb{B}_{\alpha^{2} r}$, we say $z$ is \textbf{$K$ over-counted}, if there exists a discrete loop $\Gamma \in \mathcal{L}$ with $z\in \Gamma$ and $|\Gamma|> K$ such that 
    \begin{align*}
        0\overset{\overline{\mathcal{L}}\setminus \{\overline{\Gamma}\},\mathbb{B}_{\alpha r}}{\longleftrightarrow} \overline{\Gamma} \circ \overline{\Gamma} \overset{\overline{\mathcal{L}}\setminus \{\overline{\Gamma}\}}{\longleftrightarrow} \partial \mathbb{B}_{r}.
    \end{align*}

    We say $z$ is \textbf{$(K,K^{\prime})$ over-counted}, if there exists discrete loops $\Gamma,\Gamma_{1},\Gamma_{2}\in \mathcal{L}$ and $u\in \mathbb{B}_{\alpha r}$, $v\in \mathbb{B}_{r}$ with $|z-u|\geq K^{\prime}$, $|z-v|\geq K^{\prime}$, $z\in \Gamma$, $|\Gamma|\leq K$, $u\sim \overline{\Gamma}_{1}$ and $v \sim \overline{\Gamma}_{2}$ such that 
    \begin{align*}
        0 \overset{\overline{\mathcal{L}}\setminus \{\overline{\Gamma},\overline{\Gamma}_{1},\overline{\Gamma}_{2}\},\mathbb{B}_{\alpha r}}{\longleftrightarrow} u \circ \overline{\Gamma}_{1}\overset{\overline{\mathcal{L}}\setminus \{\overline{\Gamma},\overline{\Gamma}_{1},\overline{\Gamma}_{2}\},\mathbb{B}_{\alpha r}}{\longleftrightarrow}\overline{\Gamma} \circ \overline{\Gamma} \overset{\overline{\mathcal{L}}\setminus \{\overline{\Gamma},\overline{\Gamma}_{1},\overline{\Gamma}_{2}\},\mathbb{B}_{r}}{\longleftrightarrow}\overline{\Gamma}_{2} \circ \overline{\Gamma}_{2}\overset{\overline{\mathcal{L}}\setminus \{\overline{\Gamma},\overline{\Gamma}_{1},\overline{\Gamma}_{2}\},\mathbb{B}_{r}}{\longleftrightarrow}\overline{\Gamma}_{1} \circ v\overset{\overline{\mathcal{L}}\setminus \{\overline{\Gamma},\overline{\Gamma}_{1},\overline{\Gamma}_{2}\}}{\longleftrightarrow} \partial \mathbb{B}_{r}.
    \end{align*}

    We say $z$ is \textbf{$x$-over-counted} for $x \in \mathbb{B}_{\alpha r}$ if there exists glued loops $\overline{\Gamma} , \overline{\Gamma}^{\prime}\in \mathcal{L}$  with $z \in \overline{\Gamma}$ and $x \in \overline{\Gamma}^{\prime}$ such that 
    \begin{align*}
        0\overset{\overline{\mathcal{L}}\setminus \{\overline{\Gamma} , \overline{\Gamma}^{\prime}\} }{\longleftrightarrow} x\circ \overline{\Gamma}^{\prime}\overset{\overline{\mathcal{L}}\setminus \{\overline{\Gamma} , \overline{\Gamma}^{\prime}\} }{\longleftrightarrow} \overline{\Gamma} \circ \overline{\Gamma}\overset{\overline{\mathcal{L}}\setminus \{\overline{\Gamma} , \overline{\Gamma}^{\prime}\} }{\longleftrightarrow} \partial \mathbb{B}_{r}\circ \overline{\Gamma}^{\prime}\overset{\overline{\mathcal{L}}\setminus \{\overline{\Gamma} , \overline{\Gamma}^{\prime}\} }{\longleftrightarrow}\partial \mathbb{B}_{r}.
    \end{align*}
    
    We say $z$ is not over-counted if it is neither $K$ over-counted, $(K, K^{\prime})$ over-counted nor $x$-over-counted for any $x \in \mathbb{B}_{\alpha r}$. Define 
    \begin{align*}
        N_{2}^{K,K^{\prime}} : = |\{z \in \mathbb{B}_{\alpha r}: \mathcal{T}_{r}(z) \text{ happens and $z$ is not over-counted}\}|.
    \end{align*}  
    For the similar reasons in the proof of lower bound in (\ref{1.4}) for GFF level-set, we know that for $K^{\prime}>3K$, 
    \begin{align*}
    N_{2}^{K,K^{\prime}} \leq C(K^{\prime})^{d}|\ell^{\alpha}_{0\rightarrow \partial \mathbb{B}_{r}}(\alpha)\cap \mathbb{B}_{\alpha^{2}r}|.
    \end{align*}

    Now we control the over-counting. By (\ref{5-59}), for any constant $\e>0$,there exists a large but fixed $K$, 
    \begin{align}\label{5-67}
        \sum_{z\in \mathbb{B}_{\alpha^{2} r}}\widetilde{\mathbb{P}}(z\text{ is $K$ over-counted}) &\leq \sum_{z\in \mathbb{B}_{\alpha^{2}r}}\sum_{\substack{\Gamma \ni z\\|\Gamma| \geq K}}\widetilde{\mathbb{P}}(\Gamma\in \mathcal{L})\widetilde{\mathbb{P}}(0\overset{\overline{\mathcal{L}}\setminus \{\overline{\Gamma}\},\mathbb{B}_{r}}{\longleftrightarrow} \overline{\Gamma} \circ \overline{\Gamma}\overset{\overline{\mathcal{L}}\setminus \{\overline{\Gamma}\}}{\longleftrightarrow} \partial \mathbb{B}_{r}) \leq \e.
    \end{align}

    Next, for the $(K,K^{\prime})$ over-counted, using a union bound and the BKR inequality, we have
    \begin{align*}
        &\sum_{z \in \mathbb{B}_{\alpha^{2}r}}\widetilde{\mathbb{P}}(z\text{ is $(K,K^{\prime})$ over-counted}) \nonumber\\ \leq &\sum_{z\in \mathbb{B}_{\alpha^{2} r}}\sum_{\substack{\Gamma \ni z\\ |\Gamma|\leq K}}\widetilde{\mathbb{P}}(\Gamma \in \mathcal{L})\sum_{\substack{u\in \mathbb{B}_{\alpha r}\\ |z-u|\geq K^{\prime}}}\sum_{\substack{v\in \mathbb{B}_{r}\\ |z-v|\geq K^{\prime}}}\sum_{s_{1},s_{2}\in \mathbb{B}_{r}}\sum_{\Gamma_{1}\sim u,s_{1},s_{2}}\sum_{w_{1},w_{2} \in \mathbb{B}_{r}}\sum_{\Gamma_{2} \sim v,w_{1},w_{2}}\widetilde{\mathbb{P}}(\Gamma_{1}\in \mathcal{L})\widetilde{\mathbb{P}}(\Gamma_{2}\in \mathcal{L})
    \end{align*}
    \begin{align}\label{5-68}
        \cdot \widetilde{\mathbb{P}}(0\leftrightarrow u)\widetilde{\mathbb{P}}(s_{1}\overset{\overline{\mathcal{L}}\setminus \{\overline{\Gamma}\}}{\longleftrightarrow} \overline{\Gamma})\widetilde{\mathbb{P}}(s_{2}\leftrightarrow w_{2})\widetilde{\mathbb{P}}(\overline{\Gamma}\overset{\overline{\mathcal{L}}\setminus \{\overline{\Gamma}\}}{\longleftrightarrow}  w_{1})\widetilde{\mathbb{P}}(v\leftrightarrow \partial \mathbb{B}_{r}).
    \end{align}
    Since for $\Gamma \ni z$ with $|\Gamma|\leq K$ and any $s_{3},w_{3}\sim \Gamma$, 
    \begin{align*}
        |s_{3}-z|\leq K+1,\; |w_{3}-z|\leq K+1.
    \end{align*}
    Then the two point function estimate \eqref{two point} implies that 
    \begin{align*}
        \widetilde{\mathbb{P}}(s_{3}\leftrightarrow s_{1}) \leq C\widetilde{\mathbb{P}}(z\leftrightarrow s_{1}),\; \widetilde{\mathbb{P}}(w_{3}\leftrightarrow w_{1}) \leq C\widetilde{\mathbb{P}}(z\leftrightarrow w_{2}).
    \end{align*}
    Thus, by applying a union bound, we get
    \begin{align}\label{5-69}
        &\sum_{\substack{\Gamma\ni z\\|\Gamma| \leq K}} \widetilde{\mathbb{P}}(\Gamma \in \mathcal{L})\widetilde{\mathbb{P}}(\overline{\Gamma} \overset{\overline{\mathcal{L}}\setminus \{\overline{\Gamma}\}}{\longleftrightarrow}  s_{1})\widetilde{\mathbb{P}}(\overline{\Gamma}\overset{\overline{\mathcal{L}}\setminus \{\overline{\Gamma}\}}{\longleftrightarrow}  w_{1}) \leq \sum_{\substack{\Gamma\ni z\\|\Gamma| \leq K}}\sum_{s_{3},w_{3} \sim \Gamma}\widetilde{\mathbb{P}}(\Gamma \in \mathcal{L})\widetilde{\mathbb{P}}(s_{3}\leftrightarrow s_{1})\widetilde{\mathbb{P}}(w_{3}\leftrightarrow w_{1}) \nonumber\\ &\leq C \widetilde{\mathbb{P}}(z\leftrightarrow s_{1})\widetilde{\mathbb{P}}(z\leftrightarrow w_{1})\sum_{\substack{\Gamma\ni z\\|\Gamma| \leq K}}|\Gamma|^{2}\widetilde{\mathbb{P}}(\Gamma \in \mathcal{L}) \overset{(\text{Lemma \ref{loop one point}})}{\leq} CK^{3-d/2}\widetilde{\mathbb{P}}(z\leftrightarrow s_{1})\widetilde{\mathbb{P}}(z\leftrightarrow w_{1}).
    \end{align}
    Applying Lemma (\ref{5:53}), we obtain
    \begin{align}\label{5-70}
        \sum_{s_{1},s_{2}\in \mathbb{B}_{r}}\sum_{\Gamma_{1}\sim u,s_{1},s_{2}}\widetilde{\mathbb{P}}(\Gamma_{1}\in \mathcal{L})\widetilde{\mathbb{P}}(s_{1}\leftrightarrow z)\widetilde{\mathbb{P}}(s_{2}\leftrightarrow w_{2}) \leq C|u|^{2-d}|u-w_{2}|^{2-d}.
    \end{align} 
    Furthermore, by Lemma \ref{lem 5.8}, we have
    \begin{align}\label{5-71}
        \sum_{v,w_{1}\in \mathbb{B}_{r}}\sum_{\Gamma_{2}\sim v,w_{1},w_{2}}\widetilde{\mathbb{P}}(\Gamma_{2}\in \mathcal{L})\widetilde{\mathbb{P}}(z \leftrightarrow w_{1})\widetilde{\mathbb{P}}(v\leftrightarrow \partial \mathbb{B}_{r})\widetilde{\mathbb{P}}(u\leftrightarrow v) \leq C |z-w_{2}|^{2-d} d^{\text{ext}}(w_{2}, \partial \mathbb{B}_{r})^{-2}.
    \end{align}
    Substituting (\ref{5-69})-(\ref{5-71}) into (\ref{5-68}) we get 
    \begin{align}\label{5-72}
        &\sum_{z \in \mathbb{B}_{\alpha^{2}r}}\widetilde{\mathbb{P}}(z\text{ is $(K,K^{\prime})$ over-counted}) \nonumber\\ &\leq C\sum_{z \in \mathbb{B}_{\alpha^{2}r}}\sum_{\substack{u\in \mathbb{B}_{\alpha r}\\ |z-u|\geq K^{\prime}}}\sum_{w_{2}\in \mathbb{B}_{r}}|u|^{2-d}|u-z|^{2-d}|u-w_{2}|^{2-d}|w_{2}-z|^{2-d}d^{\text{ext}}(w_{2}, \partial \mathbb{B}_{r})^{-2} \nonumber\\
        &\overset{(\ref{bond computation 2.5})}{\leq} C \alpha^{2}\e + C\alpha^{6}.
    \end{align}

    The last thing is to control the $x$ over-counting for $x\in \mathbb{B}_{\alpha r}$. By a union bound 
    \begin{align*}
        &\widetilde{\mathbb{P}}(z \text{ is $x$-over-counted}) \\ 
        &\leq \sum_{\Gamma \sim z} \sum_{u\in \mathbb{B}_{\alpha r}}\sum_{v \in \mathbb{B}_{r}}\sum_{\substack{\Gamma^{\prime}\ni x\\ \Gamma^{\prime} \sim u,v}}\widetilde{\mathbb{P}}(\Gamma\in \mathcal{L})\widetilde{\mathbb{P}}(\Gamma^{\prime}\in \mathcal{L})\widetilde{\mathbb{P}}(0\leftrightarrow x)\widetilde{\mathbb{P}}(u\overset{\overline{\mathcal{L}}\setminus \{\overline{\Gamma}\}}{\longleftrightarrow} \overline{\Gamma})\widetilde{\mathbb{P}}(\overline{\Gamma}\overset{\overline{\mathcal{L}}\setminus \{\overline{\Gamma}\}}{\longleftrightarrow} \partial \mathbb{B}_{r})\widetilde{\mathbb{P}}(v\leftrightarrow \partial \mathbb{B}_{r}).
    \end{align*}
   By Lemma \ref{lem 5.8}, for any $u\in \mathbb{B}_{\alpha r}$ and $z\in \mathbb{B}_{\alpha^{2}r}$,
   \begin{align*}
    \sum_{\Gamma\in z}\widetilde{\mathbb{P}}(u\overset{\overline{\mathcal{L}}\setminus \{\overline{\Gamma}\}}{\longleftrightarrow} \overline{\Gamma})\widetilde{\mathbb{P}}(\overline{\Gamma}\overset{\overline{\mathcal{L}}\setminus \{\overline{\Gamma}\}}{\longleftrightarrow} \partial \mathbb{B}_{r}) \leq C |u-z|^{2-d}d^{\text{ext}}(z,\partial \mathbb{B}_{r})^{-2}\leq Cr^{-2}|u-z|^{2-d}.
   \end{align*}
   The last inequality is because $d^{\text{ext}}(z,\partial \mathbb{B}_{r}) \geq (1-\alpha^{2})r$.
   Thus,
   \begin{align*}
    &\sum_{z\in \mathbb{B}_{\alpha^{2}r}}\widetilde{\mathbb{P}}(z \text{ is $x$-over-counted})   \\ \leq &C r^{-2}\sum_{u\in \mathbb{B}_{\alpha r}}\sum_{v \in \mathbb{B}_{r}}\sum_{\substack{\Gamma^{\prime}\ni x\\ \Gamma^{\prime} \sim u,v}}\widetilde{\mathbb{P}}(\Gamma^{\prime} \in \mathcal{L})\widetilde{\mathbb{P}}(0\leftrightarrow x)\widetilde{\mathbb{P}}(v\leftrightarrow \partial \mathbb{B}_{r})\sum_{z \in \mathbb{B}_{\alpha^{2}r}}|u-z|^{2-d}  \\ \overset{(\ref{2.4})}{\leq}&C r^{-2} \cdot \alpha^{4}r^{2}\sum_{u\in \mathbb{B}_{\alpha r}}\sum_{v \in \mathbb{B}_{r}}\sum_{\substack{\Gamma^{\prime}\ni x\\ \Gamma^{\prime} \sim u,v}}\widetilde{\mathbb{P}}(\Gamma^{\prime} \in \mathcal{L})\widetilde{\mathbb{P}}(0\leftrightarrow x)\widetilde{\mathbb{P}}(v\leftrightarrow \partial \mathbb{B}_{r}).
   \end{align*}
   By Lemma \ref{loop three points}, we get
   \begin{align*}
    \sum_{\substack{\Gamma^{\prime}\ni x\\ \Gamma^{\prime} \sim u,v}}\widetilde{\mathbb{P}}(\Gamma^{\prime} \in \mathcal{L}) \leq C |u-x|^{2-d}|x-v|^{2-d}|v-u|^{2-d}.
   \end{align*}
   Finally summing over $x\in \mathbb{B}_{\alpha r}$ we get 
   \begin{align*}
    &\sum_{x\in \mathbb{B}_{\alpha r}}\sum_{z\in \mathbb{B}_{\alpha^{2}r}}\widetilde{\mathbb{P}}(z \text{ is $x$-over-counted}) \\ \leq 
    &C \alpha^{4}\sum_{x\in \mathbb{B}_{\alpha r}}\sum_{u\in \mathbb{B}_{\alpha r}}\sum_{v\in \mathbb{B}_{r}}|u-x|^{2-d}|x-v|^{2-d}|v-u|^{2-d}|x|^{2-d}d^{\text{ext}}(v,\partial \mathbb{B}_{r})^{-2}.
   \end{align*}
   By a computation similar to (\ref{5-72}), we have the following estimate:
   \begin{align*}
    \sum_{x\in \mathbb{B}_{\alpha r}}\sum_{u\in \mathbb{B}_{\alpha r}}\sum_{v\in \mathbb{B}_{r}}|u-x|^{2-d}|x-v|^{2-d}|v-u|^{2-d}|x|^{2-d}d^{\text{ext}}(v,\partial \mathbb{B}_{r})^{-2} \leq C \alpha^{2}.
   \end{align*}
   Thus, it follows that
   \begin{align}\label{5-73}
    \sum_{x\in \mathbb{B}_{\alpha r}}\sum_{z\in \mathbb{B}_{\alpha^{2}r}}\widetilde{\mathbb{P}}(z \text{ is $x$-over-counted}) \leq C\alpha^{4}\cdot \alpha^{2} \leq C\alpha^{6}.
   \end{align}
   Combining (\ref{5-67}), (\ref{5-72}) and (\ref{5-73}), we get
   \begin{align}\label{5-74}
    \sum_{z\in \mathbb{B}_{\alpha^{2}r}}\widetilde{\mathbb{P}}(\mathcal{T}_{r}^{\alpha}(z)) &\leq \widetilde{\mathbb{E}}[N^{K,K^{\prime}}_{2}] + \sum_{z\in \mathbb{B}_{\alpha^{2} r}}\widetilde{\mathbb{P}}(z\text{ is $K$ over-counted}) \nonumber\\ &+\sum_{z \in \mathbb{B}_{\alpha^{2}r}}\widetilde{\mathbb{P}}(z\text{ is $(K,K^{\prime})$ over-counted})+\sum_{x\in \mathbb{B}_{\alpha r}}\sum_{z\in \mathbb{B}_{\alpha^{2}r}}\widetilde{\mathbb{P}}(z \text{ is $x$-over-counted}) \nonumber\\ &\leq \widetilde{\mathbb{E}}[N^{K,K^{\prime}}_{2}] + \e + C \alpha^{2}\e +C\alpha^{6}+ C \alpha^{6}.
   \end{align}
   By Lemma \ref{lem 5.7}, for sufficiently small $\alpha$, we have the lower bound.
   \begin{align*}
    \sum_{z\in \mathbb{B}_{\alpha^{2}r}}\widetilde{\mathbb{P}}(\mathcal{T}^{\alpha}_{r}(z)) \geq C \alpha^{4}.
   \end{align*}
   Plugging this into (\ref{5-74}), we obtain 
   \begin{align*}
    \widetilde{\mathbb{E}}[N^{K,K^{\prime}}_{2}] \geq c \alpha^{4} - 2\e - C \alpha^{6}.
   \end{align*}
   Now, by choosing $\e = \alpha^{6}$ and selecting $\alpha$ small enough such that $c\alpha^{4} - (2+C)\alpha^{6} \geq \frac{1}{2} \alpha^{4}$, we have 
   \begin{align*}
    C(K^{\prime})^{d}|\widetilde{\mathbb{E}}[\ell^{\alpha}_{0\rightarrow \partial \mathbb{B}_{r}}(\alpha)\cap \mathbb{B}_{\alpha^{2}r}|]\geq
    \widetilde{\mathbb{E}}[N^{K,K^{\prime}}_{2}] \geq \frac{1}{2}\alpha^{4}.
   \end{align*}
   Here $K^{\prime}$ and $\alpha$ are fixed constants, thus we finish the proof.
\end{proof}

\begin{proof}[Proof of (\ref{1.5})]Given the above input, the proof of \eqref{1.5} now follows along the same lines as the proof of \eqref{1.4} and is omitted. 
\end{proof}

\section{Intrinsic one-arm exponent}\label{one-arm}

In this section we prove Theorem \ref{thm 1.4}. We start with the upper bound which doesn't directly rely on our other results. However, the novel idea here is to compare the intrinsic and extrinsic metrics, a theme prevalent in our discussions throughout the paper. 

\subsection{Proof of the upper bound in Theorem \ref{thm 1.4}}

While some parts of the argument will essentially be the same as  the proof of \cite[Proposition 5.1]{ganguly2024ant}, to avoid repetition we will keep them brief, instead focusing on emphasizing the new observation that allows us to prove the bound for all $d>6$ as opposed to \cite{ganguly2024ant} where the bound was established only for $d>20.$

Before diving into our proof, we first review the argument from \cite{kozma2009alexander} in the bond case, and the new observations made in \cite{ganguly2024ant} which works for $d>20$ in the loop case. In the bond case, using a prior bound of Aizenman-Barsky \cite{barsky1991percolation}, one has $\mathbb{P}(|\mathcal{C}(0)| \geq \e r^{2}) \lesssim \frac{1}{r}$ (where recall that $\cC(0) $ is the connected component of $0$). Then one splits the event into two cases. First, if $|\mathcal{C}(0)| \geq \e r^{2}$, then the Aizenman-Barsky bound is in play. If $|\mathcal{C}(0)| \leq \e r^{2}$, by pigeonhole principle there is some $i\in [r/2,2]$ such that the number of points at (chemical) distance exactly $i$ is less than $2\e r$. Then one can reveal the cluster of the origin till the first such $i$, say $\tau$. Then by an inductive argument one obtains 
\begin{align}\label{bond recursion}
    \mathbb{P}(\partial B(0,2r)\neq \emptyset \mid B(0,\tau)) \leq 2\e r p_{r},
\end{align}
where we use the notation $p_{r}:=\mathbb{P}(\partial B(0,r) \neq \emptyset)$. This leads to a recursion of the form $p_{2r} \lesssim 2\e rp_{r}^{2} + \frac{1}{r}$ which implies that $p_{r}\lesssim \frac{1}{r}$. 


{Such a stopping domain argument fails in the loop model in the presence of potentially long loops. To do this an averaging argument was employed in \cite{ganguly2024ant} to bound the probability that at least a good fraction of the spheres (i.e. points at chemical distance exactly $i$ for some $i\in [r/2,r]$) intersect large loops.}
A crucial input in that argument was the volume estimate $\widetilde{\mathbb{E}}[|B(0,r)|]\lesssim r$ which was obtained by an \textsf{ITE} argument for $d>20$ in \cite{ganguly2024ant}.

{Unfortunately, this argument breaks down for $6<d\leq 20$ because of the absence of the volume estimate. The novel idea we introduce is to compare the intrinsic one-arm event with extrinsic one-arm event. Namely, we aim to use the quadratic relating $\{\partial B(0,r) \neq \emptyset\}$ to $\{0\leftrightarrow \partial \mathbb{B}_{\sqrt{r}}\} $. More precisely, we split $\{\partial B(0,r) \neq \emptyset\}$ into two cases: }
\begin{itemize}
\item If $B(0,r) \cap \mathbb{B}_{\sqrt{r}}^{c} \neq \emptyset$, then obviously $0\leftrightarrow \partial \mathbb{B}_{\sqrt{r}}$ occurs and we can bound this case using $\pi_{1}(\sqrt{r})\lesssim \frac{1}{r}$. 
\item {If $B(0,r) \subseteq \mathbb{B}_{\sqrt{r}}$, then we can simply replace all the intrinsic connections $\overset{r}{\longleftrightarrow}$ by $\overset{\mathbb{B}_{r}}{\longleftrightarrow}$. Thus, we can avoid using the estimate for the intrinsic ball and simply apply $\widetilde{\mathbb{E}}[|\mathcal{C}(0)\cap\mathbb{B}_{\sqrt{r}}|]\lesssim r$, which a straightforward consequence of the two point function \eqref{two point}.}
\end{itemize}

Now we are in a position to prove the upper bound in Theorem \ref{thm 1.4}. First we introduce some useful notations from \cite[Section 5]{ganguly2024ant}.
For any collection of loops $\widetilde{\mathscr{S}}$ on $\widetilde{\mathbb{Z}}^{d}$ and $w,v \in \mathbb{Z}^{d}$, let $d(w,v;\widetilde{\mathscr{S}})$ be the length of the shortest path from $u$ to $v$, only using loops in $\widetilde{\mathscr{S}}$. Define the $r$-ball and $r$-sphere w.r.t. $\widetilde{\mathscr{S}}$ as follows: For $v \in \mathbb{Z}^{d}$ and $r \in \mathbb{N}$, 
\begin{align*}
\widetilde{B}(v,r;\widetilde{\mathscr{S}})= \{ w\in \mathbb{Z}^{d}: d(w,v;\widetilde{\mathscr{S}}) \leq r\}, \quad  \partial B(v,r;\widetilde{\mathscr{S}})= \{ w\in \mathbb{Z}^d: d(w,v;\widetilde{\mathscr{S}}) =r\}.
\end{align*}

The set of loops $\widetilde{\mathscr{S}}$ we will need to encounter will all satisfy the following properties.

\begin{assumption}[Assumption 1 in \cite{ganguly2024ant}] \label{topological assumption}
    The collection $\widetilde{\mathscr{S}}$ will contain fundamental, point loops and all possible glued loops associated with all but finitely many edges, i.e. the entire support of the glued edge loop for those edges. For any of the edges whose glued edges are not included, there is at least one of its incident points such that no point loop corresponding to that point, or fundamental loop passing through that point is included.
\end{assumption}
For our applications, it will suffice to restrict to classes of loops with the above property. In fact, the entire set of loops in consideration satisfies the above property. \\

Let $\widetilde{\mathcal{L}}_{\widetilde{\mathscr{S}}}$ be the loop soup $\widetilde{\mathcal{L}}$ restricted to $\widetilde{\mathscr{S}}$. Set 
\begin{align*}
\Lambda\left(r\right):= \sup_{v,\widetilde{\mathscr{S}}}\widetilde{\mathbb{P}}(\partial B(v,r;\widetilde{\mathcal{L}}_{\widetilde{\mathscr{S}}}) \neq \emptyset),
\end{align*}
where the supremum is taken over all vertices $v\in \mathbb{Z}^{d}$ and (deterministic) collections $\widetilde{\mathscr{S}}$ of loops satisfying Assumption \ref{topological assumption}. We will prove the following proposition.

\begin{proposition}\label{one arm 2}
For $d>6$, there exists $C\left(d\right)>0$ such that for any $r \in \mathbb{N}$
\begin{align*}
    \Lambda\left(r\right) \leq C r^{-1}.
\end{align*}
\end{proposition}
The upper bound in (\ref{1.4}) can be directly deduced from Proposition \ref{one arm 2} by taking $\widetilde{\mathscr{S}}$ to be the collection of all loops and setting $v=0$.

\begin{proof}[Proof of Proposition \ref{one arm 2}]
We begin with a bit of a set up mimicking \cite[Section 5.2]{ganguly2024ant}.
By translation invariance, 
\begin{align*}
\widetilde{\mathbb{P}}(\partial B(v,r;\widetilde{\mathcal{L}}_{\widetilde{\mathscr{S}}}) \neq \emptyset) = \widetilde{\mathbb{P}}(\partial B(0,r;\widetilde{\mathcal{L}}_{\widetilde{\mathscr{S}}-v}) \neq \emptyset),
\end{align*}
where $\widetilde{\mathscr{S}}-v$ denotes the collection of loops in $\widetilde{\mathscr{S}}$ shifted by $v$. It's easy to check that $\widetilde{\mathscr{S}}-v$ also satisfies Assumption \ref{topological assumption}. Hence, it suffices to bound 
\begin{align*}
\Lambda(r) = \sup_{\widetilde{\mathscr{S}}} \widetilde{\mathbb{P}}(\partial B(0,r;\widetilde{\mathcal{L}}_{\widetilde{\mathscr{S}}}) \neq \emptyset).
\end{align*}
We aim to prove a recursive inequality for $\Lambda(\cdot)$. Precisely, we consider $r= 3^{k}$ for $k\in \mathbb{N}$ and will establish a relation between $\Lambda(3^{k-1})$ and $\Lambda(3^{k})$. 

For $i \in \mathbb{N}$, let $\widetilde{\mathcal{L}}_{\widetilde{\mathscr{S}}}^{\leq i}$ (resp. $\widetilde{\mathcal{L}}_{\widetilde{\mathscr{S}}}^{i}$) be the collection of loops in $\widetilde{\mathcal{L}}_{\widetilde{\mathscr{S}}}$ which intersect $\widetilde{B}(0,i;\widetilde{\mathcal{L}}_{\widetilde{\mathscr{S}}})$ (resp. $\partial B(0,i;\widetilde{\mathcal{L}}_{\widetilde{\mathscr{S}}})$). We also define $\mathcal{L}_{\widetilde{\mathscr{S}}}^{\leq i}$ (resp. $\mathcal{L}_{\widetilde{\mathscr{S}}}^{i}$) as the corresponding discrete loop soup for $\widetilde{\mathcal{L}}_{\widetilde{\mathscr{S}}}^{\leq i}$ ((resp. $\widetilde{\mathcal{L}}_{\widetilde{\mathscr{S}}}^{i}$)).
In addition, we say that an edge $e$ is within $\widetilde{\mathscr{S}}$-level $i$ if it intersects some loop in $\widetilde{\mathcal{L}}_{\widetilde{\mathscr{S}}}^{\leq i}$.
Let $\widetilde{\mathcal{C}}_{\widetilde{\mathscr{S}}}(0)$ be the cluster containing the origin with respect to $\widetilde{\mathcal{L}}_{\widetilde{\mathscr{S}}}$. For any $\e >0$, define 
\begin{align*}
    I := \{i \in \big[3^{k-1},3^{k}/2\big]: |\partial B(0,i;\widetilde{\mathcal{L}}_{\widetilde{\mathscr{S}}})| \leq 12 \e \cdot 3^{k}\}.
\end{align*}
If $|\widetilde{\mathcal{C}}_{\widetilde{\mathscr{S}}}(0)| \leq \e 9^{k}$, as 
\begin{align*}
|\widetilde{\mathcal{C}}_{\widetilde{\mathscr{S}}}(0)| \geq \sum_{i= 3^{k-1}}^{3^{k}/2} |\partial B(0,i;\widetilde{\mathcal{L}}_{\widetilde{\mathscr{S}}})|,
\end{align*}
by averaging the right-hand side above we have
\begin{align} \label{5.2}
    |I|\geq 3^{k-1}/4.
\end{align}

Under the event $\{\partial B(0,3^{k};\widetilde{\mathcal{L}}_{\widetilde{\mathscr{S}}}) \neq \emptyset\}$, let $\ell= (e_{1},\cdots,e_{3^{k}})$ be an $\widetilde{\mathcal{L}}_{\widetilde{\mathscr{S}}}$-geodesic (i.e. a shortest path in $\widetilde{\mathcal{L}}_{\widetilde{\mathscr{S}}}$) from the origin to $\partial B(0,3^{k};\widetilde{\mathcal{L}}_{\widetilde{\mathscr{S}}})$. For each level $i\in\big[3^{k-1},3^{k}/2\big]$, set 
\begin{align*}
   J(i):= \max \{1 \leq m\leq 3^{k}: e_{m}\text{ is within $\widetilde{\mathscr{S}}$-level $i$}\}. 
\end{align*}
Then define 
\begin{align*}
T_{i} := \{\text{Loops in $\widetilde{\mathcal{L}}^{i}_{\widetilde{\mathscr{S}}}$ intersecting the edge $e_{J(i)}$}\}.
\end{align*}
As discussed in \cite{ganguly2024ant}, $T_{i}$ is not empty. A remark here is that the recursive inequality obtained by inductive step will depend on the size of loops in $T_{i}$, and we will bound the contribution of levels $i$ such that the loops in $T_{i}$ are small or large separately.
For a large constant $L>0$ which will be chosen later, partition the set $I$ into 
\begin{align*}
    &I_{1}:= \{i \in I : |\widetilde{\Gamma}|<L , \forall \widetilde{\Gamma} \in T_{i}\},\\
    &I_{2}:=t\{i \in I : \exists \widetilde{\Gamma} \in T_{i}\text{ such that } |\widetilde{\Gamma}| \geq L\},
\end{align*}
where $|\widetilde{\Gamma}| : = |\textsf{Trace}(\widetilde{\Gamma})|$. In other words, $I_{1}$
 is the collection of levels in $I$ such that \textbf{no} loop in $\widetilde{\mathcal{L}}^{i}_{\widetilde{\mathscr{S}}}$ intersecting $e_{J(i)}$ is large. We need to define one final event stating that no loop is too large.
 \begin{align*}
    \mathcal{G}_{\widetilde{\mathscr{S}}}:= \{\text{Every loop in $\widetilde{\mathcal{L}}_{\widetilde{\mathscr{S}}}^{\leq 3^{k}}$ has a size at most $\frac{3^{k-1}}{2}-1$}\}.
 \end{align*}
Now by the discussion above, for any $\widetilde{\mathscr{S}}$, when $\{\partial B(0,3^{k};\widetilde{\mathcal{L}}_{\widetilde{\mathscr{S}}}) \neq \emptyset\}$ occurs, one of the following events must happen: 
\begin{itemize}
\item $\mathsf{C}_{1}$: $|\widetilde{\mathcal{C}}_{\widetilde{\mathscr{S}}}(0)| \geq \e 9^{k}$;
\item $\mathsf{C}_{2}$: $\widetilde{B}(0,3^{k};\widetilde{\mathcal{L}}_{\widetilde{\mathscr{S}}}) \cap \mathbb{B}_{  3^{k/2}}^{c} \neq \emptyset$;
\item $\mathsf{C}_{3}$: $\{|I_{1}| \geq 3^{k-1}/8\} \cap \{\partial B(0,3^{k};\widetilde{\mathcal{L}}_{\widetilde{\mathscr{S}}}) \neq \emptyset\} \cap \mathcal{G}_{\widetilde{\mathscr{S}}}$;
\item $\mathsf{C}_{4}$: $\{|I_{2}| \geq 3^{k-1}/8\} \cap \{\partial B(0,3^{k};\widetilde{\mathcal{L}}_{\widetilde{\mathscr{S}}}) \neq \emptyset\} \cap \mathcal{G}_{\widetilde{\mathscr{S}}}\cap \{\mathsf{C}_{5}$: $\widetilde{B}(0,3^{k};\widetilde{\mathcal{L}}_{\widetilde{\mathscr{S}}}) \subseteq \mathbb{B}_{ 3^{k/2}}\}$;
\item $\mathsf{C}_{5}$: $\mathcal{G}^{c}_{\widetilde{\mathscr{S}}}\cap \{\widetilde{B}(0,3^{k};\widetilde{\mathcal{L}}_{\widetilde{\mathscr{S}}}) \subseteq \mathbb{B}_{  3^{k/2}}\}$,
\end{itemize}
where $\e$ will be chosen later. To get a recursive inequality, we need to bound the probabilities of these five events.
The arguments to bound $\P(\mathsf{C}_{1})$, $\P(\mathsf{C}_{2})$ and $\P(\mathsf{C}_{5})$ are short and we include them.  The argument to bound $\P(\mathsf{C}_{3})$ is long but is exactly the same as from \cite{ganguly2024ant} which will be omitted. Bounding $\P(\mathsf{C}_{4})$ will involve the new idea of the extrinsic to intrinsic comparison and we will provide the argument in detail.   

\textbf{For $\mathsf{C}_{1}$,} we use a crucial bound on the cluster size. This bound for bond percolation was proved in \cite{barsky1991percolation}. Here we need the version for $\widetilde{\mathcal{L}}$ proved in \cite{cai2023one}. 

\begin{lemma}[Proposition 6.6 in \cite{cai2023one}] \label{cluster upper}
For $d>6$, there exists $C(d)>0$ such that for any $r \in \mathbb{N}$,
\begin{align*}
\widetilde{\mathbb{P}}(|\mathcal{C}(0)|\geq M) \leq \frac{C}{r^{1/2}}.
\end{align*}
\end{lemma}

Notice that the event $\{|\widetilde{\mathcal{C}}_{\widetilde{\mathscr{S}}}(0)| \geq M\}$ is monotone in the collection of loops $\widetilde{\mathscr{S}}$ (with respect to the inclusion relation), by Proposition \ref{cluster upper} the probability of $\mathsf{C}_{1}$ can be bounded by 
\begin{align}\label{C_{1}}
\widetilde{\mathbb{P}}(\mathsf{C}_{1}) \leq \widetilde{\mathbb{P}}(|\mathcal{C}(0)| \geq \e 9^{k}) \leq \frac{C}{\sqrt{\e}3^{k}}.
\end{align}

\textbf{For $\mathsf{C}_{2}$}, we notice that $\mathsf{C}_{2}$ implies that $\{0\leftrightarrow \partial \mathbb{B}_{  3^{k/2}}\}$ occurs, then by Lemma \ref{lem 2.3} we have 
\begin{align}\label{C_{2}}
\widetilde{\mathbb{P}}(\mathsf{C}_{2}) \leq 
\widetilde{\mathbb{P}}(0\leftrightarrow \partial \mathbb{B}_{  3^{k/2}}) \leq \frac{C}{  3^{k}}.
\end{align}
Introducing $\mathsf{C}_{2}$ is a novel idea in the proof. In \cite{ganguly2024ant, kozma2009alexander}, the upper bound of the intrinsic one-arm exponent relied on the bound for the volume of the intrinsic ball, which we don't have in our case. To fix that, we introduce $\mathsf{C}_{2}$ to connect the extrinsic metric and intrinsic metric, so that we can apply the estimates for the extrinsic metric to get the intrinsic one arm exponent.

Before establishing bounds for $\widetilde{\P}(\mathsf{C}_{3})$-$\widetilde{\P}(\mathsf{C}_{5})$, we need a few useful inputs. The first one is a simple fact that facilitates the application of the averaging argument for $\mathsf{C}_{4}$. For any events $A_{1},\cdots, A_{m},B$ and $n,m\in \mathbb{N}$ with $n \leq m$, 
\begin{align} \label{5.5}
\mathbb{P}(\{\text{At least $n$ events among the events $A_1,\cdots,A_m$ occur}\} \cap B) \leq \frac{1}{n}\sum_{i=1}^{m}\mathbb{P}(A_{i}\cap B).
\end{align}

The following lemma from \cite{ganguly2024ant} characterizes the measure induced on the loops by conditioning. We omit the proof and interested readers are referred to \cite[Lemma 5.9 and Lemma 5.10]{ganguly2024ant}.

\begin{lemma}[Lemma 5.9 and Lemma 5.10 in \cite{ganguly2024ant}] \label{lem 6.5}
    Let $\widetilde{\mathscr{S}}$ be a collection of loops satisfying Assumption \ref{topological assumption}. For any $i\in \N$ and {a collection $\widetilde{\mathscr{A}}$ of loops in the support of $\widetilde{\mathcal{L}}_{\widetilde{\mathscr{S}}}^{\le i}$,  let $\widetilde{\mathscr{A}}^{(i)}$ be the collection of  loops which intersect the closure of $  \widetilde B(0,i;\widetilde{\mathscr{A}})$} (as a subset of $\widetilde{\mathbb{Z}}^{d}$). Then, for any $v\in \Z^d$ and $j >i$,
\begin{align} \label{57}
  \widetilde{\mathbb{P}}  ( \partial B(v,j;    \widetilde{\mathcal{L}}_{\widetilde{\mathscr{S}}} \setminus 
 \widetilde{\mathcal{L}}_{\widetilde{\mathscr{S}}}^{\le i}     ) \neq \emptyset  \mid \widetilde{\mathcal{L}}_{\widetilde{\mathscr{S}}}^{\le i} = \widetilde{\mathscr{A}}) &=\widetilde{\mathbb{P}}  ( \partial B(v,j;         \widetilde{\mathcal{L}}_{\widetilde{\mathscr{S}}}  \setminus {\widetilde{\mathscr{A}}^{(i)}} ) \neq \emptyset ) \leq \Lambda(j).
\end{align}
\end{lemma}

The last result we need is the following proposition, which excludes the existence of a large discrete loop. This proposition is a Euclidean version of \cite[Lemma 5.5]{ganguly2024ant}, where we replace the intrinsic connection event $\{0\overset{r}{\leftrightarrow} x\}$ by the extrinsic connection event $\{0 \overset{\mathbb{B}_{r}}{\longleftrightarrow} x\}$. The proof is deferred to the end of this section.
\begin{proposition}\label{large loop 2}
For $d>6$, there exists $C(d)>0$ such that for any $r,L\in \mathbb{N}$ and $0\leq i < (d-4)/2$, 
\begin{align}\label{5.9}
    \widetilde{\mathbb{P}}(\exists \Gamma \text{ intersecting $\mathbb{B}_{r}$ such that $0\overset{\mathbb{B}_{r}}{\longleftrightarrow} \Gamma$ and }|\Gamma| \geq L) \leq Cr^{2}L^{1-d/2}
\end{align}
and
\begin{align}\label{5.10}
    \sum_{x\in \mathbb{B}_{r}} \sum_{ \substack{ d^{\text{ext}}( \Gamma,x) \le 1\\ |\Gamma|  \ge L }} |\Gamma|^i \widetilde{\mathbb{P}}(0\overset{\mathbb{B}_{r}}{\longleftrightarrow}x,\Gamma \in \mathcal{L})  \le  C r^{2} L^{i+2-d/2}.
\end{align}
\end{proposition}

Now we are ready to establish the bounds for $\widetilde{\P}(\mathsf{C}_{3})$-$\widetilde{\P}(\mathsf{C}_{5})$.

\textbf{For $\mathsf{C}_{3}$,} The bound is identical to the one presented in  \cite{ganguly2024ant}, so we will simply state the bound and omit the details. The bound is as follows:
\begin{align}\label{C_{3}}
\widetilde{\mathbb{P}}(\mathsf{C}_{3})\leq 
CL^{d-1}\e \cdot 3^{k}(\Lambda(3^{k-1}))^{2}.
\end{align}

Now it remains to bound the probabilities of $\mathsf{C}_{4}$ and $\mathsf{C}_{5}$. For these two events, we utilize the condition $\widetilde{B}(0,3^{k};\widetilde{\mathcal{L}}_{\mathscr{S}}) \subset \mathbb{B}_{   3^{k/2}}$, which helps us overcome the difficulty caused by the absence of the intrinsic volume bound.

\textbf{For $\mathsf{C}_{4}$,} we define 
\begin{align*}
\widetilde{\mathscr{S}}(r) := \{\widetilde{\Gamma} \in \widetilde{\mathscr{S}}:\widetilde{\Gamma} \text{ intersecting $\mathbb{B}_{r}$}\}.
\end{align*}  

As discussed in \cite[(122)]{ganguly2024ant}, if $\{i \in I_{2}\}\cap \{\partial B(0,3^{k};\widetilde{\mathcal{L}}_{\mathscr{S}}) \neq \emptyset\}\cap \mathscr{G}$ occurs, then the following event must occur:
\begin{center}
    $\exists$ discrete loop $\Gamma \in \mathcal{L}_{\widetilde{\mathscr{S}}}^{\leq i}$, $\exists u \in \Gamma $, $\exists v \in \mathbb{Z}^{d}$ with $d^{\text{ext}}(\Gamma,v)\leq 1$ s.t. 
\begin{align}\label{59}
    \{|\Gamma| \geq L\} \cap \{d(0,u;\widetilde{\mathcal{L}}_{\widetilde{\mathscr{S}}}) = i\} \cap 
    \{\partial B(v,3^{k-1}; \widetilde{\mathcal{L}}_{\widetilde{\mathscr{S}}} \setminus \widetilde{\mathcal{L}}_{\widetilde{\mathscr{S}}}^{\leq i} )\neq \emptyset\}.
\end{align}
\end{center}

Moreover, in $\mathsf{C}_{4}$, we have $\widetilde{B}(0,i;\widetilde{\mathcal{L}}_{\widetilde{\mathscr{S}}}^{\leq i}) \subseteq \mathbb{B}_{3^{k/2}}$, so it follows that $\widetilde{\mathcal{L}}_{\widetilde{\mathscr{S}}}^{\leq i} = \widetilde{\mathcal{L}}_{\widetilde{\mathscr{S}}(3^{k/2})}^{\leq i}$ and $\widetilde{B}(0,i; \widetilde{\mathcal{L}}_{\widetilde{\mathscr{S}}}^{\leq i})=\widetilde{B}(0,i; \widetilde{\mathcal{L}}_{\widetilde{\mathscr{S}}(3^{k/2})}^{\leq i})$. Furthermore, the inclusion $\widetilde{B}(0,i;\widetilde{\mathcal{L}}_{\widetilde{\mathscr{S}}}^{\leq i}) \subseteq \mathbb{B}_{3^{k/2}}$ implies that the $u$ we chose in (\ref{59}) actually satisfies $u\in \Gamma \cap \mathbb{B}_{3^{k/2}}$. This is crucial for our estimate.

We call the event in (\ref{59}) $\mathscr{G}_{i}$.
Hence, again by (\ref{5.5}) and a union bound, we have 
\begin{align}\label{45}
\widetilde{\mathbb{P}}(\mathsf{C}_{4}) &\leq  
\frac{8}{3^{k-1}}\sum_{i=3^{k-1}}^{3^{k}/2}\widetilde{\mathbb{P}}(\mathscr{G}_{i} \cap \{\widetilde{B}(0,i;\widetilde{\mathcal{L}}_{\widetilde{\mathscr{S}}}^{\leq i}) \subseteq \mathbb{B}_{3^{k/2}}\}) \\
&\overset{(\ref{59})}{\leq} \frac{8}{3^{k-1}}\sum_{i=3^{k-1}}^{3^{k}/2}\sum_{u\in \mathbb{B}_{3^{k/2}}} \sum_{\substack{\Gamma \ni u \\ |\Gamma| \geq L}}\sum_{\substack{v \in \mathbb{Z}^{d}\\ d^{\text{ext}}(\Gamma,v) \leq 1}}\widetilde{\mathbb{P}}\left(\right. d(0,u;\widetilde{\mathcal{L}}_{\widetilde{\mathscr{S}}} )=i, \nonumber\\   &\qquad \partial B(v, {     3^{k-1}      }; \widetilde{\mathcal{L}}_{\widetilde{\mathscr{S}}} \setminus \widetilde{\mathcal{L}}_{\widetilde{\mathscr{S}}}^{\leq i}) \neq \emptyset , \  \Gamma \in \mathcal{L}_{\widetilde{\mathscr{S}}}^{\leq i},\ \widetilde{B}(0,i;\widetilde{\mathcal{L}}_{\widetilde{\mathscr{S}}}^{\leq i}) \subset \mathbb{B}_{   3^{k/2}}\left.\right),
\end{align}
Note that the event 
\begin{align*}
    \{d(0,u;\widetilde{\mathcal{L}}_{\widetilde{\mathscr{S}}} )=i\} \cap  \{\Gamma \in \mathcal{L}_{\widetilde{\mathscr{S}}}^{\leq i}\}\cap \{\widetilde{B}(0,i;\widetilde{\mathcal{L}}_{\widetilde{\mathscr{S}}}^{\leq i}) \subset \mathbb{B}_{   3^{k/2}}\}
\end{align*}
is measurable with respect to $\widetilde{\mathcal{L}}_{\widetilde{\mathscr{S}}}^{\leq i}$, thus conditioning on $\widetilde{\mathcal{L}}_{\widetilde{\mathscr{S}}}^{\leq i}$ and applying (\ref{57}) again we got
\begin{align}\label{61}
&\widetilde{\mathbb{P}}\left(\right.\ d(0,u;\widetilde{\mathcal{L}}_{\widetilde{\mathscr{S}}} )=i,\; \partial B(v, {3^{k-1}}; \widetilde{\mathcal{L}}_{\widetilde{\mathscr{S}}} \setminus \widetilde{\mathcal{L}}_{\widetilde{\mathscr{S}}}^{\leq i}) \neq \emptyset , \  \Gamma \in \mathcal{L}_{\widetilde{\mathscr{S}}}^{\leq i},\ \widetilde{B}(0,i;\widetilde{\mathcal{L}}_{\widetilde{\mathscr{S}}}^{\leq i}) \subset \mathbb{B}_{   3^{k/2}}\left.\right) \nonumber \\ 
&=\widetilde{\mathbb{E}}\left[\widetilde{\mathbb{P}}(\partial B(v, {     3^{k-1}      }; \widetilde{\mathcal{L}}_{\widetilde{\mathscr{S}}} \setminus \widetilde{\mathcal{L}}_{\widetilde{\mathscr{S}}}^{\leq i}) \neq \emptyset \mid \widetilde{\mathcal{L}}_{\widetilde{\mathscr{S}}}^{\leq i}) \mathds{1}_{d(0,u;\widetilde{\mathcal{L}}_{\widetilde{\mathscr{S}}} )=i}\mathds{1}_{\Gamma \in \mathcal{L}_{\widetilde{\mathscr{S}}}^{\leq i}}\mathds{1}_{\widetilde{B}(0,i;\widetilde{\mathcal{L}}_{\widetilde{\mathscr{S}}}^{\leq i}) \subset \mathbb{B}_{   3^{k/2}}}\right] \nonumber\\ &\overset{(\ref{57})}{\leq} \Lambda(3^{k-1}) \widetilde{\mathbb{P}}(d(0,u;\widetilde{\mathcal{L}}_{\widetilde{\mathscr{S}}} )=i,\ \Gamma \in \mathcal{L}_{\widetilde{\mathscr{S}}}^{\leq i},\ \widetilde{B}(0,i;\widetilde{\mathcal{L}}_{\widetilde{\mathscr{S}}}^{\leq i}) \subset \mathbb{B}_{   3^{k/2}}).
\end{align}
It's straightforward to see that 
\begin{align}\label{62}
\{d(0,u;\widetilde{\mathcal{L}}_{\widetilde{\mathscr{S}}} )=i\} \cap \{\widetilde{B}(0,i;\widetilde{\mathcal{L}}_{\widetilde{\mathscr{S}}}^{\leq i})\subseteq  \B_{3^{k/2}}\} &\subseteq \{d(0,u;\widetilde{\mathcal{L}}_{\widetilde{\mathscr{S}}} )=i\} \cap  \{0\overset{\mathbb{B}_{3^{k/2}}}{\longleftrightarrow} u\};\\ 
\{\Gamma \in \mathcal{L}_{\widetilde{\mathscr{S}}}^{\leq i}\} &\subseteq \{\Gamma \in \mathcal{L}_{\widetilde{\mathscr{S}}}\}.
\end{align}
So far, we have introduced the extrinsic connection event $\{0\overset{\mathbb{B}_{3^{k/2}}}{\longleftrightarrow} u\}$, which will ultimately help us eliminate the intrinsic constraint. By (\ref{45}) and (\ref{61}) we have that 
\begin{align}\label{63}
    \widetilde{\mathbb{P}}(\mathsf{C}_{4})  & \overset{(\ref{45}), (\ref{61})}{\leq} 
    \frac{8}{3^{k-1}}\sum_{i=3^{k-1}}^{3^{k}/2}\sum_{u\in \mathbb{B}_{   3^{k/2}}} \sum_{\substack{\Gamma \ni u \\ |\Gamma| \geq L}}\sum_{\substack{v \in \mathbb{Z}^{d}\\ d^{\text{ext}}(\Gamma,v) \leq 1}} \Lambda(3^{k-1})\widetilde{\mathbb{P}}(d(0,u;\widetilde{\mathcal{L}}_{\widetilde{\mathscr{S}}} )=i, \ \Gamma \in \mathcal{L}_{\widetilde{\mathscr{S}}}^{\leq i},\ \widetilde{B}(0,i;\widetilde{\mathcal{L}}_{\widetilde{\mathscr{S}}}^{\leq i}) \subset \mathbb{B}_{   3^{k/2}}) \nonumber \\ &\leq 
    \frac{8}{3^{k-1}}\sum_{i=3^{k-1}}^{3^{k}/2}\sum_{u\in \mathbb{B}_{   3^{k/2}}} \sum_{\substack{\Gamma \ni u \\ |\Gamma| \geq L}}|\Gamma| \Lambda(3^{k-1})\widetilde{\mathbb{P}}(d(0,u;\widetilde{\mathcal{L}}_{\widetilde{\mathscr{S}}} )=i, \ \widetilde{B}(0,i;\widetilde{\mathcal{L}}_{\widetilde{\mathscr{S}}}^{\leq i}) \subset \mathbb{B}_{3^{k/2}},\ \Gamma \in \mathcal{L}_{\widetilde{\mathscr{S}}}^{\leq i})\nonumber\\
    &\overset{(\ref{62})}{\leq} 
    \frac{C}{3^{k-1}}\Lambda(3^{k-1})\sum_{u\in \mathbb{B}_{   3^{k/2}}} \sum_{\substack{\Gamma \ni u \\ |\Gamma| \geq L}}\sum_{i=3^{k-1}}^{3^{k}/2}|\Gamma|\widetilde{\mathbb{P}}(d(0,u;\widetilde{\mathcal{L}}_{\widetilde{\mathscr{S}}} )=i,\ 0 \overset{\mathbb{B}_{3^{k/2}}}{\longleftrightarrow} u,\ \Gamma \in \mathcal{L}_{\widetilde{\mathscr{S}}}).
\end{align}
Note that 
\begin{align}\label{64}
    &\sum_{i=3^{k-1}}^{3^{k}/2}\widetilde{\mathbb{P}}(d(0,u;\widetilde{\mathcal{L}}_{\widetilde{\mathscr{S}}} )=i,\ 0 \overset{\mathbb{B}_{3^{k/2}}}{\longleftrightarrow} u,\ \Gamma \in \mathcal{L}_{\widetilde{\mathscr{S}}}) \nonumber\\ 
    &= \widetilde{\mathbb{P}}(3^{k-1}\leq d(0,u;\widetilde{\mathcal{L}}_{\widetilde{\mathscr{S}}(3^{k/2})} ) \leq 3^{k}/2, \ 0 \overset{\mathbb{B}_{3^{k/2}}}{\longleftrightarrow} u,\ \Gamma \in \widetilde{\mathcal{L}}_{\widetilde{\mathscr{S}}}) \nonumber\\ &\leq \widetilde{\mathbb{P}}(0\overset{3^{k}/2}{\longleftrightarrow} u,\; 0 \overset{\mathbb{B}_{3^{k/2}}}{\longleftrightarrow} u,\; \Gamma \in \mathcal{L}) \leq \widetilde{\mathbb{P}}( 0 \overset{\mathbb{B}_{3^{k/2}}}{\longleftrightarrow} u,\; \Gamma \in \mathcal{L}).
\end{align}
Combining (\ref{63}) and (\ref{64}) we get 
\begin{align}\label{pre C_{4}}
    \widetilde{\mathbb{P}}(\mathsf{C}_{4}) \overset{(\ref{63}),(\ref{64})}{\leq} \frac{C}{3^{k-1}}\Lambda(3^{k-1}) \sum_{u\in \mathbb{B}_{   3^{k/2}}} \sum_{\substack{\Gamma \ni u \\ |\Gamma| \geq L}} \widetilde{\mathbb{P}}( 0 \overset{\mathbb{B}_{3^{k/2}}}{\longleftrightarrow} u,\; \Gamma \in \mathcal{L}).
\end{align}
Now applying (\ref{5.10}) to (\ref{pre C_{4}}), choosing $r=  3^{k/2}$ and $i=1$, gives
\begin{align}\label{C_{4}}
    \widetilde{\mathbb{P}}(\mathsf{C}_{4}) \leq \frac{C}{3^{k-1}}\Lambda(3^{k-1})\cdot   3^{k}L^{3-d/2} \leq C    L^{3-d/2}\Lambda(3^{k-1}).
\end{align}

\textbf{For $\mathsf{C}_{5}$, }when $\mathsf{C}_{5}$ happens, there exists $\Gamma \in \mathcal{L}$ such that $\Gamma$ intersects $\mathbb{B}_{3^{k/2}}$, $|\Gamma| \geq 3^{k-1}/2$ and $0\overset{\mathbb{B}_{r}}{\longleftrightarrow} \Gamma$. Then, if we apply (\ref{5.9}) choosing $r =    3^{k/2}$ and $L= 3^{k-1}/2$, we have 
\begin{align} \label{C_{5}}
\widetilde{\mathbb{P}}(\mathsf{C}_{5})\leq C    3^{k}\cdot (3^{k-1}/2)^{1-d/2} \leq C  (3^{k})^{2-d/2}.
\end{align}

\textbf{Conclusion: }combining the bounds for $\mathsf{C}_{1}$ to $\mathsf{C}_{5}$ and taking the supremum over all collection of loops $\widetilde{\mathscr{S}}$, we deduce that there exists a constant $\tilde{C}(d)>0$ such that for any $L,\e >0$ and $k\in \mathbb{N}$, 
\begin{align}\label{recursion}
\Lambda(3^{k}) &\leq \frac{\tilde{C}}{\sqrt{\e}3^{k}} + \frac{\tilde{C}}{  3^{k}}+ \tilde{C}L^{d-1}\e \cdot 3^{k}(\Lambda(3^{k-1}))^{2}+ \tilde{C}  L^{3-d/2}\Lambda(3^{k-1})+\tilde{C}  (3^{k})^{2-d/2}.
\end{align}
Set $L$ to be a large enough constant such that 
\begin{align*}
    3\tilde{C}L^{3-d/2}<\frac{1}{5},
\end{align*}
and then take a constant $D>0$ such that 
\begin{align*}
\sqrt{45}\tilde{C}^{3/2}D^{1/2}L^{(d-1)/2}<\frac{D}{5}, \ \tilde{C}<\frac{D}{5},
\end{align*}
finally, choose $\e = (45\tilde{C}DL^{d-1})^{-1}$. We claim that 
\begin{align*}
    \Lambda(3^{k}) \leq \frac{D}{3^{k}}.
\end{align*}
Assuming $\Lambda(3^{k-1})\leq \frac{D}{3^{k-1}}$, by (\ref{recursion}) and the fact $d>6$, 
\begin{align}
    \Lambda(3^{k}) \leq (\frac{\tilde{C}}{\sqrt{\e}} + \tilde{C}+ 9\tilde{C}D^{2}L^{d-1}\e + 3\tilde{C}D L^{3-d/2}+\tilde{C})\frac{1}{3^{k}} \leq \frac{D}{3^{k}},
\end{align}
where we use the conditions when choosing parameters. Thus, we prove the claim by induction. Finally, for any $3^{k-1}\leq r<3^{k}$,
\begin{align*}
    \Lambda(r) \leq \Lambda(3^{k-1}) \leq \frac{D}{3^{k-1}}\leq \frac{3D}{r}.
\end{align*}
\end{proof}

We finish this section by the proof of Proposition \ref{large loop 2}. To prove this proposition, we need the following connection property of a discrete loop $\Gamma \in \mathcal{L}$ to 0. We denote by $\mathsf{Cont}(\Gamma)$ the collection of loops $\widetilde{\Gamma}$ in $\widetilde{\mathcal{L}}$ such that $\mathsf{Trace}(\widetilde{\Gamma}) = \Gamma$ (recall that $\mathsf{Trace}(\widetilde{\Gamma})$ is the discrete loops with the same lattice trajectory of $\widetilde{\Gamma}$).
\begin{lemma}\label{lem 6.7}
    Let $r \in \mathbb{N}$ and $\Gamma$ be any discrete loops in $\mathcal{L}$ such that $0 \overset{\mathbb{B}_{r}}{\longleftrightarrow} \Gamma$, then there exists $x\in \mathbb{B}_{r}$ with $x \sim \Gamma$ such that $0\overset{\mathbb{B}_{r}, \widetilde{\mathcal{L}}\setminus \mathsf{Cont}(\Gamma)}{\longleftrightarrow} x$.
\end{lemma}
\begin{proof}
    Let $\ell$ be a path connecting 0 to $\Gamma$ which only uses edges in $\mathbb{B}_{r}$. Let $y$ be the first point on $\ell$ intersecting $\Gamma$ and let $x$ be the last point on $\ell$ before it reaches $y$. Then the connection 0 to $x$ cannot involve any loops in $\mathsf{Cont}(\Gamma)$,  since this would contradict the fact that $y$ is the first point on $\ell$ that intersects $\Gamma$. Furthermore, this provides a path from 0 to $x$
    that only uses edges within $\B_{r}$ Specifically, this path is simply the portion of $\ell$ before it passes through $y$ for the first time.
\end{proof}

\begin{proof}[Proof of Proposition \ref{large loop 2}]
By Lemma \ref{lem 6.7} the probability that $\exists \Gamma$ intersecting $\mathbb{B}_{r}$ such that $0\overset{\mathbb{B}_{r}}{\longleftrightarrow} \Gamma$ and $|\Gamma| \geq L$ can be bounded above by
\begin{align*}
\sum_{x\in \mathbb{B}_{r}}\sum_{\substack{ d^{\text{ext}}( \Gamma,x) \le 1 \\ |\Gamma|\geq L}}\widetilde{\mathbb{P}}(0\overset{\mathbb{B}_{r}, \widetilde{\mathcal{L}}\setminus \mathsf{Cont}(\Gamma)}{\longleftrightarrow} x, \; \Gamma \in \mathcal{L}) &\leq 
\sum_{x\in \mathbb{B}_{r}}\sum_{\substack{d^{\text{ext}}( \Gamma,x) \le 1 \\ |\Gamma|\geq L}}\widetilde{\mathbb{P}}(0\overset{\mathbb{B}_{r}, \widetilde{\mathcal{L}}\setminus \mathsf{Cont}(\Gamma)}{\longleftrightarrow} x) \mathbb{P}(\Gamma \in \mathcal{L})\\&\leq 
\sum_{x\in \mathbb{B}_{r}}\sum_{\substack{d^{\text{ext}}( \Gamma,x) \le 1 \\ |\Gamma|\geq L}}\widetilde{\mathbb{P}}(0\leftrightarrow x )\widetilde{\mathbb{P}}(\Gamma \in \mathcal{L}) \\ \text{By translation invariance} \qquad  &= 
\sum_{x\in \mathbb{B}_{r}}\sum_{\substack{d^{\text{ext}}( \Gamma^{\prime},0) \le 1 \\ |\Gamma^{\prime}|\geq L}}\widetilde{\mathbb{P}}(-x\leftrightarrow 0 )\widetilde{\mathbb{P}}(\Gamma^{\prime} \in \mathcal{L}) \\ &= 
\sum_{\substack{d^{\text{ext}}( \Gamma^{\prime},0) \le 1 \\ |\Gamma^{\prime}|\geq L}}\sum_{x\in \mathbb{B}_{r}}\widetilde{\mathbb{P}}(-x\leftrightarrow 0 )\widetilde{\mathbb{P}}(\Gamma^{\prime} \in \mathcal{L}) \\&\leq Cr^{2}\sum_{\substack{d^{\text{ext}}( \Gamma^{\prime},0) \le 1 \\ |\Gamma^{\prime}|\geq L}}\widetilde{\mathbb{P}}(\Gamma^{\prime} \in \mathcal{L}),
\end{align*}
where we used the estimates 
\begin{align*}
    \sum_{x\in \mathbb{B}_{r}}\mathbb{P}(0\leftrightarrow x) \leq C  \sum_{x\in \mathbb{B}_{r}} |x|^{2-d} \leq Cr^{2}.
\end{align*}
By Lemma \ref{loop one point} we know that or any $L\in \mathbb{N}$, there exists $C>0$ such that
\begin{align*}
\sum_{\substack{d^{\text{ext}}( \Gamma^{\prime},0) \le 1 \\ |\Gamma^{\prime}|= L}}\widetilde{\mathbb{P}}(\Gamma^{\prime} \in \mathcal{L}) \leq C L^{-d/2}, 
\end{align*}
then we can bound the LHS of (\ref{5.9}) by 
\begin{align*}
Cr^{2}\sum_{k \geq L}k^{-d/2} \leq Cr^{2}L^{1-d/2}.
\end{align*}
Next, we prove (\ref{5.10}). By the similar reason in Lemma \ref{lem 6.7}, $\{d^{\text{ext}}(\Gamma,x)\leq 1,\ 0\overset{\mathbb{B}_{r}}{\longleftrightarrow} x\}$ implies that there exists $y\in \mathbb{B}_{r}$ with $d^{\text{ext}}(\Gamma,y)\leq 1$ such that $0\overset{\mathbb{B}_{r}, \widetilde{\mathcal{L}}\setminus \mathsf{Cont}(\Gamma)}{\longleftrightarrow}y$ not using $\Gamma$. Hence, again by a union bound, 
\begin{align*}
\sum_{x\in \mathbb{B}_{r}} \sum_{ \substack{ d^{\text{ext}}( \Gamma,x) \le 1\\ |\Gamma|  \ge L }} |\Gamma|^i \widetilde{\mathbb{P}}(0\overset{\mathbb{B}_{r}}{\longleftrightarrow}x,\Gamma \in \mathcal{L}) &\leq \sum_{x\in \mathbb{B}_{r}} \sum_{ \substack{ d^{\text{ext}}( \Gamma,x) \le 1\\ |\Gamma|  \ge L }} \sum_{\substack{y\in\mathbb{B}_{r}\\ d^{\text{ext}}(\Gamma,y)\leq 1}} |\Gamma|^i \widetilde{\mathbb{P}}(0\overset{\mathbb{B}_{r}, \widetilde{\mathcal{L}}\setminus \mathsf{Cont}(\Gamma)}{\longleftrightarrow}y ,\ \Gamma \in \mathcal{L}) \\ &\leq 
\sum_{x\in \mathbb{B}_{r}} \sum_{ \substack{ d^{\text{ext}}( \Gamma^{\prime},0) \le 1\\ |\Gamma^{\prime}|  \ge L }} \sum_{\substack{y^{\prime}\in\mathbb{Z}^{d}\\ d^{\text{ext}}(\Gamma^{\prime},y^{\prime})\leq 1}} |\Gamma^{\prime}|^i \widetilde{\mathbb{P}}(-x\leftrightarrow y^{\prime})\widetilde{\mathbb{P}} (\Gamma^{\prime} \in \mathcal{L}) \\ &= 
\sum_{ \substack{ d^{\text{ext}}( \Gamma^{\prime},0) \le 1\\ |\Gamma^{\prime}|  \ge L }} \sum_{\substack{y^{\prime}\in\mathbb{Z}^{d}\\ d^{\text{ext}}(\Gamma^{\prime},y)\leq 1}} |\Gamma^{\prime}|^i \sum_{x\in \mathbb{B}_{r}}\widetilde{\mathbb{P}}(-x\leftrightarrow y^{\prime})\widetilde{\mathbb{P}} (\Gamma^{\prime} \in \mathcal{L}) \\&\leq 
Cr^{2}\sum_{ \substack{ d^{\text{ext}}( \Gamma^{\prime},0) \le 1\\ |\Gamma^{\prime}|  \ge L }} \sum_{\substack{y^{\prime}\in\mathbb{Z}^{d}\\ d^{\text{ext}}(\Gamma^{\prime},y)\leq 1}} |\Gamma^{\prime}|^i\widetilde{\mathbb{P}} (\Gamma^{\prime} \in \mathcal{L}) \\ 
&\leq 
Cr^{2}\sum_{ \substack{ d^{\text{ext}}( \Gamma^{\prime},0) \le 1\\ |\Gamma^{\prime}|  \ge L }}  |\Gamma^{\prime}|^{i+1}\widetilde{\mathbb{P}} (\Gamma^{\prime} \in \mathcal{L}) \\&\leq 
Cr^{2}\sum_{k \geq L}k^{i+1-d/2}\leq 
Cr^{2}L^{i+2-d/2}.
\end{align*}
Note that in the last sum above is summable because $i<(d-4)/2$ implies $i+1-d/2<-1$.
\end{proof}

\subsection{Proof of the lower bound in Theorem \ref{thm 1.4}}\label{lowertail}

We establish the lower bound of \eqref{1.6} using a second moment argument, which combines both of our comparison principles. The main strategy is to lower bound the chemical one arm probability by showing that on the event $0\leftrightarrow \partial \B_{r},$ with constant probability $d(0,\partial \B_{r})\gtrsim r^2.$ Thus, this part relies on the extrinsic-intrinsic comparison.

To do this we first use the comparison between the geodesic and its simple counterpart to reduce the analysis to the latter case. For the simple path, we use Lemma \ref{large loops on geo} to show large loops can be ignored. Proposition \ref{prop 5.11} then allows us to lower bound the expectation of the number of points landing only on small loops. We finally upper bound the second moment of the same using the BKR inequality. It is worth emphasizing that the fact that we only deal with small loops will be crucial in bounding the second moment.

\begin{proof}[Proof of the lower bound in Theorem \ref{thm 1.4}]
    Given a collection of loops $\widetilde{\mathscr{S}}$ such that $0\overset{\widetilde{\mathscr{S}}}{\longleftrightarrow} \partial \B_{r}$, one may find a deterministic rule to construct a $\widetilde{\mathscr{S}}$-geodesic from 0 to $\partial \B_{r}$. On the event $\{0\leftrightarrow \partial \B_{r}\}$, denote the resulting geodesic by $\gamma_{0\rightarrow \partial \B_{r}}$. Furthermore, as discussed in Section \ref{loopchain} and the proof of the lower bound in \eqref{1.4}, there exists a minimal glued loop sequence $\overline{L}$ associated to $\gamma_{0\rightarrow \partial \B_{r}}$, as well as a simple glued loop subsequence sequence $\overline{L}_{1}$ from 0 to $\partial \B_{r}$ only using glued loops in $\overline{L}$. One may then find a simple path $\hat{\gamma}_{0\rightarrow \partial \B_{r}}$ from 0 to $\partial \B_{r}$ associated with the simple chain $\overline{L}_{1}$. Given the loop soup realization, we can define deterministic rules to generate $\overline{L}$, $\overline{L}_{1}$ and $\hat{\gamma}_{0\rightarrow \partial \B_{r}}$. Thus, by fixing these rules, we ensure that these objects are well-defined and measurable with respect to $\widetilde{\mathcal{L}}$.
    
    Recalling that $\hat{\gamma}_{0\rightarrow \partial \B_{r}}(\alpha)$ for a fixed $\alpha \in (0,\frac{1}{3})$ denotes the portion of $\hat{\gamma}_{0\rightarrow \partial \B_{r}}$ before it exits $\B_{\alpha r}$,  we define the following random variables: 
    \begin{align*}
        X &: = |\{x\in \hat{\gamma}_{0\rightarrow \partial \B_{r}}: |\overline{\Gamma}|\leq K,\ \forall\, \overline{\Gamma} \in \mathsf{Set}(\overline{L}_{1}) \text{ wtih }x\in \overline{\Gamma}\}|,\\
        Y &: = |\{x\in \hat{\gamma}_{0\rightarrow \partial \B_{r}}(\alpha): |\overline{\Gamma}|\leq K,\ \forall \,\overline{\Gamma} \in \mathsf{Set}(\overline{L}_{1}) \text{ wtih }x\in \overline{\Gamma}\}|,\\
        Z&: = |\{x\in \hat{\gamma}_{0\rightarrow \partial \B_{r}}(\alpha): \exists\, \overline{\Gamma} \in \mathsf{Set}(\overline{L}_{1}) \text{ such that $x\in \overline{\Gamma}$ and $|\Gamma|>K$}\}|,
    \end{align*}
    where $K$ is a constant satisfying the condition in Lemma \ref{large loops on geo} for some $\e$, which will be fixed later. 
    Thus $X$ is the number of  points on $\hat{\gamma}_{0\rightarrow \partial \B_{r}}$ which are only incident on small loops in $\overline L_1$, whereas $Y$ denotes the number of them on  $\hat{\gamma}_{0\rightarrow \partial \B_{r}}(\alpha)$ and finally $Z$ denotes the number of points on $\hat{\gamma}_{0\rightarrow \partial \B_{r}}(\alpha)$ which is incident on at least one large loop in $\overline L_1$.
    
    Now we control $\widetilde{\E}[Y\mid 0\leftrightarrow \partial \B_{r}]$ and $\widetilde{\E}[Y^{2}\mid 0\leftrightarrow \partial \B_{r}]$ so that we can apply the second moment method. 

    We start with the conditional first moment of $Y$. By Proposition \ref{prop 5.11}, {$\widetilde{\E}[|\hat{\gamma}_{0\rightarrow \partial \B_{r}}(\alpha)| \mathds{1}_{0\leftrightarrow \partial \B_{r}}] \geq c\alpha^{4}$}. Additionally, by Lemma \ref{large loops on geo}, we obtain an upper bound for $\widetilde{\E}[Z \mathds{1}_{0\leftrightarrow \partial \B_{r}}]$ as 
    \begin{align*}
        \widetilde{\E}[Z \mathds{1}_{0\leftrightarrow \partial \B_{r}}] \leq  \sum_{z\in \mathbb{B}_{\alpha r}} \sum_{\substack{\Gamma \ni z\\|\Gamma| \geq K}}\widetilde{\mathbb{P}}(\Gamma\in \mathcal{L})\widetilde{\mathbb{P}}(0\overset{\overline{\mathcal{L}}\setminus \{\overline{\Gamma}\},\mathbb{B}_{r}}{\longleftrightarrow} \overline{\Gamma} \circ \overline{\Gamma} \overset{\overline{\mathcal{L}}\setminus \{\overline{\Gamma}\}}{\longleftrightarrow} \partial \mathbb{B}_{r}) \overset{\eqref{5-59}}{\leq} \e.
    \end{align*}
    Observe that $|\hat{\gamma}_{0\rightarrow \partial \B_{r}}(\alpha)| = Y+Z$. Setting $\e = \frac{c}{2}\alpha^{4}$ leads to 
    \begin{align}\label{1st moment}
        \widetilde{\E}[Y\mathds{1}_{0\leftrightarrow \partial \B_{r}}] \geq c\alpha^{4} -\frac{c}{2}\alpha^{4} = \frac{c}{2}\alpha^{4}.
    \end{align}
    
    Next, we bound the conditional second moment. Consider $u,v\in \{x\in \hat{\gamma}_{0\rightarrow \partial \B_{r}}(\alpha): |\overline{\Gamma}|\leq K,\ \forall\,  \overline{\Gamma} \in \mathsf{Set}(\overline{L}_{1}) \text{ with }x\in \overline{\Gamma}\}$. Without loss of generality, we assume that $\hat{\gamma}_{0\rightarrow \partial \B_{r}}$ passes $u$ before $v$. Then one of the following two events occurs: 
    \begin{enumerate}
        \item $A_{1}(u,v)$: there exists $\overline{\Gamma},\overline{\Gamma}^{\prime} \in \overline{L}_{1}$ with $|\overline{\Gamma}|,|\overline{\Gamma}^{\prime}|\leq K$ such that  $u\in \overline{\Gamma}$ and $v \in \overline{\Gamma}^{\prime}$, and 
        \begin{align*}
            0\overset{\overline{\mathcal{L}}\setminus \{\overline{\Gamma},\overline{\Gamma}^{\prime}\}}{\longleftrightarrow} \overline{\Gamma}\circ  \overline{\Gamma} \overset{\overline{\mathcal{L}}\setminus \{\overline{\Gamma},\overline{\Gamma}^{\prime}\}}{\longleftrightarrow}\overline{\Gamma}^{\prime}\circ \overline{\Gamma}^{\prime} \overset{\overline{\mathcal{L}}\setminus \{\overline{\Gamma},\overline{\Gamma}^{\prime}\}}{\longleftrightarrow} \partial \B_{r}.
        \end{align*}
        \item $A_{2}(u,v)$: there exists $\overline{\Gamma} \in \overline{L}_{1}$ with $|\overline{\Gamma}|\leq K$ such that $u,v \in \overline{\Gamma}$ and 
        \begin{align*}
            0\overset{\overline{\mathcal{L}}\setminus \{\overline{\Gamma}\}}{\longleftrightarrow} \overline{\Gamma}\circ \overline{\Gamma}\overset{\overline{\mathcal{L}}\setminus \{\overline{\Gamma}\}}{\longleftrightarrow} \partial \B_{r}.
        \end{align*}
    \end{enumerate}
    See Figure \ref{2 cases} for the illustrations.
    \begin{figure}[h]
        \centering
        \includegraphics[scale=.2]{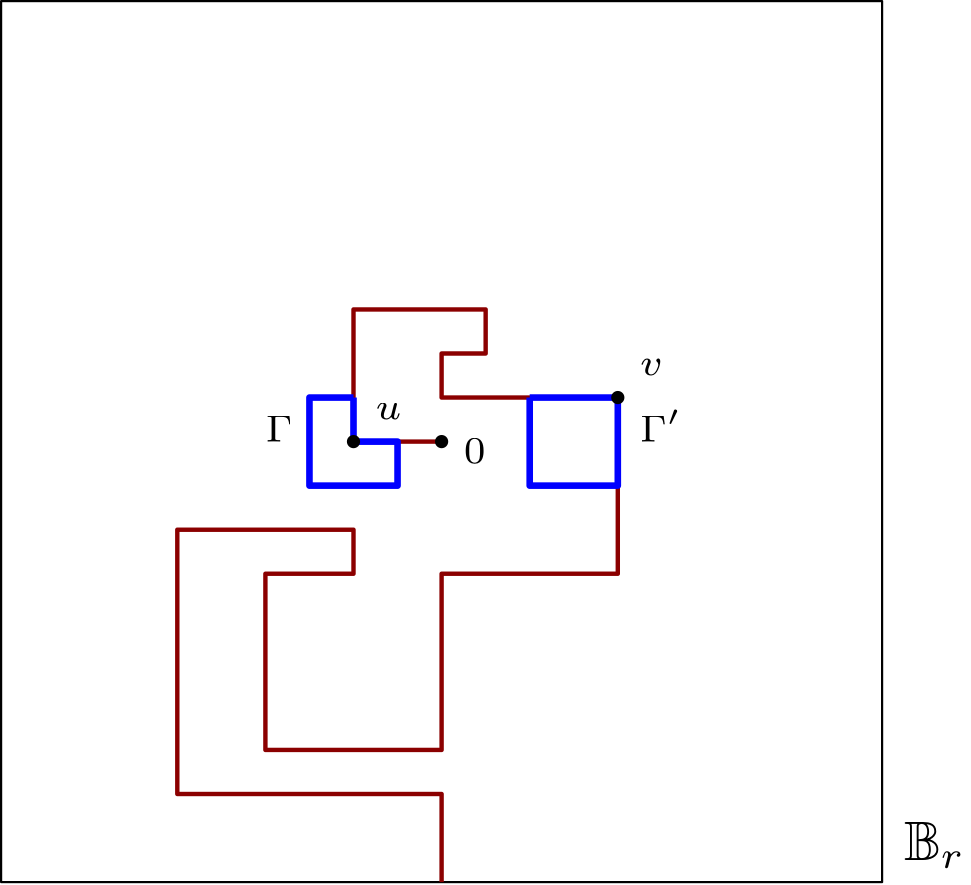}
        \includegraphics[scale=.2]{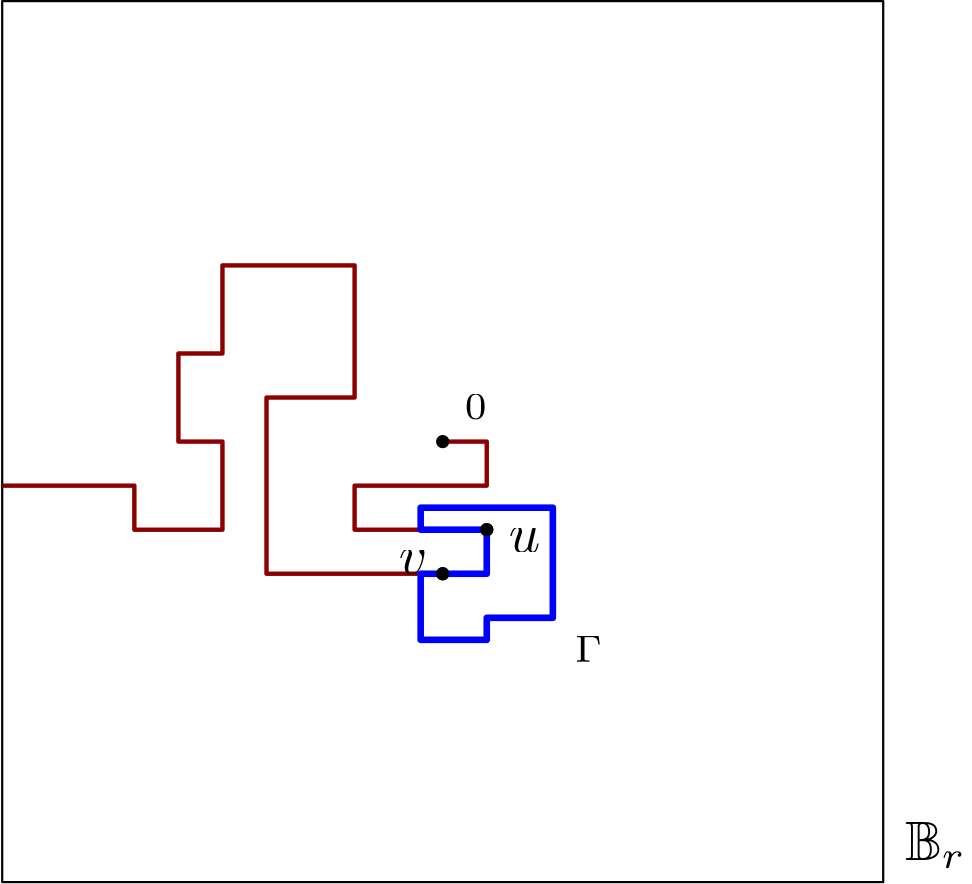}
        \caption{On the left is an illustration of $A_{1}(u,v)$, where $u$ and $v$ belong to two distinct loops $\Gamma$ and $\Gamma^{\prime}$ in $\mathsf{Set}(\overline{L}_{1})$. The right figure illustrates $A_{2}(u,v)$, where $u$ and $v$ are located on the same loop $\Gamma$.
        }
        \label{2 cases}
        \end{figure}
    Since in $A_{1}(u,v)$ and $A_{2}(u,v)$ we only deal with glued loops of lengths at most $K$, the standard BKR type argument suffices to show that 
    \begin{align*}
        \sum_{u,v \in \B_{\alpha r}}\widetilde{\P}(A_{1}(u,v)) \leq C \alpha^{4}r^{2},\
        \sum_{u,v \in \B_{\alpha r}}\widetilde{\P}(A_{2}(u,v)) \leq C \alpha^{2}.
    \end{align*}
    (If the length of $\overline{\Gamma}$ was unbounded, then the second event produces an unbounded contribution unless $d> 8$.)
    Using this, we derive an upper bound for the second moment: 
    \begin{align}\label{2nd moment}
        \widetilde{\E}[Y^{2}\mathds{1}_{0\leftrightarrow \partial \B_{r}}] \leq \sum_{u,v \in \B_{\alpha r}}\widetilde{\P}(A_{1}(u,v)) + \sum_{u,v \in \B_{\alpha r}}\widetilde{\P}(A_{2}(u,v)) \leq C\alpha^{4}r^{2} + C \alpha^{2} \leq C^{\prime} \alpha^{4}r^{2}.
    \end{align}
    Applying the  Paley–Zygmund inequality, along with \eqref{1st moment} and \eqref{2nd moment}, we conclude that for a fixed $c^{\prime} \leq \frac{c}{4C_{1}}\alpha^{4}$,
    \begin{align*}
        \widetilde{\P}(Y>c^{\prime}r^{2} \mid 0\leftrightarrow \partial \B_{r}) &\geq \widetilde{\P}(Y \geq \frac{1}{2}\widetilde{\E}[Y\mid 0\leftrightarrow \partial \B_{r}]\mid 0\leftrightarrow \partial \B_{r}) \nonumber\\
        &\geq \frac{1}{4} \cdot \frac{\widetilde{\E}[Y\mid 0\leftrightarrow \partial \B_{r}]^{2}}{\widetilde{\E}[Y^{2}\mid 0\leftrightarrow \partial \B_{r}]} \geq \frac{c^{2}c_{1}}{16C_{1}^{2}} \geq c^{\prime\prime}.
    \end{align*}
    Using this result, and noting that $X\geq Y$, we immediately have
    \begin{align}\label{tail x}
        \widetilde{\P}(X>c^{\prime}r^{2} \mid 0\leftrightarrow \partial \B_{r}) \geq c^{\prime\prime}.
    \end{align}
    Finally, by Lemma \ref{geometric lemma} and the discussion in the proof of the lower bound in \eqref{1.4}, we know that 
    \begin{align*}
        3d(0,\partial \B_{r}) = 3|\gamma| \geq |\mathsf{Set}(\overline{L}_{1})| \geq X/K.
    \end{align*}
    This combined with \eqref{tail x} gives that 
    \begin{align*}
        \widetilde{\P}(d(0,\partial \B_{r}) \geq \frac{c^{\prime}}{3K} r^{2} \mid 0\leftrightarrow \partial \B_{r}) \geq c^{\prime\prime}.
    \end{align*}
    Using this tail lower bound and the one-arm probability lower bound \eqref{2.6} we complete the proof as follows:
    \begin{align*}
        \widetilde{\P}(\partial B(0,r)\neq \emptyset)
        &\geq \widetilde{\P}(d(0,\partial \B_{\sqrt{\frac{3K}{c^{\prime}}r}}) \geq r,\ 0\leftrightarrow \partial \B_{\sqrt{\frac{3K}{c^{\prime}}r}}) \\
         &= \widetilde{\P}(d(0,\partial \B_{\sqrt{\frac{3K}{c^{\prime}}r}}) \geq r \mid 0\leftrightarrow \partial \B_{\sqrt{\frac{3K}{c^{\prime}}r}}) \cdot\pi_{1}\left(\sqrt{\frac{3K}{c^{\prime}}r}\right) \\
        &\overset{\eqref{2.6}}{\geq}c^{\prime\prime} \frac{c^{\prime}c_{1}}{3Kr} \geq \frac{c^{\prime\prime\prime}}{r}.
    \end{align*}
\end{proof}

\section{Large loops don't help connect}\label{largeloop}

In this section we prove Theorem \ref{werner12} as an application of Theorem \ref{lem 4.4}.

\begin{proof}[Proof of Theorem \ref{werner12}]
We begin by proving (\ref{large loop 3}) and (\ref{large loop 4}) will follow directly from (\ref{large loop 3}) and Theorem \ref{lem 4.4}. For any $L >0$ and $x \in \mathbb{B}_{r}$, if $0 \overset{\mathbb{B}_{\beta r}}{\longleftrightarrow}x$ and the connection must use some glued loop $\overline{\Gamma} \in \overline{\mathcal{L}}\setminus \overline{\mathcal{L}}_{\leq L}$, for the similar reasons as in the proof of Lemma \ref{lem 6.7}, there exists $u,v \in \mathbb{B}_{\beta r}$ with $u \sim \overline{\Gamma}$,  $v \sim \overline{\Gamma}$ such that $0 \overset{\overline{\mathcal{L}}\setminus \{\overline{\Gamma}\}}{\longleftrightarrow} u \circ v \overset{\overline{\mathcal{L}}\setminus \{\overline{\Gamma}\}}{\longleftrightarrow} x$. Then by BK inequality we can bound the difference $\widetilde{\mathbb{P}}(0 \overset{\mathbb{B}_{\beta r}}{\longleftrightarrow}x) - \widetilde{\mathbb{P}}(0 \overset{\mathbb{B}_{\beta r},\ \widetilde{\mathcal{L}}^{\leq L} }{\longleftrightarrow}x)$ as follows:
\begin{align*}
    \sum_{u \in \mathbb{B}_{\beta r}}\sum_{\substack{u \sim \Gamma \\ |\Gamma| \geq L}}\sum_{\substack{v \in \mathbb{B}_{\beta r}\\ v \sim \Gamma}} \widetilde{\mathbb{P}}(\Gamma \in \mathcal{L})\widetilde{\mathbb{P}}(0\leftrightarrow u)\widetilde{\mathbb{P}}(v\leftrightarrow x).
\end{align*}
Summing over $x \in \mathbb{B}_{r}$, we have that 
\begin{align*}
    \sum_{x\in \mathbb{B}_{r}}\widetilde{\mathbb{P}}(0 \overset{\mathbb{B}_{\beta r}}{\longleftrightarrow}x) - \sum_{x\in \mathbb{B}_{r}}\widetilde{\mathbb{P}}(0 \overset{\mathbb{B}_{\beta r},\ \overline{\mathcal{L}}_{\leq L} }{\longleftrightarrow}x) &\leq 
    \sum_{x\in \mathbb{B}_{r}}\sum_{u \in \mathbb{B}_{\beta r}}\sum_{\substack{u \sim \Gamma \\ |\Gamma| \geq L}}\sum_{\substack{v \in \mathbb{B}_{\beta r}\\ v \sim \Gamma}} \widetilde{\mathbb{P}}(\Gamma \in \mathcal{L})\widetilde{\mathbb{P}}(0\leftrightarrow u)\widetilde{\mathbb{P}}(v\leftrightarrow x) \\ &=
    \sum_{u \in \mathbb{B}_{\beta r}}\sum_{\substack{u \sim \Gamma \\ |\Gamma| \geq L}}\sum_{\substack{v \in \mathbb{B}_{\beta r}\\ v \sim \Gamma}}\sum_{x\in \mathbb{B}_{r}} \widetilde{\mathbb{P}}(\Gamma \in \mathcal{L})\widetilde{\mathbb{P}}(0\leftrightarrow u)\widetilde{\mathbb{P}} (v\leftrightarrow x) \\ &\leq 
    C r^{2} \sum_{u \in \mathbb{B}_{\beta r}}\sum_{\substack{u \sim \Gamma \\ |\Gamma| \geq L}}\sum_{ v \sim \Gamma} \widetilde{\mathbb{P}}(\Gamma \in \mathcal{L})\widetilde{\mathbb{P}}(0\leftrightarrow u) \\&=
    C r^{2} \sum_{u \in \mathbb{B}_{\beta r}}\sum_{\substack{u \sim \Gamma \\ |\Gamma| \geq L}} |\Gamma|\widetilde{\mathbb{P}}(\Gamma \in \mathcal{L})\widetilde{\mathbb{P}}(0\leftrightarrow u) \\ \text{By translation invariance,} \qquad &= 
    C r^{2} \sum_{u \in \mathbb{B}_{\beta r}}\sum_{\substack{0 \sim \Gamma \\ |\Gamma| \geq L}} |\Gamma|\widetilde{\mathbb{P}}(\Gamma \in \mathcal{L})\widetilde{\mathbb{P}}(0\leftrightarrow -u)
    \\&\leq 
    Cr^{4}\sum_{\substack{0 \sim \Gamma \\ |\Gamma| \geq L}} |\Gamma|\widetilde{\mathbb{P}}(\Gamma \in \mathcal{L}).
\end{align*}
Then by Lemma \ref{loop one point} we know
\begin{align} \label{6.3}
    \sum_{x\in \mathbb{B}_{r}}\widetilde{\mathbb{P}}(0 \overset{\mathbb{B}_{\beta r}}{\longleftrightarrow}x) - \sum_{x\in \mathbb{B}_{r}}\widetilde{\mathbb{P}}(0 \overset{\mathbb{B}_{\beta r},\ \overline{\mathcal{L}}_{\leq L} }{\longleftrightarrow}x) \leq Cr^{4} L^{-d/2+2},
\end{align}
Plugging $L =r^{b}$ for $b >\frac{4}{d-4}$ into (\ref{6.3}) yields the upper bound in (\ref{large loop 3}). The lower bound is trivial.
\end{proof}

\section{Appendix}\label{appendix}

In this section we provide the details for the proofs of several statements that were omitted in the main body of the paper. These are Lemma \ref{lem 5.8}, Lemma \ref{lem 5.4}, Lemma \ref{lem 5.1}, Lemma \ref{lem 5.7} and Lemma \ref{large loops on geo}. We start from the proof of Lemma \ref{lem 5.4}, which deals with the point to point connection.

\subsection{Estimates for point to point case}\label{appendix.1}
\begin{proof}[Proof of Lemma \ref{lem 5.4}]
    Recall that the aim is to prove that for any $\varepsilon$, there exists $K(\varepsilon)>0$ such that for any $x \in \mathbb{Z}^{d}$,
    \begin{align}\label{109}
        \sum_{\substack{\Gamma\ni 0 \\ |\Gamma| \geq K}}\sum_{w\in \Gamma}\widetilde{\mathbb{P}}(\Gamma \in \mathcal{L})\widetilde{\mathbb{P}}(w\leftrightarrow x) \leq \varepsilon |x|^{2-d}.
    \end{align}

We begin by considering the contribution of $w \in \mathbb{B}(x, 2^{l}) \setminus \mathbb{B}(x,2^{l-1})$, $1\leq l\leq \lceil\log_{2} (|x|/4)\rceil$. Using two-point function estimate \eqref{two point}, we know that
\begin{align*}
\widetilde{\mathbb{P}}(w \leftrightarrow x) \leq C2^{l(2-d)}. 
\end{align*}
For $k > |x|^{1.9}$, consider a simple random walk $\{X_{t}\}_{t=0,1,\cdots}$ on $\mathbb{Z}^{d}$. Denote by $\P_{x}$ and $\E_{x}$ the measure and expectation of $\{X_{t}\}_{t\geq 0}$ starting from $X_{0}=x$ and let $p_{t}(\cdot,\cdot)$ be the transition kernel at time $t$. By applying (\ref{loop prob}), we have
\begin{align*}
    \sum_{\substack{\Gamma \ni 0, w\\ |\Gamma|=k}}\widetilde{\mathbb{P}}(\Gamma \in \mathcal{L}) &\overset{(\ref{loop prob})}{\leq} \sum_{\substack{\Gamma \ni 0, w \\ |\Gamma|=k}}(2d)^{-|\Gamma|}  =\sum_{t=0}^{k}\P_{0}(X_{t}=w,\; X_{k}=0) = \sum_{t=0}^{k}p_{t}(0,w)p_{k-t}(w,0).
\end{align*}
From \cite[Theorem 1.2.1]{lawler2013intersections}, we know that
\begin{align*}
    c k^{-\frac{d}{2}} \leq p_{k}(0,0) \leq C k^{-\frac{d}{2}},\qquad 
    p_{t}(0,w) \leq C t^{-\frac{d}{2}}e^{-\frac{d|w|^{2}}{2t}} + C t^{-\frac{d+2}{2}}.
\end{align*}
Thus, for $k/2 \leq t\leq k$,
\begin{align*}
    p_{t}(0,w) \leq C p_{k}(0,0).
\end{align*}
Therefore, we can bound the sum for $k/2 \leq t\leq k$ as
\begin{align*}
    \sum_{k/2 \leq t \leq k} p_{t}(0,w)p_{k-t}(w,0) &\leq C p_{k(0,0)}\sum_{k/2 \leq t \leq k}p_{k-t}(w,0) 
    \leq Cp_{k(0,0)}\sum_{l=1}^{\infty}p_{l}(w,0) \\
    &=Cp_{k}(0,0) G(w,0) \leq Ck^{-\frac{d}{2}}|w|^{2-d},
\end{align*}
where $G(\cdot , \cdot)$ is the Green function for $\{X_{t}\}_{t\geq 0}$ and the last inequality is due to classic Green function estimate.
Similarly, for $1\leq t\leq k/2$, by the symmetry of $p_{t}(\cdot,\cdot)$ we know that, 
\begin{align*}
    \sum_{t= 1}^{k/2} p_{t}(0,w)p_{k-t}(w,0) \leq C k^{-\frac{d}{2}}|w|^{2-d}.
\end{align*}
Given $w \in \mathbb{B}(x,|x|/4)$, we know $|w| \geq |x|/4$. Thus, for $k>|x|^{1.9}$
\begin{align*}
    \sum_{\substack{\Gamma \ni 0, w\\ |\Gamma|=k}}\widetilde{\mathbb{P}}(\Gamma \in \mathcal{L})\leq C k^{-\frac{d}{2}}|x|^{2-d}.
\end{align*}
Now summing over $k>|x|^{1.9}$, we obtain
\begin{align}
   &\sum_{ \substack{\Gamma \ni 0\\ |\Gamma|\geq |x|^{1.9}}}\sum_{\substack{w\in \mathbb{B}(x,|x|/4) \\ w\ni \Gamma}}|\Gamma| \widetilde{\mathbb{P}}(\Gamma \in \mathcal{L})\widetilde{\mathbb{P}}(w\leftrightarrow x) \leq \sum_{k=|x|^{1.9}}^{\infty} \sum_{l=1}^{\lceil\log_{2}(|x|/4)\rceil}\sum_{w\in \mathbb{B}(x,2^{l}) \setminus \mathbb{B}(x,2^{l-1})} \sum_{\substack{\Gamma \sim 0,w \\ |\Gamma|=k}} |\Gamma| \widetilde{\mathbb{P}}(\Gamma \in \mathcal{L}) \mathbb{P}(w\leftrightarrow x) \nonumber\\ &\leq 
   \sum_{k=|x|^{1.9}}^{\infty} \sum_{l=1}^{\lceil\log_{2}(|x|/4)\rceil}\sum_{w\in \mathbb{B}(x,2^{l}) \setminus \mathbb{B}(x,2^{l-1})} k \cdot 2^{l(2-d)}\sum_{\substack{\Gamma \sim 0,w \\ |\Gamma|=k}} \widetilde{\mathbb{P}}(\Gamma \in \mathcal{L}) \nonumber\\
   &\leq C\sum_{k=|x|^{1.9}}^{\infty} \sum_{l=1}^{\lceil\log_{2}(|x|/4)\rceil} k \cdot 2^{l(2-d)} \cdot 2^{ld}\cdot k^{-\frac{d}{2}}|x|^{2-d} \nonumber\\ &\leq 
   C |x|^{2-d}\cdot 2^{2 \lceil \log_{2}(|x|/4)\rceil} \sum_{k=|x|^{1.9}}^{\infty}k^{1-\frac{d}{2}}\leq C|x|^{2-d}\cdot |x|^{2}\sum_{k = |x|^{1.9}}^{\infty}k^{1-\frac{d}{2}} \nonumber\\ 
   &\leq C |x|^{4-d} \cdot |x|^{1.9\cdot (2-d/2)} =C |x|^{2-d} \cdot |x|^{5.8 - 0.95d}.\label{110}
\end{align}
This is a small term compared with $|x|^{2-d}$ since $d\geq 7$ implies that $5.8 - 0.95d< 0$.

Next we bound the summation over $k < |x|^{1.9}$. Applying Lemma \ref{local clt} gives that for any $w \in \mathbb{B}(x,|x|/4)$,
\begin{align*}
    \sum_{\substack{\Gamma \ni 0,w \\ |\Gamma|\leq |x|^{1.9}}}\widetilde{\mathbb{P}}(\Gamma \in \mathcal{L}) \leq C e^{-|w|^{0.2}/4} \leq C e^{-C|x|^{0.2}}.
\end{align*}
Thus, 
\begin{align}
    &\sum_{ \substack{\Gamma \ni 0\\ |\Gamma|\leq |x|^{1.9}}}\sum_{\substack{w\in \mathbb{B}(x,|x|/2) \\ w\in \Gamma}}|\Gamma| \widetilde{\mathbb{P}}(\Gamma \in \mathcal{L})\widetilde{\mathbb{P}}(w\leftrightarrow x) \leq \sum_{w\in \mathbb{B}(x,|x|/2)}\sum_{\substack{\Gamma \ni 0,w \\ |\Gamma|\leq |x|^{1.9}}} |x|^{1.9}\widetilde{\mathbb{P}}(\Gamma \in \mathcal{L}) \nonumber\\ &\leq 
    C|x|^{1.9} \sum_{w\in \mathbb{B}(x,|x|/2)} e^{-C|x|^{0.2}} \leq C|x|^{1.9+d}e^{-C|x|^{0.2}}.\label{111}
\end{align}
This is exponentially small. Combining with (\ref{111}) gives that for $x$ with large enough $|x|$, 
\begin{align}
   &\sum_{ \substack{\Gamma \ni 0\\ |\Gamma|\geq K}}\sum_{\substack{w\in \mathbb{B}(x,|x|/2) \\ w\in \Gamma}}|\Gamma| \widetilde{\mathbb{P}}(\Gamma \in \mathcal{L})\widetilde{\mathbb{P}}(w\leftrightarrow x) \nonumber\\ 
   &\leq  \sum_{ \substack{\Gamma \ni 0\\ |\Gamma|\leq |x|^{1.9}}}\sum_{\substack{w\in \mathbb{B}(x,|x|/2) \\ w\in \Gamma}}|\Gamma| \widetilde{\mathbb{P}}(\Gamma \in \mathcal{L})\widetilde{\mathbb{P}}(w\leftrightarrow x) + \sum_{ \substack{\Gamma \ni 0\\ |\Gamma|\geq |x|^{1.9}}}\sum_{\substack{w\in \mathbb{B}(x,|x|/2) \\ w\in \Gamma}}|\Gamma| \widetilde{\mathbb{P}}(\Gamma \in \mathcal{L})\widetilde{\mathbb{P}}(w\leftrightarrow x) \nonumber\\ 
   &\leq  C|x|^{1.9+d}e^{-C|x|^{0.2}} + C |x|^{2-d} \cdot |x|^{5.8-0.95d} \leq \frac{\e}{2}\cdot |x|^{2-d}.\label{112}
\end{align}

Now we deal with the contribution of $w \notin \mathbb{B}(x, |x|/4)$. For such $w$, by two-point function estimate \eqref{two point}, $\widetilde{\mathbb{P}}(w\leftrightarrow x) \leq C|x|^{2-d}$. Therefore, we have 
\begin{align*}
    \sum_{\substack{\Gamma \sim 0 \\ |\Gamma| \geq K}}\sum_{\substack{w \notin \mathbb{B}(x,|x|/2)\\ w \sim \Gamma }} |\Gamma| \widetilde{\mathbb{P}}(\Gamma \in \mathcal{L})\widetilde{\mathbb{P}}(w \leftrightarrow x) &\leq C |x|^{2-d}\sum_{\substack{\Gamma \sim 0 \\ |\Gamma| \geq K}}\sum_{\substack{w \notin \mathbb{B}(x,|x|/2)\\ w \sim \Gamma }} |\Gamma| \widetilde{\mathbb{P}}(\Gamma \in \mathcal{L}) \\ &\leq C|x|^{2-d}\sum_{\substack{\Gamma \sim 0 \\ |\Gamma| \geq K}}|\Gamma|^{2} \widetilde{\mathbb{P}}(\Gamma \in \mathcal{L}).
\end{align*}
By applying Lemma \ref{loop one point}, we get
\begin{align*}
    \sum_{\substack{\Gamma \ni 0 \\ |\Gamma| \geq K}}|\Gamma|^{2} \widetilde{\mathbb{P}}(\Gamma \in \mathcal{L})\leq C K^{3-d/2}.
\end{align*}
Since $d>6$ implies that $3-d/2<0$, for sufficiently large $K$, 
\begin{align}
    \sum_{\substack{\Gamma \sim 0 \\ |\Gamma| \geq K}}\sum_{\substack{w \notin \mathbb{B}(x,|x|/2)\\ w \sim \Gamma }} |\Gamma| \widetilde{\mathbb{P}}(\Gamma \in \mathcal{L})\widetilde{\mathbb{P}}(w \leftrightarrow x) \leq C K^{3-\frac{d}{2}} |x|^{2-d} \leq \frac{\e}{2} \cdot |x|^{2-d}.\label{113}
\end{align}
Finally, combining (\ref{112}) and (\ref{113}) we get
\begin{align*}
    \sum_{\substack{\Gamma \ni 0 \\ |\Gamma| \geq K}}\sum_{\ w \in \Gamma } |\Gamma| \widetilde{\mathbb{P}}(\Gamma \in \mathcal{L})\widetilde{\mathbb{P}}(w \leftrightarrow x)  &\leq 
    \sum_{\substack{\Gamma \ni 0 \\ |\Gamma| \geq K}}\sum_{\substack{w \notin \mathbb{B}(x,|x|/2)\\ w \in \Gamma }} |\Gamma| \widetilde{\mathbb{P}}(\Gamma \in \mathcal{L})\widetilde{\mathbb{P}}(w \leftrightarrow x) \\& + \sum_{ \substack{\Gamma \ni 0\\ |\Gamma|\geq K}}\sum_{\substack{w\in \mathbb{B}(x,|x|/2) \\ w\in \Gamma}}|\Gamma| \widetilde{\mathbb{P}}(\Gamma \in \mathcal{L})\widetilde{\mathbb{P}}(w\leftrightarrow x) \\ &\leq 
    \frac{\e}{2}\cdot |x|^{2-d} + \frac{\e}{2}\cdot |x|^{2-d} \leq \e |x|^{2-d}.
\end{align*}
\end{proof}

\subsection{Estimates for point to boundary case}\label{appendix.2}

\begin{proof}[Proof of Lemma \ref{lem 5.8}]
    We divide our discussion into two cases:

    \textbf{Case 1: $w \in \mathbb{B}_{\frac{2}{3} r}$.} In this case, $d^{\text{ext}}(w,\partial \mathbb{B}_{r}) \geq \frac{1}{3} r$. If $\{\overline{\Gamma} \overset{\overline{\mathcal{L}}\setminus \{\overline{\Gamma}\},\mathbb{B}_{r}}{\longleftrightarrow} y\}$ occurs, there exists $u \in \mathbb{B}_{r}$ with $u\sim \Gamma$ such that $u \overset{\overline{\mathcal{L}}\setminus \{\overline{\Gamma}\},\mathbb{B}_{r}}{\longleftrightarrow} y$.
    
    First, consider the case $|\Gamma|\geq r^{1.9}$. By a union bound, we have.
    \begin{align*}
        &\sum_{\substack{\Gamma \sim w\\ |\Gamma| \geq r^{1.9}}} \widetilde{\mathbb{P}}(\Gamma \in \mathcal{L})\widetilde{\mathbb{P}}(\overline{\Gamma} \overset{\overline{\mathcal{L}}\setminus \{\overline{\Gamma}\},\mathbb{B}_{r}}{\longleftrightarrow} y\circ \overline{\Gamma} \overset{\overline{\mathcal{L}}\setminus \{\overline{\Gamma}\}}{\longleftrightarrow} \partial\mathbb{B}_{r}) \leq  \sum_{\substack{\Gamma \sim w\\ |\Gamma| \geq r^{1.9}}} \widetilde{\mathbb{P}}(\Gamma \in \mathcal{L})\widetilde{\mathbb{P}}(\overline{\Gamma} \overset{\overline{\mathcal{L}}\setminus \{\overline{\Gamma}\},\mathbb{B}_{r}}{\longleftrightarrow} y) \\ 
        &\leq \sum_{\substack{\Gamma \sim w\\ |\Gamma| \geq r^{1.9}}}\sum_{\substack{u\in \mathbb{B}_{r}\\ u \sim \Gamma}} \widetilde{\mathbb{P}}(\Gamma \in \mathcal{L})\widetilde{\mathbb{P}}(u \leftrightarrow y) \leq 
        \sum_{u \in \mathbb{B}_{r}}\widetilde{\mathbb{P}}(u\leftrightarrow y)\sum_{\substack{\Gamma \sim w,u\\ |\Gamma| \geq r^{1.9}}}\widetilde{\mathbb{P}}(\Gamma \in \mathcal{L}).
    \end{align*}
    Next, swapping the order of summation and applying Lemma \ref{loop two points}, we obtain 
    \begin{align*}
        \sum_{\substack{\Gamma \sim w,u\\ |\Gamma| \geq r^{1.9}}}\widetilde{\mathbb{P}}(\Gamma \in \mathcal{L}) \leq C r^{1.9 \cdot (1-d/2)} |u-w|^{2-d}.
    \end{align*}
    Substituting this into the summation, we get
    \begin{align*}
        &\sum_{u \in \mathbb{B}_{r}}\widetilde{\mathbb{P}}(u\leftrightarrow y)\sum_{\substack{\Gamma \sim w,u\\ |\Gamma| \geq r^{1.9}}}\widetilde{\mathbb{P}}(\Gamma \in \mathcal{L}) \leq Cr^{1.9 \cdot (1-d/2)}\sum_{u\in \mathbb{B}_{r}} |u-y|^{2-d}|u-w|^{2-d} \\ &\overset{(\ref{2.1})}{\leq}
        Cr^{1.9 \cdot (1-d/2)}|y-w|^{4-d} \leq C r^{1.9-0.95d +2} |y-w|^{2-d}.
    \end{align*}
    The last inequality is because $|y-w| \leq 3 \cdot \frac{1}{3} r=r$ implies $|y-w|^{4-d}\leq r^{2}|y-w|^{2-d}$. Notice that since $d \geq 7$, $1.9-0.95d +2 <-2$. Thus, we conclude that  
    \begin{align}
        &\sum_{\substack{\Gamma \sim w\\ |\Gamma| \geq r^{1.9}}} \widetilde{\mathbb{P}}(\Gamma \in \mathcal{L})\widetilde{\mathbb{P}}(\overline{\Gamma} \overset{\overline{\mathcal{L}}\setminus \{\overline{\Gamma}\},\mathbb{B}_{r}}{\longleftrightarrow} y\circ \overline{\Gamma} \overset{\overline{\mathcal{L}}\setminus \{\overline{\Gamma}\}}{\longleftrightarrow} \partial\mathbb{B}_{r}) \leq Cr^{1.9-0.95d +2} |y-w|^{2-d}\nonumber \\ 
        &\leq Cr^{-2} |y-w|^{2-d} \leq C |y-w|^{2-d}d^{\text{ext}}(w,\partial \mathbb{B}_{r})^{-2}.\label{114}
    \end{align}
    The last inequality is because $d^{\text{ext}}(w,\partial \mathbb{B}_{r}) \leq r$.

    Next, for $|\Gamma| \leq r^{1.9}$, if $\{\overline{\Gamma} \overset{\overline{\mathcal{L}}\setminus \{\overline{\Gamma}\},\mathbb{B}_{r}}{\longleftrightarrow} y\circ \overline{\Gamma} \overset{\overline{\mathcal{L}}\setminus \{\overline{\Gamma}\}}{\longleftrightarrow} \partial\mathbb{B}_{r}\}$ occurs, there exists $u,v\in \mathbb{B}_{r}$ with $u,v \sim \Gamma$ such that
    \begin{align*}
        y\overset{\overline{\mathcal{L}}\setminus \{\overline{\Gamma}\},\mathbb{B}_{r}}{\longleftrightarrow} u \circ v \overset{\overline{\mathcal{L}}\setminus \{\overline{\Gamma}\}}{\longleftrightarrow} \partial \mathbb{B}_{r}.
    \end{align*}
    Then using BKR inequality, we have
    \begin{align*}
        \sum_{\substack{\Gamma \sim w\\ |\Gamma| \leq r^{1.9}}} \widetilde{\mathbb{P}}(\Gamma \in \mathcal{L})\widetilde{\mathbb{P}}(\overline{\Gamma} \overset{\overline{\mathcal{L}}\setminus \{\overline{\Gamma}\},\mathbb{B}_{r}}{\longleftrightarrow} y\circ \overline{\Gamma} \overset{\overline{\mathcal{L}}\setminus \{\overline{\Gamma}\}}{\longleftrightarrow} \partial\mathbb{B}_{r}) \leq \sum_{\substack{\Gamma \sim w \\ |\Gamma| \leq r^{1.9}}}\sum_{\substack{u,v\in \mathbb{B}_{r}\\ u,v\sim \Gamma}}\widetilde{\mathbb{P}}(\Gamma \in \mathcal{L})\widetilde{\mathbb{P}}(y\leftrightarrow u) \widetilde{\mathbb{P}}(v\leftrightarrow \partial \mathbb{B}_{r}).
    \end{align*}
    When $v \in \mathbb{B}_{r} \setminus \mathbb{B}_{\frac{3}{4}r}$, $|u-v| \geq (3/4-2/3)r$. By Lemma \ref{local clt},
    \begin{align*}
        \sum_{\substack{w,v\sim \Gamma\\ |\Gamma| \leq |u-v|^{1.9}}}\widetilde{\mathbb{P}}(\Gamma \in \mathcal{L})  \leq C e^{-|u-v|^{0.1}/4} \leq Ce^{-Cr^{0.1}}.
    \end{align*}
    Thus, 
    \begin{align}
        &\sum_{\substack{\Gamma \sim w \\ |\Gamma| \leq r^{1.9}}}\sum_{\substack{u\in \mathbb{B}_{r}\\ v\in \mathbb{B}_{r}\setminus \mathbb{B}_{3/4 r}\\ u,v\sim \Gamma}}\widetilde{\mathbb{P}}(\Gamma \in \mathcal{L})\widetilde{\mathbb{P}}(y\leftrightarrow u) \widetilde{\mathbb{P}}(v\leftrightarrow \partial \mathbb{B}_{r}) \leq \sum_{v\in \mathbb{B}_{r}\setminus \mathbb{B}_{3/4 r}}\sum_{\substack{\Gamma \sim w,v \\ |\Gamma | \leq r^{1.9}}}\sum_{u\sim \Gamma}\widetilde{\mathbb{P}}(\Gamma \in \mathcal{L}) \nonumber\\ 
        &\leq \sum_{v\in \mathbb{B}_{r}\setminus \mathbb{B}_{3/4 r}}\sum_{\substack{\Gamma \sim w,v \\ |\Gamma | \leq r^{1.9}}} |\Gamma| \widetilde{\mathbb{P}}(\Gamma \in \mathcal{L}) \leq Cr^{1.9} \sum_{v\in \mathbb{B}_{r}\setminus \mathbb{B}_{3/4 r}}\sum_{\substack{\Gamma \sim w,v \\ |\Gamma | \leq r^{1.9}}} \widetilde{\mathbb{P}}(\Gamma \in \mathcal{L}) \nonumber\\ 
        &\leq Cr^{1.9} \sum_{v\in \mathbb{B}_{r}\setminus \mathbb{B}_{3\alpha r}}e^{-C r^{0.1}} \leq Cr^{1.9+d} e^{-C r^{0.1}}.\label{115}
    \end{align}
    For $v\in \mathbb{B}_{\frac{3}{4} r}$, $d^{\text{ext}}(v, \partial \mathbb{B}_{r}) \geq \frac{1}{4}r$ and hence $\widetilde{\mathbb{P}}(v\leftrightarrow \partial \mathbb{B}_{r}) \leq Cr^{-2}.$
    By Lemma \ref{loop three points},
    \begin{align}
        &\sum_{\substack{\Gamma \sim w \\ |\Gamma| \leq r^{1.9}}}\sum_{\substack{u\in \mathbb{B}_{r}\\ v\in \mathbb{B}_{r}\setminus \mathbb{B}_{3/4 r}\\ u,v\sim \Gamma}}\widetilde{\mathbb{P}}(\Gamma \in \mathcal{L})\widetilde{\mathbb{P}}(y\leftrightarrow u) \leq \sum_{u,v\in \Z^{d}}\sum_{\Gamma \sim u,v,w}\widetilde{\mathbb{P}}(\Gamma \in \mathcal{L}) \widetilde{\mathbb{P}}(y\leftrightarrow u)\nonumber\\ 
        &\leq 
        C\sum_{u,v\in \Z^{d}}|u-v|^{2-d}|v-w|^{2-d}|w-u|^{2-d} |y-u|^{2-d} \nonumber\\ 
        &\overset{(\ref{2.1})}{\leq} 
        C \sum_{u\in \Z^{d}} |u-w|^{6-2d}|u-y|^{2-d} \overset{(\ref{2.1})}{\leq} C|w-y|^{8-2d} \leq C|w-y|^{2-d}.\label{116}
    \end{align}
    Thus, combining (\ref{115}) and (\ref{116}) we know that
    \begin{align}
        &\sum_{\substack{\Gamma \sim w \\ |\Gamma| \leq r^{1.9}}}\sum_{\substack{u,v\in \mathbb{B}_{r}\\ u,v\sim \Gamma}}\widetilde{\mathbb{P}}(\Gamma \in \mathcal{L})\widetilde{\mathbb{P}}(y\leftrightarrow u) \widetilde{\mathbb{P}}(v\leftrightarrow \partial \mathbb{B}_{r}) \nonumber\\ &\leq Cr^{1.9+d} e^{-C r^{0.1}} + C|u-y|^{2-d}r^{-2} \leq C|u-y|^{2-d} r^{-2}.\label{117}
    \end{align}
    Summing (\ref{114}) and (\ref{117}) we get for $w\in \mathbb{B}_{2/3 r}$, 
    \begin{align*}
        \sum_{\Gamma \sim w} \widetilde{\mathbb{P}}(\Gamma \in \mathcal{L})\widetilde{\mathbb{P}}(\overline{\Gamma} \overset{\overline{\mathcal{L}}\setminus \{\overline{\Gamma}\},\mathbb{B}_{r}}{\longleftrightarrow} y\circ \overline{\Gamma} \overset{\overline{\mathcal{L}}\setminus \{\overline{\Gamma}\}}{\longleftrightarrow} \partial\mathbb{B}_{r}) \leq C |y-w|^{2-d}r^{-2} \leq C |y-w|^{2-d}d^{\text{ext}}(w,\partial \mathbb{B}_{r})^{-2}.
    \end{align*}

    \textbf{Case 2: $w \in \mathbb{B}_{r} \setminus \mathbb{B}_{2/3 r}$.} For $w \in \mathbb{B}_{r} \setminus \mathbb{B}_{2/3 r}$, $|w-y|\geq \frac{1}{3} r$. First by BKR inequality,
    \begin{align*}
        \sum_{\Gamma \sim w} \widetilde{\mathbb{P}}(\Gamma \in \mathcal{L})\widetilde{\mathbb{P}}(\overline{\Gamma} \overset{\overline{\mathcal{L}}\setminus \{\overline{\Gamma}\},\mathbb{B}_{r}}{\longleftrightarrow} y\circ \overline{\Gamma} \overset{\overline{\mathcal{L}}\setminus \{\overline{\Gamma}\}}{\longleftrightarrow} \partial\mathbb{B}_{r}) &\leq \sum_{\Gamma \sim w} \widetilde{\mathbb{P}}(\Gamma \in \mathcal{L})\widetilde{\mathbb{P}}(\overline{\Gamma} \overset{\overline{\mathcal{L}}\setminus \{\overline{\Gamma}\},\mathbb{B}_{r}}{\longleftrightarrow} y)\widetilde{\mathbb{P}}( \overline{\Gamma} \overset{\overline{\mathcal{L}}\setminus \{\overline{\Gamma}\}}{\longleftrightarrow} \partial\mathbb{B}_{r}) \\ &\leq 
        \sum_{k=1}^{\infty}\sum_{\substack{\Gamma \sim w \\ |\Gamma|=k}}\widetilde{\mathbb{P}}(\Gamma \in \mathcal{L})\widetilde{\mathbb{P}}(\overline{\Gamma} \overset{\overline{\mathcal{L}}\setminus \{\overline{\Gamma}\},\mathbb{B}_{r}}{\longleftrightarrow} y)\widetilde{\mathbb{P}}( \overline{\Gamma} \overset{\overline{\mathcal{L}}\setminus \{\overline{\Gamma}\}}{\longleftrightarrow} \partial\mathbb{B}_{r}).
    \end{align*}
    
    For $k >r^{1.9}$, we naively bound  $\widetilde{\mathbb{P}}( \overline{\Gamma} \overset{\overline{\mathcal{L}}\setminus \{\overline{\Gamma}\}}{\longleftrightarrow} \partial\mathbb{B}_{r})$ by 1, and hence
    \begin{align*}
        \sum_{\substack{\Gamma \sim w \\ |\Gamma|\geq r^{1.9}}} \widetilde{\mathbb{P}}(\Gamma \in \mathcal{L})\widetilde{\mathbb{P}}(\overline{\Gamma} \overset{\overline{\mathcal{L}}\setminus \{\overline{\Gamma}\},\mathbb{B}_{r}}{\longleftrightarrow} y)\widetilde{\mathbb{P}}( \overline{\Gamma} \overset{\overline{\mathcal{L}}\setminus \{\overline{\Gamma}\}}{\longleftrightarrow} \partial\mathbb{B}_{r}) \leq \sum_{\substack{\Gamma \sim w \\ |\Gamma|\geq r^{1.9}}}\widetilde{\mathbb{P}}(\Gamma \in \mathcal{L})\widetilde{\mathbb{P}}(\overline{\Gamma} \overset{\overline{\mathcal{L}}\setminus \{\overline{\Gamma}\},\mathbb{B}_{r}}{\longleftrightarrow} y).
    \end{align*}
    If $\{\overline{\Gamma} \overset{\overline{\mathcal{L}}\setminus \{\overline{\Gamma}\},\mathbb{B}_{r}}{\longleftrightarrow} y\}$ occurs, there exists $u\sim \Gamma$ such that $u\overset{\overline{\mathcal{L}}\setminus \{\overline{\Gamma}\}}{\longleftrightarrow} y$. Then by a union bound and Lemma \ref{loop two points}, we obtain
    \begin{align*}
        \sum_{\substack{\Gamma \sim w \\ |\Gamma|\geq r^{1.9}}}\widetilde{\mathbb{P}}(\Gamma \in \mathcal{L})\widetilde{\mathbb{P}}(\overline{\Gamma} \overset{\overline{\mathcal{L}}\setminus \{\overline{\Gamma}\},\mathbb{B}_{r}}{\longleftrightarrow} y) &\leq  \sum_{u\in \Z^{d}}\sum_{\substack{\Gamma\sim w,u \\ |\Gamma| \geq r^{1.9}}} \widetilde{\mathbb{P}}(\Gamma\in \mathcal{L})\widetilde{\mathbb{P}}(u\leftrightarrow y) \\ &\leq C\sum_{u\in \Z^{d}} r^{1.9 \cdot (1-d/2)}|u-w|^{2-d}|u-y|^{2-d} \\ 
        &\overset{(\ref{2.1})}{\leq}Cr^{1.9\cdot (1-d/2)}|w-y|^{4-d} \leq C r^{1.9-0.95d+2}|w-y|^{2-d}.
    \end{align*}
    The last inequality is because $|w-y|\leq 2r$. Since $d \geq 7$, $1.9-0.95d+2 <-2$. So this term can be further bounded by
    \begin{align}
        \sum_{\substack{\Gamma \sim w \\ |\Gamma|\geq r^{1.9}}} \widetilde{\mathbb{P}}(\Gamma \in \mathcal{L})\widetilde{\mathbb{P}}(\overline{\Gamma} \overset{\overline{\mathcal{L}}\setminus \{\overline{\Gamma}\},\mathbb{B}_{r}}{\longleftrightarrow} y\circ \overline{\Gamma} \overset{\overline{\mathcal{L}}\setminus \{\overline{\Gamma}\}}{\longleftrightarrow} \partial\mathbb{B}_{r}) \leq C d^{\text{ext}}(w,\partial \mathbb{B}_{r})^{-2}|w-y|^{2-d}. \label{118}
    \end{align}
    The last inequality is because $d^{\text{ext}}(w,\partial \mathbb{B}_{r}) \leq r$.

    For the contribution of $k<r^{1.9}$, we can employ the similar computation as in \eqref{111} to get
    \begin{align}\label{119}
        \sum_{\substack{\Gamma \sim w \\ |\Gamma|<r^{1.9}}} \widetilde{\mathbb{P}}(\Gamma \in \mathcal{L})\widetilde{\mathbb{P}}(\overline{\Gamma} \overset{\overline{\mathcal{L}}\setminus \{\overline{\Gamma}\},\mathbb{B}_{r}}{\longleftrightarrow} y\circ \overline{\Gamma} \overset{\overline{\mathcal{L}}\setminus \{\overline{\Gamma}\}}{\longleftrightarrow} \partial\mathbb{B}_{r}) \leq C d^{\text{ext}}(w,\partial \mathbb{B}_{r})^{-2}|w-y|^{2-d}.
    \end{align}

    Finally, combining \eqref{118} and \eqref{119} completes the proof for case 2, i.e. $w \in \mathbb{B}_{r} \setminus \mathbb{B}_{2/3 r}$.
\end{proof}

\begin{proof}[Proof of Lemma \ref{large loops on geo}]
First, using a similar computation as in (\ref{114}), we obtain
    \begin{align}\label{5-61}
        \sum_{z\in \mathbb{B}_{1/3 r}} \sum_{\substack{\Gamma \ni z\\|\Gamma|\geq r^{1.9}}}\widetilde{\mathbb{P}}(\Gamma\in \mathcal{L})\widetilde{\mathbb{P}}(0\overset{\overline{\mathcal{L}}\setminus \{\overline{\Gamma}\},\mathbb{B}_{r}}{\longleftrightarrow} \overline{\Gamma} \circ \overline{\Gamma}\overset{\overline{\mathcal{L}}\setminus \{\overline{\Gamma}\}}{\longleftrightarrow} \partial \mathbb{B}_{r}) \leq C\sum_{z\in \mathbb{B}_{1/3 r}}r^{3.9-0.95d}|z|^{2-d} \leq C r^{5.9-0.95d}.
    \end{align}
    Next, applying Lemma \ref{local clt}, we have
    \begin{align}\label{5-62}
        \sum_{z\in\mathbb{B}_{1/3 r}}\sum_{u \in \mathbb{B}_{r}\setminus \mathbb{B}_{2/3 r}}\sum_{\substack{\Gamma \in z\\ \Gamma \sim w \\ |\Gamma| < r^{1.9}}}\widetilde{\mathbb{P}}(\Gamma \in \mathcal{L}) \widetilde{\mathbb{P}}(0\overset{\overline{\mathcal{L}}\setminus \{\overline{\Gamma}\},\mathbb{B}_{r}}{\longleftrightarrow} \overline{\Gamma})\widetilde{\mathbb{P}}(u\leftrightarrow \partial \mathbb{B}_{r}) &\leq \sum_{z\in\mathbb{B}_{1/3 r}}\sum_{u \in \mathbb{B}_{r}\setminus \mathbb{B}_{2/3 r}}\sum_{\substack{\Gamma \in z\\ \Gamma \sim w \\ |\Gamma| < r^{1.9}}}\widetilde{\mathbb{P}}(\Gamma \in \mathcal{L}) \nonumber\\ 
        &\leq Cr^{2d}e^{-Cr^{0.1}}.
    \end{align}
    For $u\in \mathbb{B}_{\frac{2}{3} r}$, $d^{\text{ext}}(u,\partial \mathbb{B}_{r})\geq \frac{1}{3} r$, so $\widetilde{\mathbb{P}}(u\leftrightarrow \partial \mathbb{B}_{r}) \leq Cr^{-2}$. Thus, we have he following bound:
    \begin{align*}
        &\sum_{z\in \mathbb{B}_{1/3 r}}\sum_{\substack{\Gamma\ni z \\ K\leq |\Gamma|<r^{1.9}}}\sum_{\substack{u\in \mathbb{B}_{2/3 r}\\u\sim \Gamma}}\widetilde{\mathbb{P}}(\Gamma \in \mathcal{L})\widetilde{\mathbb{P}}(0\overset{\overline{\mathcal{L}}\setminus \{\overline{\Gamma}\},\mathbb{B}_{r}}{\longleftrightarrow} \overline{\Gamma})\widetilde{\mathbb{P}}(u\leftrightarrow \partial \mathbb{B}_{r}) \\ &\leq  Cr^{-2}\sum_{z\in \mathbb{B}_{\alpha r}}\sum_{\substack{\Gamma\ni z \\ K\leq |\Gamma|<r^{1.9}}}|\Gamma| \widetilde{\mathbb{P}}(\Gamma \in \mathcal{L}) \widetilde{\mathbb{P}}(0\overset{\overline{\mathcal{L}}\setminus \{\overline{\Gamma}\},\mathbb{B}_{r}}{\longleftrightarrow} \overline{\Gamma}).
    \end{align*}
    By Lemma \ref{lem 5.4}, for large $K$,
    \begin{align*}
        \sum_{\substack{\Gamma\ni z \\ K\leq |\Gamma|<r^{1.9}}}|\Gamma| \widetilde{\mathbb{P}}(\Gamma \in \mathcal{L}) \widetilde{\mathbb{P}}(0\overset{\overline{\mathcal{L}}\setminus \{\overline{\Gamma}\},\mathbb{B}_{r}}{\longleftrightarrow} \overline{\Gamma}) \leq \e |z|^{2-d}.
    \end{align*} 
    Thus, we get 
    \begin{align}\label{5-63}
        &\sum_{z\in \mathbb{B}_{1/3 r}}\sum_{\substack{\Gamma\ni z \\ K\leq |\Gamma|<r^{1.9}}}\sum_{\substack{u\in \mathbb{B}_{2/3 r}\\u\sim \Gamma}}\widetilde{\mathbb{P}}(\Gamma \in \mathcal{L})\widetilde{\mathbb{P}}(0\overset{\overline{\mathcal{L}}\setminus \{\overline{\Gamma}\},\mathbb{B}_{r}}{\longleftrightarrow} \overline{\Gamma})\widetilde{\mathbb{P}}(u\leftrightarrow \partial \mathbb{B}_{r}) \nonumber\\ &\leq C\e r^{-2}\sum_{z\in \mathbb{B}_{\alpha r}}|z|^{2-d}  \leq C\e r^{-2} \cdot r^{2} \leq C\e
    \end{align}
    Finally, combining (\ref{5-61})-(\ref{5-63}), we get 
    \begin{align*}
        &\sum_{z\in \mathbb{B}_{1/3 r}} \sum_{\substack{\Gamma \ni z\\|\Gamma| \geq K}}\widetilde{\mathbb{P}}(\Gamma\in \mathcal{L})\widetilde{\mathbb{P}}(0\overset{\overline{\mathcal{L}}\setminus \{\overline{\Gamma}\},\mathbb{B}_{r}}{\longleftrightarrow} \overline{\Gamma} \circ \overline{\Gamma}\overset{\overline{\mathcal{L}}\setminus \{\overline{\Gamma}\}}{\longleftrightarrow} \partial \mathbb{B}_{r}) \leq \sum_{z\in \mathbb{B}_{1/3 r}} \sum_{\substack{\Gamma \ni z\\|\Gamma|\geq r^{1.9}}}\widetilde{\mathbb{P}}(\Gamma\in \mathcal{L})\widetilde{\mathbb{P}}(0\overset{\overline{\mathcal{L}}\setminus \{\overline{\Gamma}\},\mathbb{B}_{r}}{\longleftrightarrow} \Gamma \circ \overset{\overline{\mathcal{L}}\setminus \{\overline{\Gamma}\}}{\longleftrightarrow} \partial \mathbb{B}_{r}) 
        \\ &+
        \sum_{z\in \mathbb{B}_{1/3 r}}\sum_{\substack{\Gamma \ni z\\ K <|\Gamma|\leq r^{1.9}}}\sum_{\substack{u \in \mathbb{B}_{r}\setminus \mathbb{B}_{2/3 r}\\ u\sim \Gamma}}\widetilde{\mathbb{P}}(\Gamma \in \mathcal{L})\widetilde{\mathbb{P}}(0\leftrightarrow \Gamma)\widetilde{\mathbb{P}}(u\leftrightarrow \mathbb{B}_{r})  \\ &+ \sum_{z\in \mathbb{B}_{1/3 r}}\sum_{\substack{\Gamma\ni z \\ K\leq |\Gamma|<r^{1.9}}}\sum_{\substack{u\in \mathbb{B}_{2/3 r}\\u\sim \Gamma}}\widetilde{\mathbb{P}}(\Gamma \in \mathcal{L})\widetilde{\mathbb{P}}(0\leftrightarrow \Gamma) \widetilde{\mathbb{P}}(u\leftrightarrow \partial \mathbb{B}_{r}) \\ &\leq 
        C r^{5.9 - 0.95d} + C r^{2d}e^{-Cr^{0.1}} +C\e.
    \end{align*}
    Since $d>6$ implies $5.95-0.95d <0$, we know that for $r$ large, the sum is bounded as
    \begin{align*}
        C r^{5.9 - 0.95d} + C r^{2d}e^{-Cr^{0.1}} +C\e \leq C^{\prime} \e,
    \end{align*} 
    which proves (\ref{5-59}).
    \end{proof}

\subsection{Proof of Lemma \ref{lem 5.1}}\label{appendix.3}
    An intrinsic version of this statement was proved in \cite[Section 4.1]{ganguly2024ant} for $d>20$. However, since our expressions only involve extrinsic constraints, we will be able to achieve the desired estimate for all $d>6$. Our proof strategy, nonetheless, will closely mimic that in \cite[Section 4.1]{ganguly2024ant}. Hence, we will be brief in our exposition, often referring the reader to \cite[Section 4.1]{ganguly2024ant} for omitted details.

    For $\widetilde{A}\in \widetilde{\mathbb{Z}}^{d}$, define $\widetilde{A}^{\text{int}}$ to be a subgraph of $(\mathbb{Z}^{d},E(\mathbb{Z}^{d}))$ obtained by deleting all partial edges in $\widetilde{A}$, i.e. edges $e$ such that $\widetilde{A} \cap I_{e}  \neq I_{e}$, recall that $I_{e}$ is the interval in $\mathbb{Z}^{d}$ corresponding to $e$. For any connected subgraph $S$ of $(\mathbb{Z}^{d},E(\mathbb{Z}^{d}))$, let $S^{\text{out}}$ be the subgraph of $(\mathbb{Z}^{d},E(\mathbb{Z}^{d}))$, obtained by adding all the edges in $E(\Z^{d})$ intersecting $S$. Then define $\mathsf{loop}(S)$ be the collection of all loops on $\widetilde{\Z}^{d}$ (not necessarily in $\widetilde{\mathcal{L}}$) intersecting $S^{\text{out}}$.
    
For any $x,y\in \mathbb{Z}^{d}$, we say $x\overset{\widetilde{A}}{\longleftrightarrow}y$ if there exists a lattice path $\ell$ from $x$ to $y$, only using full edges in $\widetilde{A}$. We say $x \overset{\widetilde{A}}{\nleftrightarrow}y$ otherwise.  By definition,
\begin{align*}
    x\overset{\widetilde{A}}{\longleftrightarrow}y \Leftrightarrow x\overset{\widetilde{A}^{\text{int}}}{\longleftrightarrow}y.
\end{align*}

    It will be convenient to work with $\widetilde{\mathbb{P}}(0\leftrightarrow y+u \circ y+v \leftrightarrow x)$ where $u=(-K,0,\cdots,0)$ and $v=(K,0,\cdots 0)$ for some large but fixed $K$. We claim that for any $K\in \mathbb{N}$, there exists $c(K)>0$ such that for any $x,y\in \mathbb{Z}^{d}$,
    \begin{align*}
        \widetilde{\mathbb{P}}(\mathcal{T}_{x}(y)) \geq c(K)\widetilde{\mathbb{P}}(0\leftrightarrow y+u \circ y+v \leftrightarrow x).
    \end{align*}
    This claim follows from a resampling argument carried out in \cite[Section 4.1]{ganguly2024ant}. 
    
    Next, we bound the probability that $0\leftrightarrow y+u$ and $y+v \leftrightarrow x$ and further that $y+u$ and $y+v$ are not connected.
    
    We say a subgraph $S$ of $(\mathbb{Z}^{d},E(\mathbb{Z}^{d}))$ is \textbf{$z$-admissible} for some $z\in \mathbb{Z}^{d}$ if $\widetilde{\mathbb{P}}(\widetilde{\mathcal{C}}^{\text{int}}(z)=S) > 0$, where we recall that $\widetilde{\mathcal{C}}(z)$ is the connected component in $\widetilde{\mathcal{L}}$ containing $z$. We then have the following inequality
    \begin{align}\label{5:54}
        \widetilde{\mathbb{P}}(0\leftrightarrow y+u \circ y+v \leftrightarrow x) &\geq \widetilde{\mathbb{P}}(0\leftrightarrow y+u, \; y+v \leftrightarrow x,\; y+v \notin \widetilde{\mathcal{C}}^{\text{int}}(y+u)) \nonumber\\ 
        &= \sum_{\substack{S:  y+u \text{-admissible}\\ 0\overset{S}{\longleftrightarrow}y+u,\; y+v \overset{S}{\nleftrightarrow}y+v}}\widetilde{\mathbb{P}}(y+v \leftrightarrow x\mid \widetilde{\mathcal{C}}^{\text{int}}(y+u) =S)\widetilde{\mathbb{P}}(\widetilde{\mathcal{C}}^{\text{int}}(y+u) =S).
    \end{align}
    Running a similar argument as in \cite[Section 4.1]{ganguly2024ant}, we arrive at 
    \begin{align}
       \widetilde{\mathbb{P}}(0\leftrightarrow y+u &\circ y+v \leftrightarrow x) \geq  \widetilde{\mathbb{P}}(0 \leftrightarrow y+u)\widetilde{\mathbb{P}}(y+v \leftrightarrow x) \nonumber\\
       &- \sum_{\substack{S:  y+u \text{-admissible}\\ 0\overset{S}{\longleftrightarrow}y+u}}\widetilde{\mathbb{P}}(y+v \leftrightarrow x \text{ only on } \mathsf{loop}(S))\widetilde{\mathbb{P}}(\widetilde{\mathcal{C}}^{\text{int}}(y+u) =S).\label{5-64}
    \end{align}
    
    The next step is to bound $\widetilde{\mathbb{P}}(y+v \leftrightarrow x \text{ only on } \mathsf{loop}(S))$. We note that, on the event $\{y+v \leftrightarrow x \text{ only on } \mathsf{loop}(S)\}$, there exists a path from $y+v$ to $x$ that uses a loop in $\mathsf{loop}(S)$, i.e. there exists a discrete loop $\Gamma\in \mathcal{L}$ intersecting $S^{\text{out}}$, and $w_{1},w_{2}\in \mathbb{Z}^{d}$ with $w_{i} \sim \Gamma$ such that 
    \begin{align*}
        w_{1} \overset{\overline{\mathcal{L}}\setminus \{\overline{\Gamma}\}}{\longleftrightarrow} y+v \circ w_{2} \overset{\overline{\mathcal{L}}\setminus \{\overline{\Gamma}\}}{\longleftrightarrow} x.
    \end{align*}
    Notice that the above connections occur disjointly. This is a consequence of Lemma \ref{simlpe chain lemma}, which guarantees the existence of a simple chain from $y+v$ to $x$. The glued loops along this simple chain exhibit this disjoint connection property. Applying the union bound and the BKR inequality, we obtain the following estimate
    \begin{align*}
        \widetilde{\mathbb{P}}(y+v \leftrightarrow x \text{ only on } \mathsf{loop}(S)) \leq \sum_{\Gamma \in \mathsf{loop}(S)}\sum_{w_{1},w_{2} \sim\Gamma}\widetilde{\mathbb{P}}(\Gamma \in \mathcal{L})\widetilde{\mathbb{P}}(w_{1}\leftrightarrow y+v)\widetilde{\mathbb{P}}(w_{2}\leftrightarrow x).
    \end{align*}
    Hence, for the second term, by interchanging the summation over $S$ and $\Gamma$, we get
    \begin{align*}
        &\sum_{\substack{S:  y+u \text{-admissible}\\ 0\overset{S}{\longleftrightarrow}y+u}}\widetilde{\mathbb{P}}(y+v \leftrightarrow x \text{ only on } \mathsf{loop}(S))\widetilde{\mathbb{P}}(\widetilde{\mathcal{C}}^{\text{int}}(y+u) =S) \\ &\leq \sum_{\Gamma}\sum_{w_{1},w_{2}\sim \Gamma}\widetilde{\mathbb{P}}(\Gamma \in \mathcal{L}) \widetilde{\mathbb{P}}(w_{1}\leftrightarrow y+v)\widetilde{\mathbb{P}}(w_{2}\leftrightarrow x) \sum_{\substack{S:  y+u \text{-admissible}\\ 0\overset{S}{\longleftrightarrow}y+u\\ \mathsf{loop}(S)\ni \Gamma}}\widetilde{\mathbb{P}}(\widetilde{\mathcal{C}}^{\text{int}}(y+u) =S)
    \end{align*}
    For the last summation over $S$, we have the following upper bound:
    \begin{align*}
       \sum_{\substack{S:  y+u \text{-admissible}\\ 0\overset{S}{\longleftrightarrow}y+u\\ \mathsf{loop}(S)\ni \Gamma}}\widetilde{\mathbb{P}}(\widetilde{\mathcal{C}}^{\text{int}}(y+u) =S) &\leq \widetilde{\mathbb{P}}( \exists w_{3} \sim \Gamma \text{ such that } 0\leftrightarrow y+u,\; 0\leftrightarrow w_{3}) \\ &\leq \sum_{w_{3}\sim\Gamma}\widetilde{\mathbb{P}}(0\leftrightarrow y+u,\; 0\leftrightarrow w_{3}).
    \end{align*}
    Using Lemma \ref{ete}, when $\{0\leftrightarrow y+u,\; 0\leftrightarrow w_{3}\}$ occurs, there exist a discrete loop $\Gamma^{\prime}$ and $s_{1},s_{2},s_{3}\in \mathbb{Z}^{d}$ with $s_{i}\sim \Gamma^{\prime}$ such that 
    \begin{align*}
        \Gamma^{\prime}\in \mathcal{L} \circ s_{1}\overset{\overline{\mathcal{L}}\setminus \{\overline{\Gamma}^{\prime}\}}{\longleftrightarrow }y+u\circ s_{2}  \overset{\overline{\mathcal{L}}\setminus \{\overline{\Gamma}^{\prime}\}}{\longleftrightarrow } 0 \circ s_{3}\overset{\overline{\mathcal{L}}\setminus \{\overline{\Gamma}^{\prime}\}}{\longleftrightarrow } w_{3}.
    \end{align*}
    By Lemma \ref{loop three points}, we can bound the sum over $\Gamma^{\prime}$ as follows:
    \begin{align*}
        \sum_{\substack{\Gamma^{\prime} \sim s_{i}\\ i=1,2,3}}\widetilde{\mathbb{P}}(\Gamma^{\prime} \in \mathcal{L}) \leq C |s_{1}-s_{2}|^{2-d}|s_{2}-s_{3}|^{2-d}|s_{3}-s_{1}|^{2-d}.
    \end{align*}
    Therefore, by applying the BKR inequality and a union bound, we get the following upper bound for the probability of $\{0\leftrightarrow y+u,\; 0\leftrightarrow w_{3}\}$:
    \begin{align*}
        &\widetilde{\mathbb{P}}(0\leftrightarrow y+u,\; 0\leftrightarrow w_{3}) \leq 
        \sum_{s_{1},s_{2},s_{3} \in \mathbb{Z}^{d}} \sum_{\substack{\Gamma^{\prime} \sim s_{i}\\ i=1,2,3}}\widetilde{\mathbb{P}}(\Gamma^{\prime} \in \mathcal{L}) \widetilde{\mathbb{P}}(s_{1}\leftrightarrow y+u) \widetilde{\mathbb{P}}(s_{2}\leftrightarrow 0) \widetilde{\mathbb{P}}(s_{3}\leftrightarrow w_{3}) \\ & \leq 
        C\sum_{s_{1},s_{2},s_{3} \in \mathbb{Z}^{d}}|s_{1}-s_{2}|^{2-d}|s_{2}-s_{3}|^{2-d}|s_{3}-s_{2}|^{2-d}|s_{1}-y-u|^{2-d}|s_{2}|^{2-d}|s_{3}-w_{3}|^{2-d}.
    \end{align*}
    Using the estimate (\ref{5:53}), we simplify the sum, which yields 
    \begin{align*}
        &\sum_{s_{1},s_{2} \in \mathbb{Z}^{d}}|s_{1}-s_{2}|^{2-d}|s_{2}-s_{3}|^{2-d}|s_{3}-s_{2}|^{2-d}|s_{1}-y-u|^{2-d}|s_{2}|^{2-d} \leq C|s_{3}|^{2-d}|s_{3}-y-u|^{2-d}.
    \end{align*}
    Thus, we can bound the full sum as follows 
    \begin{align*}
        &\sum_{\substack{S:  y+u \text{-admissible}\\ 0\overset{S}{\longleftrightarrow}y+u}}\widetilde{\mathbb{P}}(y+v \leftrightarrow x \text{ only on } \mathsf{loop}(S))\widetilde{\mathbb{P}}(\widetilde{\mathcal{C}}^{\text{int}}(y+u) =S) \nonumber\\ &\leq C\sum_{\Gamma}\sum_{w_{1},w_{2},w_{3}\sim \Gamma}\sum_{s_{3}\in \Z^{d}}\widetilde{\mathbb{P}}(\Gamma \in \mathcal{L})|w_{1}-y-v|^{2-d}|w_{2}-x|^{2-d}|s_{3}-w_{3}|^{2-d} |s_{3}|^{2-d}|s_{3}-y-u|^{2-d}.
    \end{align*}
    \begin{align}
       \qquad \cdot |w_{2}-x|^{2-d}|s_{3}-w_{3}|^{2-d} |s_{3}|^{2-d}|s_{3}-y-u|^{2-d}\label{5-65}
    \end{align}
    This can be rewritten as 
    \begin{align*}
        \sum_{w_{1},w_{2},w_{3}\in \Z^{d}}\sum_{\Gamma\sim w_{1},w_{2},w_{3}}\sum_{s_{3}\in \Z^{d}}\widetilde{\mathbb{P}}(\Gamma \in \mathcal{L})|w_{1}-y-v|^{2-d}|w_{2}-x|^{2-d}|s_{3}-w_{3}|^{2-d} |s_{3}|^{2-d}|s_{3}-y-u|^{2-d}.
    \end{align*}
    By Lemma \ref{loop three points} we obtain the further upper bound: 
    \begin{align*}
        \sum_{w_{1},w_{2},w_{3}\in \Z^{d}}\sum_{s_{3}\in \Z^{d}}&|w_{1}-w_{2}|^{2-d}|w_{2}-w_{3}|^{2-d}|w_{3}-w_{1}|^{2-d}\\
        &|w_{1}-y-v|^{2-d}|w_{2}-x|^{2-d}|s_{3}-w_{3}|^{2-d} |s_{3}|^{2-d}|s_{3}-y-u|^{2-d}.
    \end{align*}
    Next, summing over $w_1$ and $w_2$, and applying Lemma \eqref{5:53}, we obtain an upper bound for 
    \\
    $\sum_{\substack{S:  y+u \text{-admissible}\\ 0\overset{S}{\longleftrightarrow}y+u}}\widetilde{\mathbb{P}}(y+v \leftrightarrow x \text{ only on } \mathsf{loop}(S))\widetilde{\mathbb{P}}(\widetilde{\mathcal{C}}^{\text{int}}(y+u) =S)$ as:  
    \begin{align*}
        \sum_{s_{3},w_{3}\in \Z^{d}}|s_{3}|^{2-d}|s_{3}-y-u|^{2-d}|s_{3}-w_{3}|^{2-d}|w_{3}-y-v|^{2-d}|w_{3}-x|^{2-d}.
    \end{align*} 
    Finally, summing over $y$, we have
    \begin{align*}
        &\sum_{y\in \Z^{d}}\sum_{\substack{S:  y+u \text{-admissible}\\ 0\overset{S}{\longleftrightarrow}y+u}}\widetilde{\mathbb{P}}(y+v \leftrightarrow x \text{ only on } \mathsf{loop}(S))\widetilde{\mathbb{P}}(\widetilde{\mathcal{C}}^{\text{int}}(y+u) =S) \\
        &\leq C\sum_{y,s_{3},w_{3}\in \Z^{d}}|s_{3}|^{2-d}|s_{3}-y-u|^{2-d}|s_{3}-w_{3}|^{2-d}|w_{3}-y-v|^{2-d}|w_{3}-x|^{2-d}.
    \end{align*}
    By (\ref{bond computation 1}), this sum can be further bounded by $\varepsilon |x|^{4-d}$.
    Substituting this into (\ref{5-64}) and summing over $y$, we get
    \begin{align*}
        \sum_{y\in \mathbb{Z}^{d}}\widetilde{\mathbb{P}}(0 \leftrightarrow y+u \circ y+v \leftrightarrow x) \geq  c\sum_{y\in \Z^{d}}|y+u|^{2-d}|y+v-x|^{2-d} - \e |x|^{4-d} \geq c|x|^{4-d}.
    \end{align*}
    \qed

    \subsection{Proof of Lemma \ref{lem 5.7}}\label{appendix.4}
  {The proof of Lemma \ref{lem 5.7} is quite similar to the proof of Lemma \ref{lem 5.1}. Therefore, we provide only a sketch, highlighting how the proof of Lemma \ref{lem 5.1} can be adapted by replacing the two-point connection with the one-arm event.}

    Similar with what was done in the proof of Lemma \ref{lem 5.1}, we divide the point $z$, for which $\mathcal{T}_{r}^{\alpha}(z)$ occurs, into two points with distance $K$ and consider the event $\{0 \overset{\mathbb{B}_{\alpha r}}{\longleftrightarrow} z \circ z+v \leftrightarrow \partial \mathbb{B}_{r}\}$, where $v= (K,0\cdots,0)$.  Due to the same resampling argument in \cite[Section 4.1]{ganguly2024ant}, it suffices to bound the new events.

    Next, instead of considering $\widetilde{\mathcal{C}}(0)$ as in the proof of Lemma \ref{lem 5.1}, we now consider $\widetilde{\mathcal{C}}_{\B_{\alpha r}}(0)$, which denotes the cluster containing $0$ and only using edges in $\B_{\alpha r}$. This allows us to keep track of the local connection constraint. Then, following the same argument in Lemma \ref{lem 5.1} while maintaining the local connection constraint, we obtain a lower bound that corresponds to \eqref{5-64}:
    \begin{align}\label{98}
        \sum_{z\in \B_{\alpha^{2} r}}\widetilde{\mathbb{P}}(0 \overset{\mathbb{B}_{\alpha r}}{\longleftrightarrow} z &\circ z+v \leftrightarrow \partial \mathbb{B}_{r}) \geq \sum_{z\in \B_{\alpha^{2} r}}\widetilde{\mathbb{P}}(0\overset{\mathbb{B}_{\alpha r}}{\longleftrightarrow}z)\widetilde{\mathbb{P}}(z+v \leftrightarrow \partial\mathbb{B}_{r}) \nonumber\\ - &\sum_{z\in \B_{\alpha^{2} r}}\sum_{\substack{S\text{: $z$ admissible}\\ 0\overset{S}{\longleftrightarrow}z }}\widetilde{\mathbb{P}}(z+v \leftrightarrow \partial\mathbb{B}_{r}\text{ only on $\mathsf{loop}(S)$})\widetilde{\mathbb{P}}(\widetilde{\mathcal{C}}_{\mathbb{B}_{\alpha r}}^{\text{int}}(z) =S).
    \end{align}
    Using Theorem \ref{lem 4.4} and the one-arm estimate from Lemma \ref{lem 2.3}, the first term in \eqref{98} can be bounded from below as 
    \begin{align*}
        \sum_{z\in \B_{\alpha^{2} r}}\widetilde{\mathbb{P}}(0\overset{\mathbb{B}_{\alpha r}}{\longleftrightarrow}z)\widetilde{\mathbb{P}}(z+v \leftrightarrow \partial\mathbb{B}_{r}) \geq cr^{-2}\sum_{r\in \B_{\alpha^{2}r}}\widetilde{\mathbb{P}}(0\overset{\mathbb{B}_{\alpha r}}{\longleftrightarrow}z) \geq  cr^{-2} \cdot \alpha^{4}r^{2} = c\alpha^{4}.
    \end{align*} 
    For the second term, we can use the same tree expansion argument, while maintaining the local connection constraint, along with a similar computation as in the proof of Proposition \ref{prop 5.11}. This gives the following bound: for any $\e>0$, there exists $K(\e)>0$ such that
    \begin{align*}
        \sum_{z\in \B_{\alpha^{2} r}}\sum_{\substack{S\text{: $z$ admissible}\\ 0\overset{S}{\longleftrightarrow}z }}\widetilde{\mathbb{P}}(z+v \leftrightarrow \partial\mathbb{B}_{r}\text{ only on $\mathsf{loop}(S)$})\widetilde{\mathbb{P}}(\widetilde{\mathcal{C}}_{\mathbb{B}_{\alpha r}}^{\text{int}}(z) =S)\leq C \alpha^{2}\e.
    \end{align*}
    Finally, we choose $\e = \alpha^{3}$ and select $\alpha$ sufficiently small such that $c\alpha^{4} - C\alpha^{5} \geq \frac{c}{2}\alpha^{4}$. Then Substituting the above bounds into \eqref{98} yields 
    \begin{align*}
        \sum_{z\in \B_{\alpha^{2} r}}\widetilde{\mathbb{P}}(0\overset{\mathbb{B}_{\alpha r}}{\longleftrightarrow}z)\widetilde{\mathbb{P}}(z+v \leftrightarrow \partial\mathbb{B}_{r}) \geq \frac{c}{2} \alpha^{4}.
    \end{align*}
    \qed

\bibliographystyle{plain}
 \bibliography{references}

\end{document}